\documentclass[11pt,reqno]{amsart}
\usepackage[margin=20mm]{geometry}
\usepackage{amssymb}
\usepackage{bbm}
\usepackage{dsfont}
\usepackage{amsfonts}
\usepackage{pifont}
\usepackage{bbding}
\usepackage{latexsym}
\usepackage{mathrsfs}
\usepackage{amsmath}
\allowdisplaybreaks[4]
\usepackage{amsthm}
\usepackage{amssymb}
\usepackage{graphicx}
\usepackage{hyperref}
\usepackage{fancyhdr}
\newtheorem{theorem}{\indent Theorem}[section]
\newtheorem{corollary}{\indent Corollary}[section]

\newtheorem{lemma}{\indent Lemma}[section]
\newtheorem{remark}{\indent Remark}[section]
\numberwithin{equation}{section}

\newcommand{\R}{\mathbb{R}}

\newcommand{\al}{\alpha}

\newcommand{\ga}{\gamma}
\newcommand{\de}{\delta}
\newcommand{\mb}{\mathbb}

\allowdisplaybreaks

\begin{document}
	
\title[convergence rates for slow-fast system]{Strong and weak convergence rates for fully coupled multiscale stochastic differential equations driven by $\alpha$-stable processes}
\author{Kun Yin}

\address{Department of Mathematics, Shanghai Jiao Tong University, 200240 Shanghai, PR China.}
\email{epsilonyk@sjtu.edu.cn}



\begin{abstract}
	We first establish strong convergence rates for multiscale systems driven by $\alpha$-stable processes, with analyses constructed in two distinct scaling regimes. When addressing weak convergence rates of this system, we derive four averaged equations with respect to four scaling regimes. Notably, under sufficient H\"{o}lder regularity conditions on  the time-dependent drifts of slow process, the strong convergence orders are related to the optimal strong convergence order $1-\frac{1}{\alpha_{2}}$, and the weak convergence orders are 1. Our primary approach involves employing fractional nonlocal Poisson equations to construct ``corrector equations" that effectively eliminate inhomogeneous terms.
\end{abstract}
\maketitle

\section{Introduction}

\subsection{Backgrounds}Multiscale models are extensively applied in fields such as chemistry, biology, material sciences,physics and other fields. These models, often characterized by different time scales, referred to as slow-fast models or models with fast oscillation, serve to bridge partial differential equations and stochastic processes. The slow-fast stochastic differential equations, driven by Brownian motion as demonstrated in references \cite{RZK,AYV}, are represented as
\begin{equation}\label{1}
\left\{
\begin{aligned}
	& dX_{t}^{\varepsilon}=b(X_{t}^{\varepsilon},Y_{t}^{\varepsilon})dt+\delta_{1}(X_{t}^{\varepsilon},Y_{t}^{\varepsilon})dB^{1}_{t},\ X^{\varepsilon}_{0}=x\in \mathbb{R}^{d_{1} },\\
	&dY_{t}^{\varepsilon}=\frac{1}{\varepsilon}f(X_{t}^{\varepsilon},Y_{t}^{\varepsilon})dt+\frac{1}{\varepsilon^{\frac{1}{2}}}\delta_{2}(X_{t}^{\varepsilon},Y_{t}^{\varepsilon})dB^{2}_{t},\ Y^{\varepsilon}_{0}=y\in \mathbb{R}^{d_{2}},
\end{aligned}
\right.
\end{equation}
here $ B^{1}_{t}$ and $ B^{2}_{t}$ represent two independent Brownian processes. With certain dissipative condition on $f(x,y)$, a concept from dynamical system theory,
i.e., $\exists\beta>0$, s.t.,
$$
( f(x,y_{1})-f(x,y_{2}),y_{1}-y_{2})  \leq-\beta|y_{1}-y_{2}|^{2},
$$
 this condition is important in proving the existence and uniqueness of the invariant measure  $ \mu^{x}(dy)$ for the frozen equation which is related to fast process $Y_{t}^{\varepsilon}$,
$$
		dY_{t}^{x,y}=f(x,Y_{t})dt+\delta_{2}(x,Y_{t})dB^{2}_{t},\ Y_{0}=y\in \mathbb{R}^{d_{2}},
$$
$x$ is fixed here, then $X_{t}^{\varepsilon}$ converges as
$\varepsilon \rightarrow0$ to averaged equation 
$$
d\bar{X}_{t}=\bar{b}(\bar{X}_{t})dt+\bar{\delta_{1}}(\bar{X}_{t})dB^{1}_{t},\ X_{0}=x\in \mathbb{R}^{d_{1}},
$$
here $\bar{b}(x)=\int_{\mathbb{R}^{d_{2}}}b(x,y)\mu^{x}(dy)$, $\bar{\delta_{1}}(x)=\int_{\mathbb{R}^{d_{2}}}\delta_{1}(x,y)\mu^{x}(dy)$. 

Given the established existence of the averaged equation, we aim to further investigate the convergence rate of the slow-fast system. C.-E. Br\'{e}hier \cite{CEB2} explored the stochastic averaging principle for a class of randomly perturbed systems of partial differential equations, asserting a strong convergence order through the Khasminskii method for the stochastic averaging principle of SDEs. Meanwhile, the weak convergence order was determined by estimating the first-order term in an asymptotic expansion of the solution to one of the Kolmogorov equations associated with the system. In \cite{CEB} Br\'{e}hier examined a semilinear stochastic partial differential equation with slow-fast time scales and demonstrated that the orders of strong and weak convergence are $\frac{1}{2}$ and $1$, respectively.   It is worth to remind that the proof relies heavily on the Poisson equation technique, which generally yields the optimal convergence order and discusses an efficient numerical scheme based on heterogeneous multiscale methods. Other studies such as \cite{DIR,DL,RZK}, have utilized Khasminskii’s time discretization technique to analyse  strong convergence rate, while asymptotic expansion of solutions to Kolmogorov equations has been applied to examine the weak convergence rate. However, compared to these two approaches, the Poisson equation offers significant advantages in determining convergence rates.

Pardoux and Veretennikov studied diffusion approximations for slow-fast stochastic differential equations by Poisson equation method in their celebrated works \cite{EY1,EY2,EY3}, 
\begin{equation}
	\begin{aligned}
\mathcal{L}u(x,y)+g(x,y)=0,
	\end{aligned}\nonumber
\end{equation}
where $x\in \mathbb{R}^{d_{1}}$ is fixed and $y\in \mathbb{R}^{d_{2}}$,
\begin{equation}
	\begin{aligned}		
		\mathcal{L}u(x,y)=\sum_{i,j=1}^{d_{2}}a_{i,j}(x,y)\frac{\partial^{2}}{\partial y_{i}\partial y_{j}}u(x,y)+\sum_{i=1}^{d_{2}}f_{i}(x,y)\partial _{y_{i}}u(x,y), 
	\end{aligned}\nonumber
\end{equation}
  the probabilistic representation of solution in bounded domain $D$ with a smooth boundary and zero boundary condition (Dirichlet boundary condition) is 
\begin{equation}
	\begin{aligned}
	u(x,y)=\int_{0}^{\tau}\mathbb{E}g(x,Y_{t}^{x,y})dt,\ \tau=inf\{t>0, Y_{t}^{x,y}\notin D\},
	\end{aligned}\nonumber
\end{equation}
while for $Y_{t}^{x,y} \in \mathbb{R}^{d_{2}}$, 
\begin{equation}\label{1-1}
	\begin{aligned}
		u(x,y)=\int_{0}^{\infty}\mathbb{E}g(x,Y_{t}^{x,y})dt,
	\end{aligned}
\end{equation}
so the Centering condition 
\begin{equation}\label{1-2}
	\begin{aligned}
		\bar{g}(x)=\int_{\mathbb{R}^{d_{2}}}g(x,y)\mu^{x}(dy)=0,
		\end{aligned}\nonumber
	\end{equation}
 is necessary, together with ergodicity of $Y_{t}^{x,y}$, is essential to guarantee the existence of the solution $u(x,y)$ given by \eqref{1-1} and its local boundedness, see \cite[Theorem 1]{EY1}.

The time-dependent case of \eqref{1} has been studied in \cite{LRSX}, where the coefficients are locally Lipschitz continuous and satisfy the dissipative condition as follows, i.e.,  $\exists \lambda>0$, 
$$
		2( f(t,x,y_{1})-f(t,x,y_{2}),y_{1}-y_{2})  +\Arrowvert\delta_{2}(t,x,y_{1})-\delta_{2}(t,x,y_{2})\Arrowvert^{2}\leq-\lambda|y_{1}-y_{2}|^{2},
$$
here $t,x$ are fixed, this condition enables the existence and uniqueness of the invariant measure $\mu^{t,x}(dy)$ corresponding to the frozen equation 
$$
	dY^{t,x}_{s}=f(t,x,Y^{t,x}_{s})ds+\delta_{2}(t,x,Y^{t,x}_{s})dB^{2}_{s},\ Y_{0}=y\in \mathbb{R}^{d_{2} }.
$$

In the context of slow-fast SPDEs, C. Sandra \cite{CF} explored the averaging principle for stochastic reaction-diffusion equations. Their work on the solvability of Kolmogorov equations in Hilbert spaces and the regularity of solutions enables the generalization of classical approaches to finite-dimensional problems of this nature for SPDEs. Z. Dong et al. \cite{DSXZ} investigated the one-dimensional stochastic Burgers equation with slow and fast time scales, driven by a Wiener process, deriving both strong and weak convergence rates. Subsequently, \cite{YXJ}  extended this research to the stochastic Burgers equation Lévy process. Further studies on stochastic dispersive equations and hyperbolic  equations can be found in \cite{CSS,WC,PG,FWL,PXW}.

Compared to continuous systems, slow-fast systems driven by processes with jumps also have seen significant advancements in recent decades. X.-B. Sun et al \cite{SXX} studied a slow-fast system driven by  independent $\alpha$-stable processes $L^{1}_{t}$ and $L^{2}_{t}$, where $\alpha\in(1,2)$,
they demonstrated that the optimal strong convergence  order of $X^{\varepsilon}_{t}$ is $1-\frac{1}{\alpha}$, and the weak convergence order is $1$. We emphasize that they employ nonlocal Poisson equations, resembles corrector equation in homogenization theory, to eliminate the difference between $b(x,y)$ and $\bar{b}(x)$ through the generator of fast component  $Y^{\varepsilon}_{t}$. 

\subsection{Sketch of this paper}
In this paper we study the following fully coupled multiscale system, which extends the results of \cite{EY1,EY2,EY3,RX} about weak convergence rates of system driven by Brown motions to strong and weak convergence rates of system driven by $\al$-stable processes. For independent $\alpha$-stable processes  $L^{1}_{t}$, $L^{2}_{t}$ we have $\alpha_{1},\alpha_{2}$ respectively, and $1<\alpha_{1},\alpha_{2}<2$,  $\gamma_{\varepsilon},\eta_{\varepsilon},\beta_{\varepsilon}\rightarrow 0$ as $\varepsilon \rightarrow 0$, especially  $\dfrac{\eta_{\varepsilon}}{\beta_{\varepsilon}}<1$. We remind that $X_{t}^{\varepsilon}$, of which the drifts are time-dependent, is the slow process with a rapidly oscillating term, however, $Y_{t}^{\varepsilon}$ is fast process with two time scales and its drifts are time-independent, 
\begin{equation}\label{1.1}
	\left\{
	\begin{aligned}
		& dX_{t}^{\varepsilon}=b(t,X_{t}^{\varepsilon},Y_{t}^{\varepsilon})dt+\frac{1}{\gamma_{\varepsilon}}H(t,X_{t}^{\varepsilon},Y_{t}^{\varepsilon})dt+dL^{1}_{t},\ X^{\varepsilon}_{0}=x\in \mathbb{R}^{d_{1} },\\
		&dY_{t}^{\varepsilon}=\frac{1}{\eta_{\varepsilon}}f(X_{t}^{\varepsilon},Y_{t}^{\varepsilon})dt+\frac{1}{\beta_{\varepsilon}}c(X_{t}^{\varepsilon},Y_{t}^{\varepsilon})dt+\frac{1}{\eta_{\varepsilon}^{\frac{1}{\alpha_{2}}}}dL^{2}_{t},\ Y^{\varepsilon}_{0}=y\in \mathbb{R}^{d_{2}}.
	\end{aligned}
	\right.
\end{equation}

We investigate strong convergence rates between  $X_{t}^{\varepsilon}$ and its averaged equation $\bar{X}_{t}$ in Theorem \ref{SCRJ} over two regimes as follows,
\begin{equation}\label{1.1-1}
	\left\{
	\begin{aligned}
		&\lim\limits_{\varepsilon \rightarrow0} \frac{\eta_{\varepsilon}^{1-\frac{1-(1\wedge v)}{\alpha_{2}}}}{\gamma_{\varepsilon}^{2}}=0,\ \lim\limits_{\varepsilon \rightarrow0}\frac{\eta_{\varepsilon}}{\gamma_{\varepsilon}\beta_{\varepsilon} }=0,\\
		&\lim\limits_{\varepsilon \rightarrow0} \frac{\eta_{\varepsilon}^{1-\frac{1-(1\wedge v)}{\alpha_{2}}}}{\gamma_{\varepsilon}^{2}}=0,\ \eta_{\varepsilon}=\gamma_{\varepsilon}\beta_{\varepsilon},
	\end{aligned}
	\right.
\end{equation}
 the exponent $v\in((\alpha_{1}-\alpha_{2})^{+},\alpha_{1}]$, $(a)^{+}=max\{a, 0\}$, governing Hölder regularities of $H(t,x,y)$ and $b(t,x,y)$ with respect to $t$ and $x$, plays vital roles in our analysis. The condition  holding uniformly‌ across the two regimes is expressed as
\begin{equation}
	\begin{aligned}
		 \lim\limits_{\varepsilon \rightarrow0}\frac{\eta_{\varepsilon}^{\left[ \left( \frac{v}{\alpha_{2}}\right) \wedge \left( 1-\frac{1\vee(\alpha_{1}-v)}{\alpha_{2}}\right) \right] }}{\gamma_{\varepsilon} } =0,
	\end{aligned}\nonumber
\end{equation}
 we notice that significant simplifications emerge when  $v\geq[(\alpha_{1}-1)\vee (\alpha_{2}-1)]$,  \begin{equation}\label{1.5-2}
 	  \frac{\eta_{\varepsilon}^{\left[ \left( \frac{v}{\alpha_{2}}\right) \wedge \left( 1-\frac{1\vee(\alpha_{1}-v)}{\alpha_{2}}\right) \right] }}{\gamma_{\varepsilon} } =\frac{\eta_{\varepsilon}^{1-\frac{1}{\alpha_{2}}}}{\gamma_{\varepsilon} },
 \end{equation}
 it is worth emphasizing that  $\dfrac{\eta_{\varepsilon}^{1-\frac{1}{\alpha_{2}}}}{\gamma_{\varepsilon} }$ corresponds to the optimal strong convergence order $1-\frac{1}{\alpha}$ 
 illustrated in \cite{SXX}. Moreover, when $v\geq1$ the regime classification  $\dfrac{\eta_{\varepsilon}^{1-\frac{1-(1\wedge v)}{\alpha_{2}}}}{\gamma_{\varepsilon}^{2}}= \dfrac{\eta_{\varepsilon}}{\gamma_{\varepsilon}^{2}}$ , the term $\dfrac{\eta_{\varepsilon}}{\gamma_{\varepsilon}^{2}}$  fundamentally distinguishes the dynamical behaviors,  aligning with the framework first established in \cite{EY2,EY3} and more precise classifications in \cite{RX}, see more details in Corollary \ref{c71}.

While Theorem \ref{WCRD} establishes weak convergence rates of $X_{t}^{\varepsilon}$ across the following four regimes,
\begin{equation}\label{1.1-2}
	\left\{
	\begin{aligned}
		&\lim\limits_{\varepsilon \rightarrow0} \frac{\eta_{\varepsilon}^{1-\frac{1-(1\wedge v)}{\alpha_{2}}}}{\gamma_{\varepsilon}^{2}}=0,\ \lim\limits_{\varepsilon \rightarrow0}\frac{\eta_{\varepsilon}}{\gamma_{\varepsilon}\beta_{\varepsilon} }=0,\\
		&\lim\limits_{\varepsilon \rightarrow0} \frac{\eta_{\varepsilon}^{1-\frac{1-(1\wedge v)}{\alpha_{2}}}}{\gamma_{\varepsilon}^{2}}=0,\ \eta_{\varepsilon}=\gamma_{\varepsilon}\beta_{\varepsilon},\\
		& \lim\limits_{\varepsilon \rightarrow0}\frac{\gamma_{\varepsilon}}{\beta_{\varepsilon}}=0,\ \eta_{\varepsilon}=\gamma_{\varepsilon}^{2},\\
		&\eta_{\varepsilon}=\gamma_{\varepsilon}^{2}=\gamma_{\varepsilon}\beta_{\varepsilon},
	\end{aligned}
	\right.
\end{equation}
the condition holding in both Regime 1 and Regime 2 is 
 $ \lim\limits_{\varepsilon \rightarrow0}\frac{ \eta_{\varepsilon}^{\left[ \frac{v}{\alpha_{2}}\wedge \left( 1-\frac{\alpha_{1}-v}{\alpha_{2}} \right)\right]  }}{\gamma_{\varepsilon}}=0$, $v\in((\alpha_{1}-\alpha_{2})^{+},\alpha_{1}]$.
In contrast, Regime 3 and Regime 4 exhibit more complicated  relationships as $v\in ( \frac{\alpha_{2}}{2}\vee \frac{2\alpha_{1}-\alpha_{2}}{2}  ,\alpha_{1}]$, which ensures the validity of  $\lim\limits_{\varepsilon \rightarrow0}\gamma_{\varepsilon}^{\frac{2v}{\alpha_{2}}-\left[ 1\vee \left( \frac{2\alpha_{1}}{\alpha_{2}}-1\right)   \right] }=0$, and when $v=\alpha_{1}=\alpha_{2}$, 
\begin{equation}
	\frac{\eta_{\varepsilon}^{1-\frac{1-(1\wedge v)}{\alpha_{2}}}}{\gamma_{\varepsilon}^{2}}= \frac{\eta_{\varepsilon}}{\gamma_{\varepsilon}^{2}},\quad  
	\frac{ \eta_{\varepsilon}^{\left[ \frac{v}{\alpha_{2}}\wedge \left( 1-\frac{\alpha_{1}-v}{\alpha_{2}} \right)\right]  }}{\gamma_{\varepsilon}}=\frac{\eta_{\varepsilon}}{\gamma_{\varepsilon}},\quad \gamma_{\varepsilon}^{\frac{2v}{\alpha_{2}}-\left[ 1\vee \left( \frac{2\alpha_{1}}{\alpha_{2}}-1\right)   \right] }=\gamma_{\varepsilon},\nonumber
\end{equation}
similar to the analysis with \eqref{1.5-2}, we observe that $\dfrac{\eta_{\varepsilon}}{\gamma_{\varepsilon}}$ and $\gamma_{\varepsilon}$ are consistent with weak convergence order 1 
proposed in \cite{SXX}, see more discussions in Corollary \ref{c72}.

Owing to some technical challenges, the convergence rates for Regime 3 and Regime 4 cannot be established in the ‌strong convergence‌ sense. These regimes, formulated in ‌weak convergence analysis‌, will instead be rigorously examined in Remark \ref{R51} of strong convergence results.

\subsection{Organization of this paper}
We start with introducing some backgrounds on the multiscale system. In Section 2, we outline some important assumptions and present our main results. 
Section 3 is devoted to studying the well-posedness of  \eqref{1.1}, with moment estimates for   $(X^{\varepsilon}_{t},Y^{\varepsilon}_{t})$  presented in Theorem \ref{T32}.  Section 4 investigates the invariant measure of the frozen equation associated with $Y^{\varepsilon}_{t}$ in \eqref{1.1}.
In Section 5, we estibalish  nonlocal Poisson equations, serving as the ``corrector equation" in homogenization theory, to bridge the gap between 
$X^{\varepsilon}_{t}$ and $\bar{X}_{t}$, with regularity estimates, LLN type  and CLT type estimates of solutions derived. 
Section 6 delves into the weak convergence of $X^{\varepsilon}_{t}$, the procedure is similar to those in Section 5. Finally, proofs of Theorem \ref{SCRJ} and Theorem \ref{WCRD} are given in Section 7.

\section{Some settings and main results}

\subsection{Notations and assumptions}
In this section we give some notions and definitions about calculitions in $d_{i}$-dimensional Euclidean space $\mathbb{R}^{d_{i} }(d_{i}\geq 1) $, we mention that  $\mathbb{R}^{d_{1} }$ and $\mathbb{R}^{d_{2} }$ have disadjoint orthogonal basis. 
$( \cdot) $ denotes inner product. Let $(\Omega,\mathcal{F},\mathbb{P})$ be the probability space  that describes random environments, denote by $ \mathbb{E}$ the expectation with respect to the probability measure $\mathbb{P}$. Define $(a)^{+}=max\{a, 0\}$.

For any $ k\in \mathbb{N}$, $\delta\in (0,1)$, we define

$ C^{k}(\mathbb{R}^{d} )$=\{$u:\mathbb{R}^{d}\longrightarrow \mathbb{R}$: $u$ and all its partial derivatives up to order $k\geq0$ are continuous.\}

$ C^{k}_{b}(\mathbb{R}^{d} )$=\{$u\in  C^{k}(\mathbb{R}^{d})$: $u$ and its all partial derivatives up to order $k\geq0$ are bounded continuous.\}

$ C^{k+\delta}_{b}(\mathbb{R}^{d} )$=\{$u\in  C^{k}_{b}(\mathbb{R}^{d})$: $u$ and its all  partial derivatives up to order $k\geq0$ are $\delta$-H\"{o}lder continuous.\}

Then $ C_b^{k+\de}\subset C_b$ for $k\geq0, \de\in(0,1)$. The spaces $ C^{k}_{b}$, $C^{k+\delta}_{b}$  equipped with $\Arrowvert \cdot\Arrowvert_{C^{k}_{b}}$ and $\Arrowvert \cdot\Arrowvert_{C^{k+\delta}_{b}}$ are Banach spaces. We emphasize that $u\in C^{k_{1}+\delta_{1},k_{2}+\delta_{2}}_{b}(\mathbb{R}^{d} )$ means that: 
(i). For $0<|\beta_{1}|<k_{1}$,  $0<|\beta_{2}|<k_{2}$, $\partial_{x}^{\beta_{1}}\partial_{y}^{\beta_{2}}u$ is bounded continuous; (ii).  $\partial_{x}^{k_{1}}$ is $\delta_{1}$-H\"{o}lder continuous with respect to $x$ uniformly in $y$,   $\partial_{y}^{k_{2}}$ is $\delta_{2}$-H\"{o}lder continuous with respect to $y$ uniformly in $x$. We denote that $f(\cdot,x,y)\in C_{b}^{v,\delta_{1},\delta_{2}}$ if $\forall (x,y)\in \mathbb{R}^{d_{1}+d_{2}}$, $f(\cdot,x,y)\in C_{b}^{v}(\mathbb{R_{+}})$, $f(t,\cdot,\cdot)\in C_{b}^{\delta_{1},\delta_{2}}(\mathbb{R}^{d_{1}+d_{2}})$. $X^{x,y}_{t}$  denotes the process $X_{t}$ starts from $(x,y)$. 

Additionally, we must introduce a function space $C^{\gamma,\eta,\delta}_p(\mb{R}^{+}\times\mathbb{R}^{d_{1}+d_{2}})$, $\forall \gamma, \eta, \delta\in(0,1)$,  $C^{\gamma,\eta,\delta}_p(\mb{R}^{+}\times\mathbb{R}^{d_{1}+d_{2}})$ consists of all functions which are local H\"{o}lder continuous and have polynomial growth in $y$ at most of order $1$ uniformly with respect to $t, x$.   We start with defining  $C_p(\mb{R}^{+}\times\mathbb{R}^{d_{1}+d_{2}})$, i.e., $\forall t\in \R^{+}$, $ x\in \R^{d_{1}}$,  $\exists C>0$, 
\begin{align*}
	\sup_{\substack{t\in \R^{+}}}\sup_{\substack{x\in \R^{d_{1}}}}|f(t,x,y)|\leq C (1+|y|),
\end{align*}
then for $C^{\gamma,0,0}_p(\mb{R}^{+}\times\mathbb{R}^{d_{1}+d_{2}})$, $\forall x\in \R^{d_{1}}$, $\exists C>0$,  
\begin{align*}
	\sup_{\substack{x\in \R^{d_{1}}}}|f(t_{1},x,y)-f(t_{2},x,y)|\leq C [|t_{1}-t_{2}|^{\ga}]\cdot(1+|y|),
\end{align*}
$ C^{0,\eta,0}_p(\mb{R}^{+}\times\mathbb{R}^{d_{1}+d_{2}})$is defined in a similar way, then for $ C^{0,0,\de}_p(\mb{R}^{+}\times\mathbb{R}^{d_{1}+d_{2}})$,$\forall t\in \R^{+}$, $x\in \R^{d_{1}}$,  $\exists C>0$,
\begin{align*}
	\sup_{\substack{t\in \R^{+}}}\sup_{\substack{x\in \R^{d_{1}}}}|f(t,x,y_{1})-f(t,x,y_{2})|\leq C [|y_{1}-y_{2}|^{\de}]\cdot(1+|y|),
\end{align*}
finally we give the definition of $C^{\gamma,\eta,\delta}_p(\mb{R}^{+}\times\mathbb{R}^{d_{1}+d_{2}})$, $\exists C>0$,
\begin{align}\label{cp}
	\|f(t_{1},x_{1},y_{1})-f(t_{2},x_{2},y_{2})\|_{\infty}\leq C [|t_{1}-t_{2}|^{\ga}+|x_{1}-x_{2}|^{\eta}+|y_{1}-y_{2}|^{\delta}]\cdot(1+|y_{1}|+|y_{2}|),
\end{align}
then we define the qusi-norm, 
\begin{align*}
	\|f\|_{C^{\gamma,\eta,\delta}_p}=\sup_{t_{1},t_{2}\geq0}\sup_{x_{1},x_{2}\in\mathbb{R}^{d_{1}}}\sup_{\substack{y_{1},y_{2}\in \mathbb{R}^{d_{2}}}}\dfrac{	|f(t_{1},x_{1},y_{1})-f(t_{2},x_{2},y_{2})|}{|t_{1}-t_{2}|^{\gamma}+|x_{1}-x_{2}|^\eta+|y_{1}-y_{2}|\delta},
\end{align*}
and $C^{k+\gamma,k+\eta,k+\delta}_p$=\{$u\in  C^{k,k,k}_{p}(\mb{R}^{+}\times\mathbb{R}^{d_{1}+d_{2}})$: $u$ and its all  partial derivatives up to order k are $C^{\gamma,\eta,\delta}_p$.\}
\begin{remark}
In \eqref{cp}, the order of $|y|$ on the right hand side is $1$, which is vital in our anslysis, since we require $p\in[1,\alpha_{2})$ in moment estimates of $|Y_{t}^{\varepsilon}|^{p}$ in Theorem $\ref{T32}$. Moreover, this is consistent with estimates in Lemma $\ref{L51}$, Theorem $\ref{T51}$, Theorem $\ref{T51-1}$, Theorem $\ref{T61}$,  and Theorem $\ref{T61-1}$.
\end{remark}

Define $K_{t}$ as an $\mathbb{R}_{+}$-valued $\mathcal{F}_{t}$ adapted process such that  $\forall p\in [1,\alpha_{1}\wedge\alpha_{2})$ we have
\begin{equation}\label{kt}
\alpha_{\infty}=\int_{0}^{\infty}|K_{s}|^{p}ds<\infty\ on \ \Omega,\  \mathbb{E}e^{p\alpha_{\infty}}<\infty.
\end{equation}

Throughout this paper we assume that $\nu_{1}$ and $\nu_{2}$ are symmetric L\'{e}vy measures, i.e., $$\int_{\mathbb{R}^{d_{i} }}(|z|^{2}\wedge 1)\nu_{i}(dz)<\infty, \ i=1,2.$$

Define fractional nonlocal operators in \eqref{1.1} as follows
$$ \mathcal{L}_{1}(t,x,y)u(x,y)=	-(-\Delta_{x})^{\frac{\alpha_{1}}{2}}u(x,y)+b(t,x,y)\nabla_{x}u(x,y),$$
$$ \mathcal{L}_{2}(x,y)u(x,y)=	-(-\Delta_{y})^{\frac{\alpha_{2}}{2}}u(x,y)+f(x,y)\nabla_{y}u(x,y),$$
$$\mathcal{L}_{3}(t,x,y)u(x,y)=H(t,x,y)\nabla_{x}u(x,y),$$ $$ \mathcal{L}_{4}(x,y)u(x,y)=c(x,y)\nabla_{y}u(x,y),$$
where 
$$ \mathcal{L}_{1}(t,x,y)u(x,y)=P.V.\int_{\mathbb{R}^{d_{1}}}(u(x+z,y)-u(x,y)-(z,\nabla_{x}u(x,y))I_{|z|\leq1})\nu_{1}(dz)+b(t,x,y)\nabla_{x}u(x,y),$$
here $\nu_{1}(dz)=\frac{c_{\alpha_{1},d_{1}}}{|z|^{d_{1}+\alpha_{1}}}dz$ is symmetric L\'{e}vy measure, $c_{\alpha_{1},d_{1}}>0$ is constant,  $ \mathcal{L}_{2}(x,y)u$ is defined similarly.

We next state some important conditions on coefficients.

\noindent\textbf{Dissipative condition}: $\forall x\in \mathbb{R}^{d_{1}}$, $y\in \mathbb{R}^{d_{2}}$,  $t\geq0$, $\exists C>0$,
\begin{equation}\label{2.2}
	\begin{split}
		& \sup_{x\in \mathbb{R}^{d_{1}}}f(x,0)<\infty, \ ( f(x,y_{1})-f(x,y_{2}), y_{1}-y_{2})  \leq -C_{1}|y_{1}-y_{2}|^{2},\\
		& \sup_{x\in \mathbb{R}^{d_{1}}}c(x,0)<\infty, \ ( c(x,y_{1})-c(x,y_{2}), y_{1}-y_{2})  \leq -C_{1}|y_{1}-y_{2}|^{2},
			\end{split}
	\end{equation}
		\begin{equation}\label{2.2-2}
			\begin{split}
		& \sup_{t\geq0}\sup_{y\in \mathbb{R}^{d_{2}}}b(t,0,y)<\infty, \ ( b(t,x_{1},y)-b(t,x_{2},y), x_{1}-x_{2})  \leq -C_{1}|x_{1}-x_{2}|^{2},\\
		& \sup_{t\geq0}\sup_{y\in \mathbb{R}^{d_{2}}}H(t,0,y)<\infty, \ ( H(t,x_{1},y)-H(t,x_{2},y), x_{1}-x_{2}) \leq -C_{1}|x_{1}-x_{2}|^{2},
	\end{split}
\end{equation}
\eqref{2.2} and \eqref{2.2-2} imply that, $\exists C>0$ s.t.,
\begin{equation}\label{2.4}
	\begin{split}
		&( f(x,y),y )=( f(x,y)-f(x,0), y ) + (f(x,0),y)  \leq C_{3}|y|-C_{1}|y|^{2},\\
		&( c(x,y),y )=( c(x,y)-c(x,0), y ) +( c(x,0),y ) \leq C_{3}|y|-C_{1}|y|^{2},\\
		& (b(t,x,y),x) = (b(t,x,y)-b(t,0,y), x ) + (b(t,0,y),x)  \leq C_{3}|x|-C_{1}|x|^{2},\\
		& (H(t,x,y),x) = (H(t,x,y)-H(t,0,y), x) + (H(t,0,y),x)  \leq C_{3}|x|-C_{1}|x|^{2}.
	\end{split}
\end{equation}

\begin{remark}
We apply dissipative condition  of $c(x,y)$ to the estimate $\mathbb{E}(\underset{t\in [0,T]}{\sup} |Y_{t}^{\varepsilon}|^{p})$, which is necessary for strong convergence analysis, see Lemma $\ref{L43}$.
\end{remark}

\noindent\textbf{Growth condition and boundedness condition}: for
$\forall x\in \mathbb{R}^{d_{1}}$, $y\in \mathbb{R}^{d_{2}}$,  $t\geq0$, $\exists C_{4},C_{5}>0$ s.t.,
\begin{equation}\label{2.77}
	\begin{split}	
		&|b(t,x,y)| \leq C_{4}(1+K_{t}),\  |H(t,x,y)| \leq C_{4}(1+K_{t}),\\
		&  |f(x,y)| \leq C_{5}(|x|+|y|),\ |c(x,y)|_{\infty} \leq C_{5}.
	\end{split}
\end{equation}
\begin{remark}
Due to the restrictions on the orders of moment estimates in Theorem $\ref{T32}$, \eqref{kt}, growth conditions of $b$ and $H$, with boundedness condition of $c$ are crucial to obtain \eqref{5.50},\eqref{5.51} and \eqref{5.52} respectively in Theorem $\ref{T52}$.
\end{remark}

\noindent\textbf{Lipschitz condition}: $\forall x\in \mathbb{R}^{d_{1}},\ y\in \mathbb{R}^{d_{2}}$, $t_{1},t_{2}\in [0,T],$ $\exists \theta_{1},\theta_{2},\in (0,1]$, $\exists C_{T},C_{6}>0$ s.t.,
\begin{equation}\label{2.111}
	\begin{split}
		&|b(t_{1},x_{1},y_{1})-b(t_{2},x_{2},y_{2})| \leq C_{T}(|t_{1}-t_{2}|^{\theta_{1}}+|x_{1}-x_{2}|^{\theta_{2}}+|y_{1}-y_{2}|),\\		&|H(t_{1},x_{1},y_{1})-H(t_{2},x_{2},y_{2})| \leq C_{T}(|t_{1}-t_{2}|^{\theta_{1}}+|x_{1}-x_{2}|^{\theta_{2}}+|y_{1}-y_{2}|),\\
	\end{split}
\end{equation}
\begin{equation}\label{2.112}
	\begin{split}
			&|f(x_{1},y_{1})-f(x_{2},y_{2})| \leq C_{6}(|x_{1}-x_{2}|^{\theta_{2}}+|y_{1}-y_{2}|),\\
			&|c(x_{1},y_{1})-c(x_{2},y_{2})| \leq C_{6}(|x_{1}-x_{2}|^{\theta_{2}}+|y_{1}-y_{2}|).
	\end{split}
\end{equation}

\noindent\textbf{Centering condition}: $\forall t\geq0, x\in \mathbb{R}^{d_{1}}, y\in \mathbb{R}^{d_{2}}$, we have for  $H(t,x,y)$,
\begin{equation}\label{2.13}
	\begin{split}
		\int_{\mathbb{R}^{ d_{2}}}H(t,x,y)\mu^{x}(dy)&=0,
	\end{split}
\end{equation}
here $\mu^{x}$ is invariant measures defined by \eqref{2.23}.

\subsection{Main results}
Next we state the main resluts of this paper.
\begin{theorem}\label{SCRJ}(Strong convergence rates)
	Assume that above conditions hold, let $p\in [1,\alpha_{1}\wedge\alpha_{2})$, $b(\cdot,\cdot,\cdot)\in C_{p}^{\frac{v}{\alpha_{1}},v, 2+\gamma}\cap C_{b}^{2}(\mathbb{R}^{d_{2}})$, $c(x,y)\in C_{b}^{v, 2+\gamma}$, $f(\cdot,\cdot)\in C_{b}^{v, 2+\gamma}$,  $\gamma\in (0,1)$, for any initial data $x\in \mathbb{R}^{d_{1}}$, $y\in \mathbb{R}^{d_{2}}$, $T>0$, $t\in[0,T]$, additionally assume that $	 \lim\limits_{\varepsilon \rightarrow0}\frac{\eta_{\varepsilon}^{\left[ \left( \frac{v}{\alpha_{2}}\right) \wedge \left( 1-\frac{1\vee(\alpha_{1}-v)}{\alpha_{2}}\right) \right] }}{\gamma_{\varepsilon} } =0$,   $v\in((\alpha_{1}-\alpha_{2})^{+},\alpha_{1}]$, we have:
	
	Regime 1:  $H(t,x,y)\in C_{p}^{\frac{v}{\alpha_{1}},v, 2+\gamma}\cap C_{b}^{2}(\mathbb{R}^{d_{2}})$, 
	\begin{equation}\nonumber
		\mathbb{E}\left( \sup_{t\in [0,T]}|X_{t}^{\varepsilon}-\bar{X}^{1}_{t}|^{p}\right) \leq   C_{T,p}\left(\left(  \frac{\eta_{\varepsilon}}{\gamma_{\varepsilon}\beta_{\varepsilon}}\right) ^{p}+\left(  \frac{\eta_{\varepsilon}^{1-\frac{1-(1\wedge v)}{\alpha_{2}}}}{\gamma_{\varepsilon}^{2}}\right) ^{p}+\left( \frac{\eta_{\varepsilon}^{\left[ \left( \frac{v}{\alpha_{2}}\right) \wedge \left( 1-\frac{1\vee(\alpha_{1}-v)}{\alpha_{2}}\right) \right] }}{\gamma_{\varepsilon} } \right) ^{p}\right) ,
	\end{equation}
	here 
	\begin{equation}\label{S1}
		d\bar{X}^{1}_{t}=\bar{b}(t,\bar{X}^{1}_{t})dt+dL_{t}^{1};
	\end{equation}

	Regime 2: let $H(\cdot,\cdot,\cdot)\in C_{p}^{\frac{v}{\alpha_{1}},v, 2+\gamma}\cap C_{b}^{2}(\mathbb{R}^{d_{2}})$,
	\begin{equation}\nonumber
		\mathbb{E}\left( \sup_{t\in [0,T]}|X_{t}^{\varepsilon}-\bar{X}^{2}_{t}|^{p}\right) \leq  C_{T,p}\left(                                                                                                                                                                                                                                                                           \left(  \frac{\eta_{\varepsilon}^{1-\frac{1-(1\wedge v)}{\alpha_{2}}}}{\gamma_{\varepsilon}^{2}}\right) ^{p}+ \left(  \frac{\eta_{\varepsilon}^{\left[ \left( \frac{v}{\alpha_{2}}\right) \wedge \left( 1-\frac{1\vee(\alpha_{1}-v)}{\alpha_{2}}\right) \right] }}{\gamma_{\varepsilon} } \right) ^{p}+\gamma_{\varepsilon}^{p}\right) ,
	\end{equation}
	we have 
	\begin{equation}\nonumber
		d\bar{X}^{2}_{t}=(\bar{b}(t,\bar{X}^{2}_{t})+\bar{c}(t,\bar{X}^{2}_{t}))dt+dL_{t}^{1};
	\end{equation}
	we mention that  $\bar{b}(t,x)=\int_{\mathbb{R}^{d_{2}}}b(t,x,y)\mu^{x}(dy)$, $\mu^{x}(dy)$ is the unique invariant measure for the transition semigroup of the corresponding frozen equation, 
	\begin{equation}\label{2.23}
		dY^{x,y}_{t}=f(x,Y_{t})dt+dL^{2}_{t},\ Y_{0}=y\in \mathbb{R}^{d_{2}},
	\end{equation}
	$\bar{c}(t,x)$ is defined as follows
	\begin{equation}\label{2.24}
		\bar{c}(t,x)=\int_{\mathbb{R}^{ d_{2}}}c(x,y)\nabla_{y}u(t,x,y)\mu^{x}(dy),\nonumber
	\end{equation}
	here $u(t,x,y)$ is the solution the  following nonlocal Poisson equation
	\begin{equation}\label{2.26}
		\mathcal{L}_{2}(x,y)u(t,x,y)+ H(t,x,y)=0.
	\end{equation}
\end{theorem}

\begin{remark}\label{R21}
The averaged equations are typically assumed to take the form as
	\begin{equation}\label{2.R1}
	d\bar{X}^{3}_{t}=(\bar{b}(t,\bar{X}^{3}_{t})+\bar{H}(t,\bar{X}^{3}_{t}))dt+dL_{t}^{1},
\end{equation}
and
\begin{equation}\label{2.R2}
	d\bar{X}^{4}_{t}=(\bar{b}(t,\bar{X}^{4}_{t})+\bar{c}(t,\bar{X}^{4}_{t})+\bar{H}(t,\bar{X}^{4}_{t}))dt+dL_{t}^{1},
\end{equation}
where \begin{equation}
	\bar{H}(t,x)=\int_{\mathbb{R}^{ d_{2}}}H(t,x,y)\nabla_{x}u(t,x,y)\mu^{x}(dy),
	\nonumber
\end{equation}
$u(t,x,y)$ is the solution of \eqref{2.26}, which necessitates the scaling conditions $ \lim\limits_{\varepsilon \rightarrow0}\dfrac{\gamma_{\varepsilon}}{\beta_{\varepsilon} }=0,\ \eta_{\varepsilon}=\gamma_{\varepsilon}^{2},$
 and $\eta_{\varepsilon}=\gamma_{\varepsilon}^{2}=\gamma_{\varepsilon}\beta_{\varepsilon}$ respectively. However, these conditions lead to contradictions with our basic  assumption  $1<\alpha_{2}<2$, a detailed discussion of this inconsistency will be provided in Remark $\ref{R51}$.
\end{remark}

\begin{remark}\label{R22}
	In contrast to strong convergence analysis in the
	$L^{p}$ norm, where martingale terms and expectation of maximal values obstruct the derivations of \eqref{2.R1} and \eqref{2.R2}, weak convergence offers following distinct advantages: 
	
	$(i).$these martingale terms associated with  $ Y^{\varepsilon}_{t}$ vanish upon taking expectation;
	
	$(ii).$instead of using $\mathbb{E}(\underset{t\in [0,T]}{\sup} |Y_{t}^{\varepsilon}|^{p})$, we may adopt the weaker estimate  $\underset{\varepsilon\in (0,1)}{\sup}\underset{t\geq0}{\sup}\mathbb{E} |Y_{t}^{\varepsilon}|^{p}$, which imposes  less stringent requirement. This substitution avoids the need to control the uniform-in-time moment bounds within the expectation, thereby broadening the applicability of the result.
	
	These advantages enable the successful derivations of the above two averaged equations in weak convergence scenarios.
\end{remark}

The following theorem is about the weak convergence rates.

\begin{theorem}\label{WCRD}(Weak convergence rates)
	Assume that above conditions hold, and   $x\in \mathbb{R}^{d_{1}}$, $y\in \mathbb{R}^{d_{2}}$,   $T>0$, $t\in[0,T]$,   $\forall \phi(x)\in C^{2+\gamma}_{b}$, 
	we have
	
	Regime 1:  $H(t,x,y)\in C_{p}^{\frac{v}{\alpha_{1}},v,2+\gamma}\cap C_{b}^{1,1+\gamma,2}$, $v\in((\alpha_{1}-\alpha_{2})^{+},\alpha_{1}]$, and  $b(\cdot,\cdot,\cdot)\in C_{p}^{\frac{v}{\alpha_{1}},v,2+\gamma}\cap C_{b}^{1,1+\gamma,2}$, $c(\cdot,\cdot)\in C_{b}^{v,2+\gamma}$, $f(\cdot,\cdot)\in C_{b}^{v,2+\gamma}$,   $\gamma\in (0,1)$,  assume additionally $ \lim\limits_{\varepsilon \rightarrow0}\frac{ \eta_{\varepsilon}^{\left[ \frac{v}{\alpha_{2}}\wedge \left( 1-\frac{\alpha_{1}-v}{\alpha_{2}} \right)\right]  }}{\gamma_{\varepsilon}}=0$, 
	\begin{equation}\nonumber
		\sup_{t\in [0,T]}|\mathbb{E}\phi(X_{t}^{\varepsilon})-\mathbb{E}\phi(\bar{X}^{1}_{t})|\leq C_{T,x,y} \left( \frac{ \eta_{\varepsilon}^{\left[ \frac{v}{\alpha_{2}}\wedge \left( 1-\frac{\alpha_{1}-v}{\alpha_{2}} \right)  \right]  }}{\gamma_{\varepsilon}}+\frac{ \eta_{\varepsilon}^{1-\frac{1-(1\wedge v)}{\alpha_{2}}}}{\gamma^{2}_{\varepsilon}}+\frac{\eta_{\varepsilon}}{\gamma_{\varepsilon}\beta_{\varepsilon}}\right) ,
	\end{equation}
	here
	\begin{equation}\nonumber
		d\bar{X}^{1}_{t}=\bar{b}(t,\bar{X}^{1}_{t})dt+dL_{t}^{1};
	\end{equation}

	Regime 2: $H(t,x,y)\in C_{p}^{\frac{v}{\alpha_{1}},v,2+\gamma}\cap C_{b}^{1,1+\gamma,2}$, $v\in((\alpha_{1}-\alpha_{2})^{+},\alpha_{1}]$, and  $b(\cdot,\cdot,\cdot)\in C_{p}^{\frac{v}{\alpha_{1}},v,2+\gamma}\cap C_{b}^{1,1+\gamma,2}$, $c(\cdot,\cdot)\in C_{b}^{v,2+\gamma}$, $f(\cdot,\cdot)\in C_{b}^{v,2+\gamma}$,   $\gamma\in (0,1)$,  we  further suppose that  $ \lim\limits_{\varepsilon \rightarrow0}\frac{ \eta_{\varepsilon}^{\left[ \frac{v}{\alpha_{2}}\wedge \left( 1-\frac{\alpha_{1}-v}{\alpha_{2}} \right)  \right]  }}{\gamma_{\varepsilon}}=0$,  
	\begin{equation}\nonumber
		\sup_{t\in [0,T]}|\mathbb{E}\phi(X_{t}^{\varepsilon})-\mathbb{E}\phi(\bar{X}^{2}_{t})|\leq  C_{T,x,y}\left(\frac{ \eta_{\varepsilon}^{\left[ \frac{v}{\alpha_{2}}\wedge \left( 1-\frac{\alpha_{1}-v}{\alpha_{2}} \right)  \right]  }}{\gamma_{\varepsilon}}+\frac{ \eta_{\varepsilon}^{1-\frac{1-(1\wedge v)}{\alpha_{2}}}}{\gamma^{2}_{\varepsilon}}+\gamma_{\varepsilon}\right)  ,
	\end{equation}
	and 
	\begin{equation}\nonumber
		d\bar{X}^{2}_{t}=(\bar{b}(t,\bar{X}^{2}_{t})+\bar{c}(t,\bar{X}^{2}_{t}))dt+dL_{t}^{1};
	\end{equation}
	
	Regime 3:  $H(t,x,y)\in C_{p}^{\frac{v}{\alpha_{1}},2+\gamma,2+\gamma}\cap C_{b}^{1,2+\gamma,2+\gamma}$, $v\in ( \frac{\alpha_{2}}{2}\vee \frac{2\alpha_{1}-\alpha_{2}}{2}  ,\alpha_{1}]$, and specially we have $b(\cdot,\cdot,\cdot)\in C_{p}^{\frac{v}{\alpha_{1}},v,2+\gamma}\cap C_{b}^{1,1+\gamma,2+\gamma}$, $c(\cdot,\cdot)\in C_{b}^{1+\gamma,2+\gamma}$, $f(\cdot,\cdot)\in C_{b}^{1+\gamma,2+\gamma}$,   $\gamma\in (\alpha_{1}-1,1)$, 
	\begin{equation}\nonumber
		\sup_{t\in [0,T]}|\mathbb{E}\phi(X_{t}^{\varepsilon})-\mathbb{E}\phi(\bar{X}^{3}_{t})| \leq  C_{T,x,y}\left( \gamma_{\varepsilon}^{\frac{2v}{\alpha_{2}}-\left[ 1\vee \left( \frac{2\alpha_{1}}{\alpha_{2}}-1\right)   \right] }+\frac{\gamma_{\varepsilon}}{\beta_{\varepsilon}}\right)  ,
	\end{equation}
	here
	\begin{equation}\nonumber
		d\bar{X}^{3}_{t}=(\bar{b}(t,\bar{X}^{3}_{t})+\bar{H}(t,\bar{X}^{3}_{t}))dt+dL_{t}^{1};
	\end{equation}

	Regime 4:  $H(t,x,y)\in C_{p}^{\frac{v}{\alpha_{1}},2+\gamma,2+\gamma}\cap C_{b}^{1,2+\gamma,2+\gamma}$,  $v\in ( \frac{\alpha_{2}}{2}\vee \frac{2\alpha_{1}-\alpha_{2}}{2}  ,\alpha_{1}]$, and  $b(\cdot,\cdot,\cdot)\in C_{p}^{\frac{v}{\alpha_{1}},v,2+\gamma}\cap C_{b}^{1,1+\gamma,2+\gamma}$, $c(\cdot,\cdot)\in C_{b}^{1+\gamma,2+\gamma}$, $f(\cdot,\cdot)\in C_{b}^{1+\gamma,2+\gamma}$,   $\gamma\in (\alpha_{1}-1,1)$, 
	\begin{equation}\nonumber
		\begin{split}
			\sup_{t\in [0,T]}|\mathbb{E}\phi(X_{t}^{\varepsilon})-\mathbb{E}\phi(\bar{X}^{4}_{t})|\leq C_{T,x,y}\cdot \gamma_{\varepsilon}^{\frac{2v}{\alpha_{2}}-\left[ 1\vee \left( \frac{2\alpha_{1}}{\alpha_{2}}-1\right)   \right] },
		\end{split}
	\end{equation}
	in this case
	\begin{equation}\nonumber
		d\bar{X}^{4}_{t}=(\bar{b}(t,\bar{X}^{4}_{t})+\bar{c}(t,\bar{X}^{4}_{t})+\bar{H}(t,\bar{X}^{4}_{t}))dt+dL_{t}^{1},
	\end{equation} 
	here we have  $\bar{b}(t,x)=\int_{\mathbb{R}^{d_{2}}}b(t,x,y)\mu^{x}(dy)$, $\mu^{x}(dy)$ is the unique invariant measure for the transition semigroup of the frozen equation $Y^{x,y}_{t}$ in \eqref{2.23}.
	$\bar{c}(t,x)$, $\bar{H}(t,x)$ are defined as follows
	\begin{equation}\label{2.36}
			\bar{c}(t,x)=\int_{\mathbb{R}^{ d_{2}}}c(x,y)\nabla_{y}\Phi(t,x,y)\mu^{x}(dy),
	\end{equation}
	\begin{equation}\label{2.37}
\bar{H}(t,x)=\int_{\mathbb{R}^{ d_{2}}}H(t,x,y)\nabla_{x}\Phi(t,x,y)\mu^{x}(dy),
	\end{equation}
	here $\Phi(t,x,y)$ is the solution the  following nonlocal Poisson equation
	\begin{equation}\label{2.26-2}
		\mathcal{L}_{2}(x,y)\Phi_{}(t,x,y)+H(t,x,y)=0.
	\end{equation}
\end{theorem}

\begin{remark}
	Exponential ergodicity, regularity estimates rely on H\"{o}lder regularity and boundedness of the coefficients $b(t,x,y)$ and $H(t,x,y)$ in $C^{\gamma,\eta,\delta}_b$ spaces, see Lemma $\ref{L42}$, Theorem $\ref{T51}$, Theorem $\ref{T51-1}$, Theorem $\ref{T61}$ and Theorem $\ref{T61-1}$ for details, while molification estimates need we employ $C^{\gamma,\eta,\delta}_p$ spcaes in Lemma $\ref{L51}$, so we must take intersections of $C^{\gamma,\eta,\delta}_p$ and $C^{\gamma,\eta,\delta}_b$.
\end{remark}

\begin{remark}
We may consider the weak convergence for diffusive scaling when  $\alpha_{2}=2$ and $\eta_{\varepsilon}=\varepsilon$, inspiring us to  employ the ``corrector equation" from homogenization theory to eliminate the difference between $X_{t}^{\varepsilon}$ and averaged equation driven by Brownian process,  however, this cannot be solved in our method due to some essential issues, we will explain it in Remark $\ref{R61}$.
\end{remark}

\section{Well-posedness and some moment estimates of $(X_{t}^{\varepsilon},Y_{t}^{\varepsilon})$ }

Recall that $ L_{t}^{i},\ i=1,2,$ denote the isotropic $\alpha$-stable processes associated with $X_{t}^{\varepsilon}$ and $Y_{t}^{\varepsilon}$ respectively, the corresponding Poission random measures are defined by \cite{DA},
$$ N^{i}(t,A)=\sum_{s\leq t}1_{A}(L_{s}^{i}-L_{s-}^{i}),\ \forall A\in \mathcal{B}(\mathbb{R}^{d_{i} }),$$
then compensated Poisson measures will be 
$$ \tilde{N}^{i}(t,A)=N^{i}(t,A)-t\nu_{i}(A),$$
where $\nu_{i}(dz)=\frac{c_{\alpha_{i},d_{i}}}{|z|^{d_{i}+\alpha_{i}}}dz$ is symmetric L\'{e}vy measure, $c_{\alpha_{i},d_{i}}>0$ is constant. By L\'{e}vy-It\^{o} decomposition and symmetry of $\nu_{i}(dz)$, we have
\begin{equation}\label{3.1}
	 L^{i}_{t}=\int_{|z|\leq1}z \tilde{N}^{i}(t,dz)+\int_{|z|>1}zN^{i}(t,dz),
\end{equation}	 
so \eqref{1.1} with initial data  $X^{\varepsilon}_{0}=x\in \mathbb{R}^{d_{1} }$, $ Y^{\varepsilon}_{0}=y\in \mathbb{R}^{d_{2} }$ can be rewritten in Poisson processes form as 
\begin{equation}\label{3.2}
	\left\{
\begin{aligned}
	&dX_{t}^{\varepsilon}=b(t,X_{t}^{\varepsilon},Y_{t}^{\varepsilon})dt+\frac{1}{\gamma_{\varepsilon}}H(t,X_{t}^{\varepsilon},Y_{t}^{\varepsilon})dt+\int_{|z|\leq1}z \tilde{N}^{1}(dt,dz)+\int_{|z|>1}zN^{1}(dt,dz),\\
	&dY_{t}^{\varepsilon}=\frac{1}{\eta_{\varepsilon}}f(X_{t}^{\varepsilon},Y_{t}^{\varepsilon})dt+\frac{1}{\beta_{\varepsilon}}c(X_{t}^{\varepsilon},Y_{t}^{\varepsilon})dt+\frac{1}{\eta_{\varepsilon}^{\frac{1}{\alpha_{2}}}}\left( \int_{|z|\leq1}z \tilde{N}^{2}(dt,dz)+\int_{|z|>1}zN^{2}(dt,dz)\right).
\end{aligned}
\right.
\end{equation}

\begin{theorem}\label{T31}
(Well-posedness of \eqref{1.1}) Assume that above conditions hold, $\forall \varepsilon >0$, given any initial data $x\in \mathbb{R}^{d_{1} }$, $y \in \mathbb{R}^{d_{2} }$, there exists unique solution $(X_{t}^{\varepsilon},Y_{t}^{\varepsilon})$ to \eqref{1.1}.
\end{theorem}
Under Lipschitz conditions, growth conditions of $b$ , $f$, $H$ and boundedness condition of $c$,	well-posedness of \eqref{3.2} can be  established following the same procedures outlined in  \cite[Theorem 6.2.3, Theorem 6.2.9, Theorem 6.2.11]{DA}, which leads to well-posedness of \eqref{1.1}.
\begin{theorem}\label{T32}
	For any solution $(X_{t}^{\varepsilon},Y_{t}^{\varepsilon})$ to \eqref{1.1}, $\forall p\in [1,\alpha_{1}\wedge\alpha_{2})$,  $t\geq0$, $\exists C_{p}>0$ s.t.,
	\begin{equation}\label{3.4}
		\sup_{\varepsilon\in (0,1)}\sup_{t\geq0} \mathbb{E}|X_{t}^{\varepsilon}|^{p}\leq C_{p}(1+|x|^{p}),
	\end{equation}
	\begin{equation}\label{3.6}
		\sup_{\varepsilon\in (0,1)}\sup_{t\geq0}\mathbb{E} |Y_{t}^{\varepsilon}|^{p}\leq  C_{p}(1+|y|^{p}).
	\end{equation}
\end{theorem}
\begin{proof}
	Our methods are based on \cite{LRSX} and \cite{SXX}.	We observe that for $X_{t}^{\varepsilon} $,
	\begin{equation*}
			X_{t}^{\varepsilon}=x+\int_{0}^{t}b(s,X_{s}^{\varepsilon},Y_{s}^{\varepsilon})ds+\int_{0}^{t}\left( \int_{|z|\leq 1}z \tilde{N}^{1}(ds,dz)+\int_{|z|> 1}zN^{1}(ds,dz)\right)+\int_{0}^{t}\frac{1}{\gamma_{\varepsilon}}H(s,X_{s}^{\varepsilon},Y_{s}^{\varepsilon})ds,
	\end{equation*}
due to the fact that $p<\alpha_{1}\wedge\alpha_{2}<2$, here we can not use It\^{o} formula directly, however,  we observe that $ |x|^{2\cdot \frac{p}{2}}<(|x|+1)^{2\cdot \frac{p}{2}}<(|x|^{2}+1)^{\frac{p}{2}}$,  $ |y|^{2\cdot \frac{p}{2}}<(|y|+1)^{2\cdot \frac{p}{2}}<(|y|^{2}+1)^{\frac{p}{2}}$, so we define 
	\begin{equation}
		\begin{split}
			U(t,x)&=e^{-\frac{p}{2}\alpha_{t}}(|x|^{2}+1)^{\frac{p}{2}},\	U(y)=(|y|^{2}+1)^{\frac{p}{2}},
		\end{split}\nonumber
	\end{equation}
	we can see that  $U(t,x)>0$, $U(y)>0$, and 
	\begin{equation}\label{du}
		\begin{split}
			&|DU(t,x)|=\left| e^{-\frac{p}{2}\alpha_{t}} \frac{px}{(|x|^{2}+1)^{1-\frac{p}{2}}}\right|\leq C_{p} e^{-\frac{p}{2}\alpha_{t}}|x|^{p-1},\\
			& |DU(y)|=\left|  \frac{py}{(|y|^{2}+1)^{1-\frac{p}{2}}}\right|\leq C_{p} |y|^{p-1},
				\end{split}
		\end{equation}
	\begin{equation}\label{d2u}
	\begin{split}
			&|D^{2}U(t,x)|=\left| e^{-\frac{p}{2}\alpha_{t}}\left(  \frac{pI_{d_{2}\times d_{2} }}{(|x|^{2}+1)^{1-\frac{p}{2}}}-\frac{p(p-2)x\otimes x}{(|x|^{2}+1)^{2-\frac{p}{2}}}\right) \right|\leq \frac{C_{p}e^{-\frac{p}{2}\alpha_{t}}}{(|x|^{2}+1)^{1-\frac{p}{2}}}\leq C_{p}e^{-\frac{p}{2}\alpha_{t}},\\
			&|D^{2}U(y)|=\left|   \frac{pI_{d_{2}\times d_{2} }}{(|y|^{2}+1)^{1-\frac{p}{2}}}-\frac{p(p-2)y\otimes y}{(|y|^{2}+1)^{2-\frac{p}{2}}} \right|\leq \frac{C_{p}}{(|y|^{2}+1)^{1-\frac{p}{2}}}\leq C_{p}.
		\end{split}
	\end{equation}
	
	Applying It\^{o} formula, and taking expectation on both sides, with the fact that $\mathbb{E}\tilde{N}^{1}(ds,dz)=0$,
	\begin{equation}\label{3.9}
		\begin{split}
			&	\frac{d\mathbb{E}U(t,X_{t}^{\varepsilon})}{dt}=-\frac{p}{2}\mathbb{E}K_{t}U(t,X_{t}^{\varepsilon})+\mathbb{E}( b(t,X_{t}^{\varepsilon},Y_{t}^{\varepsilon}),DU(t,X_{t}^{\varepsilon})) \\
			&+\mathbb{E} \int_{|z|\leq 1}\left( U(t,X_{t}^{\varepsilon}+z)-U(t,X_{t}^{\varepsilon})-( DU(t,X_{t}^{\varepsilon}), z)  \right) \nu_{1}(dz)\\
			&+\mathbb{E} \int_{|z|> 1}\left( U(t,X_{t}^{\varepsilon}+z)-U(t,Y_{t}^{\varepsilon}) \right) \nu_{1}(dz)+\mathbb{E} \frac{1}{\gamma_{\varepsilon}}( H(t,X_{t}^{\varepsilon},Y_{t}^{\varepsilon}),DU(t,Y_{t}^{\varepsilon})) \\
			&\leq \mathbb{E} (b(t,X_{t}^{\varepsilon},Y_{t}^{\varepsilon}),DU(t,X_{t}^{\varepsilon})) +\mathbb{E} \int_{|z|> 1}\left( U(t,X_{t}^{\varepsilon}+z)-U(t,X_{t}^{\varepsilon}) \right) \nu_{1}(dz)\\
			&+\mathbb{E} \int_{|z|\leq 1}\left( U(t,X_{t}^{\varepsilon}+z)-U(t,X_{t}^{\varepsilon})- (DU(t,X_{t}^{\varepsilon}), z)  \right) \nu_{1}(dz)\\
			&+\mathbb{E} \frac{1}{\gamma_{\varepsilon}} (H(t,X_{t}^{\varepsilon},Y_{t}^{\varepsilon}),DU(t,Y_{t}^{\varepsilon}))=I_{1}+I_{2}+I_{3}+I_{4}.
		\end{split}
	\end{equation}
	
	For $I_{1}$, by dissipative condition  \eqref{2.2-2},  \eqref{2.4},
	\begin{equation}\label{3.10}
		\begin{split}
			&I_{1}=\mathbb{E} ( b(t,X_{t}^{\varepsilon},Y_{t}^{\varepsilon}),DU(t,X_{t}^{\varepsilon}))\\
			& \leq \mathbb{E} e^{-\frac{p}{2}\alpha_{t}} \frac{( b(t,X_{t}^{\varepsilon},Y_{t}^{\varepsilon})-b(t,0,Y_{t}^{\varepsilon}),pX_{t}^{\varepsilon}) +( b(t,0,Y_{t}^{\varepsilon}),pX_{t}^{\varepsilon})}{(|X_{t}^{\varepsilon}|^{2}+1)^{1-\frac{p}{2}}}\\
			&\leq C_{p}\mathbb{E}e^{-\frac{p}{2}\alpha_{t}} \frac{C_{3}|X_{t}^{\varepsilon}|-C_{1}|X_{t}^{\varepsilon}|^{2}}{(|X_{t}^{\varepsilon}|^{2}+1)^{1-\frac{p}{2}}}\leq C_{p,C_{B}} \mathbb{E}\left( 1-(|X_{t}^{\varepsilon}|^{2}+1)^{\frac{p}{2}}\right)=C_{p}-C_{p}\mathbb{E}U(t,X_{t}^{\varepsilon}), 
		\end{split}
	\end{equation}
	thus for $I_{2}$, by \eqref{du}, and young inequality,
	\begin{equation}\label{3.11}
		\begin{split}
		&I_{2}=\mathbb{E} \int_{|z|> 1}\left( U(t,X_{t}^{\varepsilon}+z)-U(t,X_{t}^{\varepsilon}) \right) \nu_{1}(dz)\\
	& \leq C_{p}\mathbb{E}e^{-\frac{p}{2}\alpha_{t}}\int_{|z|> 1}\left(|X_{t}^{\varepsilon}|^{p-1}+|z|^{p-1} \right) \nu_{1}(dz) \leq C_{p} +C_{p}\mathbb{E}U(t,X_{t}^{\varepsilon}),
		\end{split}
	\end{equation}
	we derive the last inequality from $1\leq p<\alpha$ and H\"{o}lder inequality. Similarly, 
	\begin{equation}\label{3.12}
		\begin{split}
			&I_{3}=\mathbb{E} \int_{|z|\leq 1}\left( U(t,X_{t}^{\varepsilon}+z)-U(t,X_{t}^{\varepsilon})-( DU(t,X_{t}^{\varepsilon}),z )\right)  \nu_{1}(dz)\leq C_{p},
		\end{split}
	\end{equation}
	and for $I_{4}$,  by  \eqref{2.2-2},  \eqref{2.4},
	\begin{equation}\label{3.13}
		\begin{split}
			&I_{4}=\mathbb{E} \frac{1}{\gamma_{\varepsilon}}( H(t,X_{t}^{\varepsilon},Y_{t}^{\varepsilon}),DU(t,X_{t}^{\varepsilon}))\\
			& \leq \frac{1}{\gamma_{\varepsilon}} \mathbb{E} e^{-\frac{p}{2}\alpha_{t}} \frac{( H(t,X_{t}^{\varepsilon},Y_{t}^{\varepsilon})-H(t,0,Y_{t}^{\varepsilon}),pX_{t}^{\varepsilon}) +( H(t,0,Y_{t}^{\varepsilon}),pX_{t}^{\varepsilon})}{(|Y_{t}^{\varepsilon}|^{2}+1)^{1-\frac{p}{2}}}\\
			&\leq  \frac{C_{p}}{\gamma_{\varepsilon}}\mathbb{E}e^{-\frac{p}{2}\alpha_{t}} \frac{C_{3}|X_{t}^{\varepsilon}|-|X_{t}^{\varepsilon}|^{2}}{(|X_{t}^{\varepsilon}|^{2}+1)^{1-\frac{p}{2}}}\leq \frac{C_{p}}{\gamma_{\varepsilon}}\mathbb{E}\left( 1-(|X_{t}^{\varepsilon}|^{2}+1)^{\frac{p}{2}}\right)=\frac{C_{p}}{\gamma_{\varepsilon}}-\frac{C_{p}\mathbb{E}U(t,X_{t}^{\varepsilon})}{\gamma_{\varepsilon}},   
		\end{split}
	\end{equation}
	combining \eqref{3.9}-\eqref{3.13}, we obtain
	\begin{equation}
		\begin{split}
			&\frac{d\mathbb{E}U(t,X_{t}^{\varepsilon})}{dt}\leq\frac{C_{p}}{\gamma_{\varepsilon}}+C_{p}-C_{p}\mathbb{E}U(t,Y_{t}^{\varepsilon})-\frac{C_{p}\mathbb{E}U(t,Y_{t}^{\varepsilon})}{\gamma_{\varepsilon}},
		\end{split}\nonumber
	\end{equation}
	by Gronwall inequality we have 
	\begin{equation}
		\mathbb{E}U(t,X_{t}^{\varepsilon})\leq e^{-C_{p}(\frac{1}{\gamma_{\varepsilon}}+1)t}(|x|^{2}+1)^{\frac{p}{2}}+C_{p}(\frac{1}{\gamma_{\varepsilon}}+1)\int^{t}_{0}e^{-C_{p}(\frac{1}{\gamma_{\varepsilon}}+1)(t-s)}ds,\nonumber
	\end{equation}
	which means
	$$\mathbb{E}(|X_{t}^{\varepsilon}|^{2}+1)^{\frac{p}{2}}\leq \mathbb{E}e^{-C_{p}(\frac{1}{\gamma_{\varepsilon}}+1)t}(|x|^{2}+1)^{\frac{p}{2}}+\mathbb{E}(1-e^{-C_{p}(\frac{1}{\gamma_{\varepsilon}}+1)t}),$$
	so we yield, 
	\begin{equation}\label{3.14}
		\begin{split}
			\sup\limits_{\varepsilon\in (0,1)}\sup\limits_{t\geq0} \mathbb{E}\left(|X_{t}^{\varepsilon}|^{p}\right)\leq C_{p}(1+|x|^{p}),
		\end{split}
	\end{equation}
	we get \eqref{3.4}. Next we need to estimate $\sup\limits_{\varepsilon\in (0,1)}\sup\limits_{t\geq0}\mathbb{E}\left(|Y_{t}^{\varepsilon}|^{p}\right) $.

	From \eqref{3.2} we deduce that 
	\begin{equation}
		\begin{split}
			Y_{t}^{\varepsilon}=&y+\int_{0}^{t}\frac{1}{\eta_{\varepsilon}}f(X_{s}^{\varepsilon},Y_{s}^{\varepsilon})ds+\int_{0}^{t}\frac{1}{\eta_{\varepsilon}^{\frac{1}{\alpha_{2}}}}\left( \int_{|z|\leq \eta_{\varepsilon}^{\frac{1}{\alpha_{2}}}}z \tilde{N}^{2}(ds,dz)+\int_{|z|> \eta_{\varepsilon}^{\frac{1}{\alpha_{2}}}}zN^{2}(ds,dz)\right)\\
		&+\int_{0}^{t}\frac{1}{\beta_{\varepsilon}}c(X_{s}^{\varepsilon},Y_{s}^{\varepsilon})ds, 
		\end{split}\nonumber
	\end{equation}
	applying It\^{o} formula and taking expectation on both sides, with $\mathbb{E}\tilde{N}^{2}(ds,dz)=0$ we derive,
	\begin{align}
		&\frac{d\mathbb{E}U(Y_{t}^{\varepsilon})}{dt}=\mathbb{E} \frac{1}{\eta_{\varepsilon}}( f(X_{t}^{\varepsilon},Y_{t}^{\varepsilon}),DU(Y_{t}^{\varepsilon}))\nonumber \\
		&+\mathbb{E} \int_{|z|\leq \eta_{\varepsilon}^{\frac{1}{\alpha}}}\left( U(Y_{t}^{\varepsilon}+\eta_{\varepsilon}^{-\frac{1}{\alpha_{2}}}z)-U(Y_{t}^{\varepsilon})-( DU(Y_{t}^{\varepsilon}), \eta_{\varepsilon}^{-\frac{1}{\alpha_{2}}}z) \right) \nu_{2}(dz)\nonumber\\
		&+\mathbb{E} \int_{|z|> \eta_{\varepsilon}^{\frac{1}{\alpha}}}\left( U(Y_{t}^{\varepsilon}+\eta_{\varepsilon}^{-\frac{1}{\alpha_{2}}}z)-U(Y_{t}^{\varepsilon}) \right) \nu_{2}(dz)+\mathbb{E} \frac{1}{\beta_{\varepsilon}}( c(X_{t}^{\varepsilon},Y_{t}^{\varepsilon}),DU(t,Y_{t}^{\varepsilon}))\nonumber \\
		&\leq \mathbb{E} \frac{1}{\eta_{\varepsilon}}( f(X_{t}^{\varepsilon},Y_{t}^{\varepsilon}),DU(Y_{t}^{\varepsilon})) +\mathbb{E} \int_{|z|> \eta_{\varepsilon}^{\frac{1}{\alpha}}}\left( U(Y_{t}^{\varepsilon}+\eta_{\varepsilon}^{-\frac{1}{\alpha_{2}}}z)-U(Y_{t}^{\varepsilon}) \right) \nu_{2}(dz)\nonumber\\
		&+\mathbb{E} \int_{|z|\leq \eta_{\varepsilon}^{\frac{1}{\alpha}}}\left( U(Y_{t}^{\varepsilon}+\eta_{\varepsilon}^{-\frac{1}{\alpha_{2}}}z)-U(Y_{t}^{\varepsilon})-(DU(Y_{t}^{\varepsilon}), \eta_{\varepsilon}^{-\frac{1}{\alpha_{2}}}z ) \right) \nu_{2}(dz)\nonumber\\
		&+\mathbb{E} \frac{1}{\beta_{\varepsilon}}\left(  c(X_{t}^{\varepsilon},Y_{t}^{\varepsilon}),DU(Y_{t}^{\varepsilon})\right) =I_{1}+I_{2}+I_{3}+I_{4},\label{3.15}
	\end{align}
	we then estimate four terms respectively.
	
	For $I_{1}$, by   dissipative condition \eqref{2.2}, \eqref{2.4}, 
	\begin{equation}\label{3.16}
		\begin{split}
			&I_{1}=\mathbb{E} \frac{1}{\eta_{\varepsilon}}( f(X_{t}^{\varepsilon},Y_{t}^{\varepsilon}),DU(Y_{t}^{\varepsilon})) \leq\frac{1}{\eta_{\varepsilon}} \mathbb{E}  \frac{( f(X_{t}^{\varepsilon},Y_{t}^{\varepsilon})-f(X_{t}^{\varepsilon},0),pY_{t}^{\varepsilon}) +( f(X_{t}^{\varepsilon},0),pY_{t}^{\varepsilon})}{(|Y_{t}^{\varepsilon}|^{2}+1)^{1-\frac{p}{2}}}\\
			&\leq  \frac{C_{p}}{\eta_{\varepsilon}}\mathbb{E} \frac{C_{1}|Y_{t}^{\varepsilon}|-C_{1}|Y_{t}^{\varepsilon}|^{2}}{(|Y_{t}^{\varepsilon}|^{2}+1)^{1-\frac{p}{2}}}\leq \frac{C_{p,C_{1}}}{\eta_{\varepsilon}}\mathbb{E}\left( 1-(|Y_{t}^{\varepsilon}|^{2}+1)^{\frac{p}{2}}\right)=\frac{C_{p,C_{1}}}{\eta_{\varepsilon}}-\frac{C_{p}\mathbb{E}U(t,Y_{t}^{\varepsilon})}{\eta_{\varepsilon}},  
		\end{split}
	\end{equation}
	in addition, taking $y=\eta_{\varepsilon}^{-\frac{1}{\alpha_{2}}}z$,  we obtain 
	\begin{equation}\label{3.17-1}
	\begin{split}\nu_{2}(dz)=\frac{c}{|z|^{d_2+\alpha_{2}}}dz=\frac{c}{|\eta_{\varepsilon}^{\frac{1}{\alpha_{2}}}y|^{d_2+\alpha_{2}}}(\eta_{\varepsilon}^{\frac{1}{\alpha_{2}}})^{d_2}dy=\frac{1}{\eta_{\varepsilon}}\frac{c}{|y|^{d_2+\alpha_{2}}}dy=\frac{1}{\eta_{\varepsilon}}\nu_{2}(dy),		
	\end{split}
	\end{equation}
	thus for $I_{2}$, similar to \eqref{3.11},
	\begin{equation}\label{3.17}
		\begin{split}
			&I_{2}=\frac{1}{\eta_{\varepsilon}}\mathbb{E} \int_{|y|> 1}\left( U(Y_{t}^{\varepsilon}+y)-U(Y_{t}^{\varepsilon}) \right) \nu_{2}(dy) \\
			&\leq  \frac{C_{p}}{\eta_{\varepsilon}}\mathbb{E}\int_{|y|> 1}\left(|Y_{t}^{\varepsilon}|^{p-1}+|y|^{p-1} \right) \nu_{2}(dy) \leq\frac{C_{p}}{\eta_{\varepsilon}}+\frac{C_{p}\mathbb{E}U(Y_{t}^{\varepsilon})}{\eta_{\varepsilon}},  
		\end{split}
	\end{equation}
	then 
	\begin{equation}\label{3.18}
		\begin{split}
			&I_{3}=\frac{1}{\eta_{\varepsilon}}\mathbb{E} \int_{|y|\leq 1}\left( U(Y_{t}^{\varepsilon}+y)-U(Y_{t}^{\varepsilon})-( DU(Y_{t}^{\varepsilon}),y) \right) \nu_{2}(dy)\leq  \frac{C_{p}}{\eta_{\varepsilon}},  
		\end{split}
	\end{equation}
	and for $I_{4}$,
	\begin{equation}\label{3.19}
		\begin{split}
			&I_{4}=\mathbb{E} \frac{1}{\beta_{\varepsilon}}( c(X_{t}^{\varepsilon},Y_{t}^{\varepsilon}),DU(Y_{t}^{\varepsilon})) \leq \frac{1}{\beta_{\varepsilon}} \mathbb{E} \frac{( c(X_{t}^{\varepsilon},Y_{t}^{\varepsilon})-c(X_{t}^{\varepsilon},0),pY_{t}^{\varepsilon}) +( c(X_{t}^{\varepsilon},0),pY_{t}^{\varepsilon})}{(|Y_{t}^{\varepsilon}|^{2}+1)^{1-\frac{p}{2}}}\\
			&\leq  \frac{C_{p}}{\beta_{\varepsilon}}\mathbb{E} \frac{C_{1}|Y_{t}^{\varepsilon}|-|Y_{t}^{\varepsilon}|^{2}}{(|Y_{t}^{\varepsilon}|^{2}+1)^{1-\frac{p}{2}}}\leq \frac{C_{p}}{\beta_{\varepsilon}}\mathbb{E}\left( 1-(|Y_{t}^{\varepsilon}|^{2}+1)^{\frac{p}{2}}\right)=\frac{C_{p}}{\beta_{\varepsilon}}-\frac{C_{p}\mathbb{E}U(Y_{t}^{\varepsilon})}{\beta_{\varepsilon}},  
		\end{split}
	\end{equation}
	combining \eqref{3.15}-\eqref{3.19}, take $C_{1}$ in \eqref{2.2} large enough, we derive
	\begin{equation}\label{3.20}
		\begin{split}
			&\frac{d\mathbb{E}U(Y_{t}^{\varepsilon})}{dt}\leq\frac{C_{p}}{\beta_{\varepsilon}}+\frac{C_{p}}{\eta_{\varepsilon}}-\frac{C_{p}\mathbb{E}U(Y_{t}^{\varepsilon})}{\eta_{\varepsilon}},\nonumber
		\end{split}
	\end{equation}
	so that by Gronwall inequality we have 
	\begin{equation}\label{3.21}
		\begin{split}
			\mathbb{E}U(Y_{t}^{\varepsilon})\leq e^{-C_{p}(\frac{1}{\eta_{\varepsilon}}+\frac{1}{\beta_{\varepsilon}})t}(|y|^{2}+1)^{\frac{p}{2}}+C_{p}\left( \frac{1}{\eta_{\varepsilon}}+\frac{1}{\beta_{\varepsilon}}\right) \int^{t}_{0}e^{-C_{p}(\frac{1}{\eta_{\varepsilon}}+\frac{1}{\beta_{\varepsilon}})(t-s)}ds,
		\end{split}\nonumber
	\end{equation}
	which means
	$$\mathbb{E}(|Y_{t}^{\varepsilon}|^{2}+1)^{\frac{p}{2}}\leq \mathbb{E}e^{-C_{p}(\frac{1}{\eta_{\varepsilon}}+\frac{1}{\beta_{\varepsilon}})t}(|y|^{2}+1)^{\frac{p}{2}}+\mathbb{E}(1-e^{-C_{p}(\frac{1}{\eta_{\varepsilon}}+\frac{1}{\beta_{\varepsilon}})t}),$$
	so that,
	\begin{equation}\label{3.22}
		\begin{split}
			\sup\limits_{\varepsilon\in (0,1)}\sup\limits_{t\geq0} \mathbb{E}\left(|Y_{t}^{\varepsilon}|^{p}\right)\leq C_{p}(1+|y|^{p}),
		\end{split}
	\end{equation}
proof is complete.
\end{proof}

\section{The frozen equation for \ref{1.1}}
We state the frozen equation corresponding to the process $Y_{t}^{\varepsilon}$ in \eqref{1.1} for any fixed $x\in \mathbb{R}^{d_{1} }$,
\begin{equation}\label{4.1}
	dY_{t}=f(x,Y_{t})dt+dL^{2}_{t},\ Y_{0}=y\in \mathbb{R}^{d_{2} }.
\end{equation}

\subsection{Invariant measure of \eqref{4.1}}
If dissipative condition, growth condition, Lipschitz condition hold, for any fixed  $x\in\mathbb{R}^{d_{1}}$, and initial data $y\in\mathbb{R}^{d_{2}}$, \eqref{4.1} has unique solution $ \{Y_{t}^{x,y}\}_{t\geq 0}$, let $ \{P_{t}^{x}\}_{t\geq 0} $ be the transition semigroups of $ \{Y_{t}^{x,y}\}_{t\geq 0}$. We next state the existence and uniquness of invariant measure possesed by  $ \{Y_{t}^{x,y}\}_{t\geq 0}$.

\begin{lemma}\label{L41}
	Suppose that $ f(x,\cdot)\in C^{1}_{b}$, Lipschitz condition and dissipative condition hold, for any fixed  $ x\in \mathbb{R}^{d_{1}}$, $\forall  t\geq0, y_{1}, y_{2}\in \mathbb{R}^{d_{2}}$, we have $\exists \beta>0$ s.t.
	$$ | Y_{t}^{ x, y_{1}}-Y_{t}^{x, y_{2}}|\leq e^{-\frac{\beta t}{2}}| y_{1}- y_{2}|.$$
\end{lemma}
\begin{proof}
	The arguement directly follows from \cite[Lemma 3.1]{SXX}, we omit the details here.
\end{proof}

Considering the estimate provided in \eqref{4.4} in Theorem \ref{T41}, which is derived from  \eqref{3.6}, we naturally observe that the family $ \{P_{t}^{x}\}_{t\geq 0} $ depends continuously on the initial data $y$. The tightness with respect to  $y\in\mathbb{R}^{d_{2}} $ can be inferred from  Lemma \ref{L41}.  Subsequently, employing the Bogoliubov-Krylov theorem allows us to establish the existence of the invariant measure $ \mu^{x}$. Define
$$ \bar{f}(x)=\int_{\mathbb{R}^{d_{2}}}f(x,y)\mu^{x}(dy).$$

In addition, for $1\leq p<\alpha_{2}$,
\begin{equation}\label{4.3}
	\begin{split}
		\sup_{x\in \mathbb{R}^{d_{1}}}\int_{\mathbb{R}^{d_{2}}}|y|^{p}\mu^{x}(dy)&=\int_{\mathbb{R}^{d_{2}}}\mathbb{E}|Y_{t}^{x,y}|^{p} \mu^{x}(dy)\leq  \int_{\mathbb{R}^{d_{2}}}C_{p}(1+|y|^{p}) \mu^{x}(dy)\\
		&=\int_{\mathbb{R}^{d_{2}}}|y|^{p} \mu^{x}(dy)+C_{p} \leq C_{p}(1+|y|^{p}).
	\end{split}
\end{equation}

For any bounded measurable function $g:\mathbb{R}^{d_{2}}\rightarrow \mathbb{R}$, denote $ g(y)$, we have
$$ P_{t}^{x}g(y)=\mathbb{E}g(Y_{t}^{ x,y}),\ t\geq 0,\ y\in \mathbb{R}^{d_{2}}.$$
\begin{lemma}\label{L42}
	Suppose that $ f(x,\cdot)\in C^{1}_{b}$, dissipative condition  is valid. $\forall  t\geq0$,  we have for any $g(y)\in C^{1}_{b}$, fixed $x\in \mathbb{R}^{d_{1}},  y_{1}, y_{2}\in \mathbb{R}^{d_{2}}$, $\exists \beta>0$ s.t., 
	$$ \sup_{x\in \mathbb{R}^{d_{1}}}| P_{t}^{x}g(x,y)-\bar{g}(x)|\leq C\cdot Lip(g)e^{-\frac{\beta t}{2}}(1+|y|),$$
	here $ Lip(g)=\sup_{x\neq y}\frac{|g(x)-g(y)|}{|x-y|}.$
\end{lemma}
\begin{proof}
	See details in \cite[Proposition 3.8]{SXX}.
\end{proof}

From Lemma \ref{L42} we derive the exponential ergodicity of invarinat measure.

\subsection{Moment estimates of $ Y_{t}^{ x,y}$}

\begin{theorem}\label{T41}
	Suppose that \eqref{2.2}, \eqref{2.4} hold, we have for $1\leq p<\alpha_{2}$, for $T\geq 1$,
	\begin{equation}\label{4.4}
	\sup_{t\geq0} \mathbb{E}|Y_{t}^{ x,y}|^{p}\leq C_{p}(1+|y|^{p}),
	\end{equation}
	\begin{equation}\label{4.5}
		\mathbb{E}\left(\sup_{t\in [0,T]} |Y_{t}^{ x,y}|^{p}\right)\leq C_{p}(T^{\frac{p}{\alpha_{2}}}+|y|^{p}).
	\end{equation}
\end{theorem}
\begin{proof}
	\eqref{4.4} follows from \eqref{3.6} directly, so we just need to prove \eqref{4.5}.
	We define 
	\begin{equation}\label{4.6}
		U_{T}(y)=(|y|^{2}+T^{\frac{2}{\alpha_{2}}})^{\frac{p}{2}},
	\end{equation}
	so that similar to \eqref{du} and \eqref{d2u},
	\begin{equation}\label{4.7}
		\begin{split}
				|DU_{T}(y)|&=\left|  \frac{py}{(|y|^{2}+T^{\frac{2}{\alpha_{2}}})^{1-\frac{p}{2}}}\right|\leq C_{p} |y|^{p-1},\
		|D^{2}U_{T}(y)|=\left|   \frac{pI_{d_{2}\times d_{2} }}{(|y|^{2}+T^{\frac{2}{\alpha_{2}}})^{1-\frac{p}{2}}}-\frac{p(p-2)y\otimes y}{(|y|^{2}+T^{\frac{2}{\alpha_{2}}})^{2-\frac{p}{2}}} \right|\leq C_{p}T^{\frac{p-2}{\alpha_{2}}},
		\end{split}
	\end{equation}
by It\^{o} formula,
	\begin{equation}\label{4.8}
		\begin{split}
			&U_{T}(Y_{t}^{ x,y})=U_{T}(y)+ \int_{0}^{t} (f(x,Y_{r}^{ x,y}),DU_{T}(Y_{r}^{x,y})) dr \\
			&+ \int_{0}^{t}\int_{|z|\leq T^{\frac{1}{\alpha_{2}}}}\left( U_{T}(Y_{r}^{ x,y}+z)-U_{T}(Y_{r}^{ x,y})-(DU_{T}(Y_{r}^{ x,y}), z)  \right) \tilde{N}^{2}(dr,dz)\\
			&+\int_{0}^{t} \int_{|z|> T^{\frac{1}{\alpha_{2}}}}\left( U_{T}(Y_{r}^{x,y}+z)-U_{T}(Y_{r}^{ x,y}) \right) \nu_{2}(dz)dr\\
			&\leq  \int_{0}^{t} (f(x,Y_{r}^{x,y}),DU_{T}(Y_{r}^{ x,y})) dr+\mathbb{E} \int_{0}^{t}\int_{|z|> T^{\frac{1}{\alpha_{2}}}}\left( U_{T}(Y_{r}^{x,y}+z)-U_{T}(Y_{r}^{ x,y}) \right) \nu_{2}(dz)dr\\
			&+\int_{0}^{t} \int_{|z|\leq T^{\frac{1}{\alpha_{2}}}}\left( U_{T}(Y_{r}^{x,y}+z)-U_{T}(Y_{r}^{ x,y})- (DU_{T}(Y_{r}^{x,y}), z)\right) \nu_{2}(dz)dr\\
			&+\int_{0}^{t} \int_{|z|> T^{\frac{1}{\alpha_{2}}}}\left( U_{T}(Y_{r}^{x,y}+z)-U_{T}(Y_{r}^{ x,y}) \right) N^{2}(dr,dz)+U_{T}(y)=\hat{I}_{1}+\hat{I}_{2}+\hat{I}_{3}+\hat{I}_{4}+U_{T}(y),
		\end{split}
	\end{equation}
	so by dissipative condition of $f(x,y)$ in  \eqref{2.4} and $T\geq 1$, we have for $\hat{I}_{1}$,
	\begin{equation}\label{4.9}
		\mathbb{E}\left(\sup_{t\in [0,T]} |\hat{I}_{1}(t)|\right)\leq \int_{0}^{T} \frac{C_{p}}{(|Y_{r}^{ x,y}|^{2}+T^{\frac{2}{\alpha_{2}}})^{1-\frac{p}{2}}}dr\leq C_{p}T^{\frac{p}{\alpha_{2}}-\frac{2}{\alpha_{2}}+1}\leq C_{p}T^{\frac{p}{\alpha_{2}}},
	\end{equation}
	meanwhile for $\hat{I}_{3}$,
	\begin{equation}\label{4.10}
		\begin{split}
			&\mathbb{E}\left(\sup_{t\in [0,T]} |\hat{I}_{3}(t)|\right)\leq C_{p}T^{\frac{p}{\alpha_{2}}-\frac{2}{\alpha_{2}}}\int_{0}^{T}\int_{|z|\leq T^{\frac{1}{\alpha_{2}}}}|z|^{2}\nu_{2}(dz)dr\leq C_{p}T^{\frac{p}{\alpha_{2}}},\\
		\end{split}
	\end{equation}
	and for $\hat{I}_{2}$, Burkholder-Davies-Gundy's inequality and \eqref{4.4},
	\begin{equation}\label{4.11}
		\begin{split}
			&\mathbb{E}\left(\sup_{t\in [0,T]} |\hat{I}_{2}(t)|\right)\leq \mathbb{E}\left[ \int_{0}^{T} \int_{|z|\leq T^{\frac{1}{\alpha_{2}}}}\left|  U_{T}(Y_{r}^{ x,y}+z)-U_{T}(Y_{r}^{ x,y}) \right|^{2}  N_{2}(dz)dr\right] ^{\frac{1}{2}}\\
			&\leq \mathbb{E}\left[ \int_{0}^{T} \int_{|z|\leq T^{\frac{1}{\alpha_{2}}}}\left( \left| Y_{r}^{ x,y}  \right|^{2p-2}|z|^{2}+|z|^{2p} \right)  \nu_{2}(dz)dr\right] ^{\frac{1}{2}}\\
			&\leq \frac{1}{4}\mathbb{E} \left(\sup_{r\in [0,T]} |Y_{r}^{ x,y}|^{p}\right)+C_{p}\left(\int_{0}^{T} \int_{|z|\leq T^{\frac{1}{\alpha_{2}}}} |z|^{2}  \nu_{2}(dz)dr\right) ^{p}+\int_{0}^{T} \int_{|z|\leq T^{\frac{1}{\alpha_{2}}}} |z|^{p}  \nu_{2}(dz)dr\\
			&\leq \frac{1}{4}\mathbb{E} \left(\sup_{r\in [0,T]} |Y_{r}^{ x,y}|^{p}\right)+C_{p}T^{\frac{p}{\alpha_{2}}},
		\end{split}
	\end{equation}
	for $\hat{I}_{4}$,
	\begin{equation}\label{4.11-1}
	\begin{split}
		&\mathbb{E}\left(\sup_{t\in [0,T]} |\hat{I}_{4}(t)|\right)\leq \mathbb{E}\left[ \int_{0}^{T} \int_{|z|> T^{\frac{1}{\alpha_{2}}}}\left|  U_{T}(Y_{r}^{ x,y}+z)-U_{T}(Y_{r}^{ x,y}) \right|  N_{2}(dz)dr\right] \\
		&\leq \mathbb{E}\left[ \int_{0}^{T} \int_{|z|>T^{\frac{1}{\alpha_{2}}}}\left( \left| Y_{r}^{ x,y}  \right|^{p-1}|z|+|z|^{p} \right)  \nu_{2}(dz)dr\right]\\
		&\leq \frac{1}{4}\mathbb{E} \left(\sup_{r\in [0,T]} |Y_{r}^{ x,y}|^{p}\right)+C_{p}\left(\int_{0}^{T} \int_{|z|>T^{\frac{1}{\alpha_{2}}}} |z|^{2}  \nu_{2}(dz)dr\right) ^{p}+\int_{0}^{T} \int_{|z|> T^{\frac{1}{\alpha_{2}}}} |z|^{p}  \nu_{2}(dz)dr\\
		&\leq \frac{1}{4}\mathbb{E} \left(\sup_{r\in [0,T]} |Y_{r}^{ x,y}|^{p}\right)+C_{p}T^{\frac{p}{\alpha_{2}}},
	\end{split}
	\end{equation}
	where we used Young inequality in third inequality. From \eqref{4.9}-\eqref{4.11-1}, we derive \eqref{4.5}.
\end{proof}

Next we study $\mathbb{E}\left(\underset{t\in [0,T]}{\sup} |Y_{t}^{\varepsilon}|^{p}\right)$, which is essential to strong convergence estimates.

\begin{lemma}\label{L43}
	$\forall t\in [0,T]$,  $T\geq1$, 
		\begin{equation}
		\begin{split}\label{4L3}
			& \mathbb{E}\left(\sup_{t\in [0,T]} |Y_{t}^{\varepsilon}|^{p}\right)\leq C_{T,p}\left( \eta_{\varepsilon}^{-\frac{p}{\alpha_{2}}}+|y|^{p}\right) .
		\end{split}
	\end{equation}
\end{lemma}
\begin{proof}
	Denote $\tilde{L}^{2}_{t}=\frac{1}{\eta_{\varepsilon}^{\frac{1}{\alpha_{2}}}}L^{2}_{t\eta_{\varepsilon}}$, so that 
		\begin{equation}
		\begin{split}
			\tilde{Y}_{t}^{\varepsilon}&=y+\frac{1}{\eta_{\varepsilon}}\int ^{t\eta_{\varepsilon}}_{0}f(X_{s}^{\varepsilon},\tilde{Y}_{s}^{\varepsilon})ds+\frac{1}{\beta_{\varepsilon}}\int ^{t\eta_{\varepsilon}}_{0}c(X_{s}^{\varepsilon},\tilde{Y}_{s}^{\varepsilon})ds+\frac{1}{\eta_{\varepsilon}^{\frac{1}{\alpha_{2}}}}L^{2}_{t\eta_{\varepsilon}}\\
			&=y+\int ^{t}_{0}f(X_{s\eta_{\varepsilon}}^{\varepsilon},\tilde{Y}_{s\eta_{\varepsilon}}^{\varepsilon})ds+\frac{\eta_{\varepsilon}}{\beta_{\varepsilon}}\int ^{t}_{0}c(X_{s\eta_{\varepsilon}}^{\varepsilon},\tilde{Y}_{s\eta_{\varepsilon}}^{\varepsilon})ds+\tilde{L}^{2}_{t},\\
		\end{split}\nonumber
	\end{equation}
we can see that $\tilde{Y}_{t}^{\varepsilon}$ and $Y_{t}^{\varepsilon}$ have the same law, then similar to the proof of \eqref{4.5}, with the fact that $\frac{\eta_{\varepsilon}}{\beta_{\varepsilon}}<1$, and dissipative condition of $c(x,y)$,
	\begin{equation}
	\begin{split}
	\mathbb{E}\left(\sup_{t\in [0,T]} \frac{\eta_{\varepsilon}}{\beta_{\varepsilon}}\left| \int_{0}^{t}  (c(x,\tilde{Y}_{s}^{\varepsilon}),DU_{T}(\tilde{Y}_{s}^{\varepsilon})) ds \right| \right) \leq \int_{0}^{T} \frac{C_{p}}{(|Y_{s}^{ x,y}|^{2}+T^{\frac{2}{\alpha_{2}}})^{1-\frac{p}{2}}}ds\leq C_{p}T^{\frac{p}{\alpha_{2}}-\frac{2}{\alpha_{2}}+1}\leq C_{p}T^{\frac{p}{\alpha_{2}}},	\end{split}\nonumber
\end{equation}
then we have
	\begin{equation}
	\begin{split}
		& \mathbb{E}\left(\sup_{t\in [0,T]} |\tilde{Y}_{t}^{\varepsilon}|^{p}\right)\leq C_{p}\left( T^{\frac{p}{\alpha_{2}}}+|y|^{p}\right),
	\end{split}\nonumber
\end{equation}
from \eqref{2.2} and \eqref{4.5}, for any $T\geq 1$,
	\begin{equation}\label{4.12}
		\begin{split}
			&\mathbb{E}\left(\sup_{t\in [0,T]} |Y_{t}^{\varepsilon}|^{p}\right)= \mathbb{E}\left(\sup_{t\in [0,\frac{T}{\eta_{\varepsilon}}]} |\tilde{Y}_{t}^{\varepsilon}|^{p}\right)\leq C_{p}\left( \left( \frac{T}{\eta_{\varepsilon}}\right) ^{\frac{p}{\alpha_{2}}}+|y|^{p}\right) \leq C_{T,p}\left( \eta_{\varepsilon}^{-\frac{p}{\alpha_{2}}}+|y|^{p}\right).\\
		\end{split}
	\end{equation}
\end{proof}

\section{Strong convergence estimates for \ref{1.1}}

 Insipred by \cite{RX} and \cite{SXX}, we next consider the following associated nonlocal Poisson equation, which can be regarded as a corrector equation to eliminate the effects of drift term $b(t,x,y)$, $\frac{1}{\gamma_{\varepsilon}}H(t,X_{t}^{\varepsilon},Y_{t}^{\varepsilon})$, and effects of $Y^{\varepsilon}_{t}$
in $X^{\varepsilon}_{t}$ by the generator of $Y_{t}$, so we next construct the following nonlocal equation.

\subsection{Regularity estimates of nonlocal Poisson equation}

Let $g(\cdot,\cdot,\cdot)\in C_{b}^{\frac{v}{\alpha_{1}},1+\gamma, 2+\gamma}$  satisfies Lipschitz condition, growth condition, dissipative condition,
\begin{equation}\label{5.1}
	\mathcal{L}_{2}(x,y)u(t,x,y)+g(t,x,y)-\bar{g}(t,x)=0,
\end{equation}
 here $ \bar{g}(t,x)=\int_{\mathbb{R}^{d_{2}}}g(t,x,y)\mu^{x}(dy),$ some regularity estimates of $u(t,x,y)$ are necessary.

\begin{theorem}\label{T51}
	For any  $x\in \mathbb{R}^{d_{1}}, $  $y\in \mathbb{R}^{d_{2}} $, and $t\in [0,T]$, $g(t,\cdot,\cdot)\in C_{b}^{1+\gamma, 2+\gamma}$ we define 
	\begin{equation}\label{5.2}
		\begin{split}
			u(t,x,y)=\int^{\infty}_{0}\left( \mathbb{E}g(t,x,Y_{s}^{x,y})-\bar{g}(t,x)\right) ds,
		\end{split}
	\end{equation}
	then $u(t,x,y)$ is a solution of \eqref{5.1} and  $u(t,\cdot,y)\in C_{b}^{1}(\mathbb{R}^{d_{1}})$, $u(t,x,\cdot)\in C_{b}^{2}(\mathbb{R}^{d_{2}})$, $\exists C>0$ s.t.,
	\begin{equation}\label{5.3}
			\sup_{t\in [0,T]}\sup_{x\in \mathbb{R}^{d_{1}}}|u(t,x,y)|\leq C_{T}(1+|y|),
	\end{equation}
	\begin{equation}\label{5.4}
		\sup_{t\in [0,T]}\sup_{\substack{x\in \mathbb{R}^{d_{1}}\\ y\in \mathbb{R}^{d_{2}}}}|\nabla_{y}u(t,x,y)|\leq C,
	\end{equation}
	\begin{equation}\label{5.5}
		\sup_{t\in [0,T]}\sup_{x\in \mathbb{R}^{d_{1}}}|\nabla_{x}u(t,x,y)|\leq C(1+|y|),
	\end{equation}
	\begin{equation}\label{5.6}
		\sup_{t\in [0,T]}|\nabla_{x}u(t,x_{1},y)-\nabla_{x}u(t,x_{2},y)|\leq C|x_{1}-x_{2}|^{\gamma}(1+|x_{1}-x_{2}|^{1-\gamma})(1+|y|),
	\end{equation}
	here $\gamma\in (\alpha_{1}-1,1)$.
\end{theorem}
\begin{proof}
	It is easy to see that $u(t,x,y)$ in \eqref{5.2} is a solution of \eqref{5.1}, which can be deduced by It\^{o} formula, and properties of  $u(t,\cdot,y)\in C_{b}^{1+\gamma}(\mathbb{R}^{d_{1}})$, $u(t,x,\cdot)\in C_{b}^{2+\gamma}(\mathbb{R}^{d_{2}})$ inherit from  regularities of $g(t,x,y)$, other properties follow from \cite[Proposition 3.3]{SXX}.
\end{proof}

We also need to introduce  mollification of functions which will be used to tackle the difficulities related to time derivative and different regimes. Let $\rho_{1}:\mathbb{R}\rightarrow [0,1]$, $\rho_{2}: \mathbb{R}^{ d_{1}}\rightarrow [0,1]$ be two nonnegative smooth mollifiers s.t. 

(1). $\rho_{1}\in C_{0}^{\infty}(\mathbb{R})$, supp $  \rho_{1}\subset \overline{B_{1}(0)}=\left\lbrace t\in \mathbb{R}:|t|\leq 1 \right\rbrace $, and $ \rho_{2}\in C_{0}^{\infty}(\mathbb{R}^{ d_{1}})$, supp $  \rho_{2}\subset \overline{B_{1}(0)}=\left\lbrace x\in \mathbb{R}^{ d_{1}}:|x|\leq 1 \right\rbrace $;

(2). $\int_{\mathbb{R}}\rho_{1}(t)dt=\int_{\mathbb{R}^{ d_{1}}}\rho_{2}(x)dx=1$;

(3). $\forall k\geq 0$, $\exists C_{k}>0$ s.t. $|\nabla^{k}\rho_{1}(t)|\leq C_{k}\rho_{1}(t)$, $|\nabla^{k}\rho_{2}(x)|\leq C_{k}\rho_{2}(x)$.

Then for any $n\in  \mathbb{N}^{+}$, let $\rho_{1}^{n}(t)=n^{\alpha_{1}}\rho_{1}(n^{\alpha_{1}}t)$, $\rho_{2}^{n}(x)=n^{d_{1}}\rho_{2}(nx)$, then for $g(t,x,y)$, mollification of $g(t,x,y)$ in $t$ and $x$ is defined by
\begin{equation}\label{5.56}
	g_{n}(t,x,y)=g\ast\rho_{2}^{n}\ast\rho_{1}^{n}=\int_{\mathbb{R}^{ d_{1}+1}}g(t-s, x-z, y)\rho_{2}^{n}(z)\rho_{1}^{n}(s)dzds,
\end{equation} 
in addition we define the fractional Laplacian operator  $-(-\Delta_{x})^{\frac{\alpha}{2}}f(x)$, $x,z\in \mathbb{R}^{d_{1}}$,  $0<\alpha<2$, as follows 
\begin{equation}\label{5.56-1}
	-(-\Delta_{x})^{\frac{\alpha}{2}}f(x)=P.V.\int_{\mathbb{R}^{d_{1}}}(u(x+z)-u(x)-(z,\nabla_{x}u(x))I_{|z|\leq1})\nu(dz),
\end{equation} 
where $\nu(dz)=\frac{c}{|z|^{d_{1}
		+\alpha}}dz$ is symmetric L\'{e}vy measure.
We mention that by mollification method we have $g_{n}(\cdot,x,y)\in  C_{0}^{\infty}(\mathbb{R})$, $g_{n}(t,\cdot,y)\in  C_{0}^{\infty}(\mathbb{R}^{ d_{1}})$, so we can get higher regularity estimates of $g_{n}(\cdot,\cdot,y)$ with respect to $t$ and $x$, thus we have the following lemma.
\begin{lemma}\label{L51}
	Let $g(t,x,y)\in C_{p}^{\frac{v}{\alpha_{1}}, v,\delta }$ with $0<v\leq\alpha_{1}$, $0<\delta<1$, and 
	 define $g_{n}$ by \eqref{5.56},  then we have 
	\begin{equation}\label{5.57}
		\|g(\cdot,\cdot,y)-g_{n}(\cdot,\cdot,y)\|_{\infty}\leq C\cdot n^{-v}(1+|y|),
	\end{equation} 
	\begin{equation}\label{5.58-1}
		\|\partial_{t}g_{n}(\cdot,\cdot,y)\|_{\infty}\leq C\cdot n^{\alpha_{1}-v}(1+|y|), 
			\end{equation} 
		\begin{equation}\label{5.58-4}
				\|(-\Delta_{x})^{\frac{\alpha_{1}}{2}}g_{n}(\cdot,\cdot,y)\|_{\infty}\leq C\cdot n^{\alpha_{1}-v}(1+|y|),
			\end{equation} 
				\begin{equation}\label{5.58-3}
		\|\nabla_{x}g_ {n}(\cdot,\cdot,y)\|_{\infty}\leq C\cdot n^{1-(1\wedge v)}(1+|y|),
	\end{equation} 
	we can further estimate that $ 	 \|\nabla_{x}^{2}g_{n}(\cdot,\cdot,y)\|_{\infty}\leq C\cdot n^{2-v}(1+|y|^{m}).$
\end{lemma}
\begin{proof}
The proof mainly refers to \cite[Lemma 4.1]{RX}. By definition of  H\"{o}lder derivative and a change of variable, for $0<v\leq1$, definition of $ C_{p}^{\frac{v}{\alpha_{1}}, v,\delta }$ and \eqref{cp},
\begin{equation}
		\begin{split}
	&|g(t,x,y)-g_{n}(t,x,y)|\leq \int_{\mathbb{R}^{ d_{1}+1}}	|g(t,x,y)-g(t-s,x-z,y)|\rho_{1}^{n}(s)\rho_{2}^{n}(z)dzds\\
	&\leq C\cdot\int_{\mathbb{R}^{ d_{1}+1}}	(|s|^{\frac{v}{\alpha_{1}}}+|z|^{v})(1+|y|)\rho_{1}^{n}(s)\rho_{2}^{n}(z)dzds\leq C\cdot n^{-v}(1+|y|),
	\end{split}\nonumber
\end{equation}
similar to \eqref{3.17-1}, taking $y=nz$, from the definintion of $\nu(dz)$ in \eqref{5.56-1} we observe that 
\begin{equation}
	\begin{split} \nu(dz)=\frac{c}{|z|^{d_{1}+\alpha}}dz=\frac{c}{|n^{-1}y|^{d_{1}+\alpha}}(n^{-1})^{d_{1}}dy=n^{\alpha}\frac{c}{|y|^{d_{1}+\alpha}}dy=n^{\alpha}\nu(dy),	
	\end{split}\nonumber
\end{equation}
 therefore,
\begin{equation}\label{5.56-2}
	\begin{split}
	\left| (-\Delta_{x})^{\frac{\alpha}{2}}\rho_{2}^{n}(x)\right|& =c\left| \int_{\mathbb{R}^{d_{1}}}\left( n^{d_{1}}\rho_{2}(nx+nz)-n^{d_{1}}\rho_{2}(nx)-(nz,\nabla_{x}n^{d_{1}}\rho_{2}(nx))I_{|nz|\leq1}\right) \nu(dz)\right| 	\\
&=c\cdot n^{\alpha}\cdot n^{d_{1}}\left| \int_{\mathbb{R}^{d_{1}}}\left( \rho_{2}(nx+y)-\rho_{2}(nx)-(y,\nabla_{x}\rho_{2}(nx))I_{|y|\leq1}\right) \nu(dy)\right| \\
& \leq C_{\alpha}\cdot n^{\alpha}\cdot n^{d_{1}}\rho_{2}(nx)\leq C_{\alpha}\cdot n^{\alpha}\rho_{2}^{n}(x),
	\end{split}
\end{equation}
we used definition in \eqref{5.56-1} and the fact that $\forall k\geq 0$, $\exists C_{k}>0$ s.t. $|\nabla^{k}\rho_{2}(x)|\leq C_{k}\rho_{2}(x)$ in first inequality. Consequently, by \eqref{5.56-2}
\begin{equation}
	\begin{split}
		|	(-\Delta_{x})^{\frac{\alpha_{1}}{2}}g_{n}(\cdot,\cdot,y)|&\leq \int_{\mathbb{R}^{ d_{1}+1}}	|g(t-s,x-z,y)-g(t-s,x,y)|\rho_{1}^{n}(s)|(-\Delta_{z})^{\frac{\alpha_{1}}{2}}\rho_{2}^{n}(z)|dzds\\
		&\leq C\cdot n^{\alpha_{1}}\int_{\mathbb{R}^{ d_{1}+1}}	|z|^{v}(1+|y|)\rho_{1}^{n}(s)\rho_{2}^{n}(z)dzds\leq C\cdot n^{\alpha_{1}-v}(1+|y|),
	\end{split}\nonumber
\end{equation}
furthermore,
\begin{equation}
	\begin{split}
		|\nabla_{x}^{2}g_{n}(\cdot,\cdot,y)|&\leq \int_{\mathbb{R}^{ d_{1}+1}}	|g(t-s,x-z,y)-g(t-s,x,y)|\rho_{1}^{n}(s)|\nabla^{2}_{z}\rho_{2}^{n}(z)|dzds\\
		&\leq C\cdot n^{2}\int_{\mathbb{R}^{ d_{1}+1}}	|z|^{v}(1+|y|)\rho_{1}^{n}(s)\rho_{2}^{n}(z)dzds\leq C\cdot n^{2-v}(1+|y|),
	\end{split}\nonumber
\end{equation}
\begin{equation}
	\begin{split}
		|\nabla_{x}g_{n}(\cdot,\cdot,y)|&\leq \int_{\mathbb{R}^{ d_{1}+1}}	|g(t-s,x-z,y)-g(t-s,x,y)|\rho_{1}^{n}(s)|\nabla_{z}\rho_{2}^{n}(z)|dzds\\
		&\leq C\cdot n\int_{\mathbb{R}^{ d_{1}+1}}	|z|^{v}(1+|y|)\rho_{1}^{n}(s)\rho_{2}^{n}(z)dzds\leq C\cdot n^{1-v}(1+|y|),
	\end{split}\nonumber
\end{equation}
\begin{equation}\label{5.L5}
	\begin{split}
		|\partial_{t}g_{n}(t,x,y)|&\leq \int_{\mathbb{R}^{ d_{1}}+1}	|g(t-s,x-z,y)-g(t,x-z,y)|\partial_{s}\rho_{1}^{n}(s)|\rho_{2}^{n}(z)dzds\\
		&\leq C\cdot n^{\alpha_{1}}\cdot\int_{\mathbb{R}^{ d_{1}}+1}|s|^{\frac{v}{\alpha_{1}}}\rho_{1}^{n}(s)\rho_{2}^{n}(z)(1+|y|)dzds\leq C\cdot n^{\alpha_{1}-v}(1+|y|),
	\end{split}
\end{equation}
for $1<v\leq\alpha_{1}$,
\begin{equation}
	\begin{split}
	|g(t,x,y)-g_{n}(t,x,y)|	&\leq \int_{\mathbb{R}^{ d_{1}}+1}	|g(t-s,x+z,y)+g(t-s,x-z,y)-2g(t,x,y)|\rho_{1}^{n}(s)\rho_{2}^{n}(z)dzds\\
		&\leq C\cdot\int_{\mathbb{R}^{ d_{1}}+1}	(|s|^{\frac{v}{\alpha_{1}}}+|z|^{v})(1+|y|)\rho_{1}^{n}(s)\rho_{2}^{n}(z)dzds\leq C\cdot n^{-v}(1+|y|),\nonumber
	\end{split}
\end{equation} 
applying \eqref{5.56-2} again, we have
\begin{equation}
\begin{split}
	|(-\Delta_{x})^{\frac{\alpha_{1}}{2}}g_{n}(\cdot,\cdot,y)|&\leq \int_{\mathbb{R}^{ d_{1}+1}}	|\nabla_{x}g(t-s,x-z,y)-\nabla_{x}g(t-s,x,y)|\rho_{1}^{n}(s)|(-\Delta_{z})^{\frac{\alpha_{1}-1}{2}}\rho_{2}^{n}(z)|dzds\\
	&\leq C\cdot n^{\alpha_{1}-1}\int_{\mathbb{R}^{ d_{1}+1}}	|z|^{v-1}(1+|y|)\rho_{1}^{n}(s)\rho_{2}^{n}(z)dzds\leq C\cdot n^{\alpha_{1}-v}(1+|y|),
\end{split}\nonumber
\end{equation}
\begin{equation}
	\begin{split}
		|\nabla_{x}^{2}g_{n}(\cdot,\cdot,y)|&\leq \int_{\mathbb{R}^{ d_{1}+1}}	|\nabla_{x}g(t-s,x-z,y)-\nabla_{x}g(t-s,x,y)|\rho_{1}^{n}(s)|\nabla_{z}\rho_{2}^{n}(z)|dzds\\
		&\leq C\cdot n\int_{\mathbb{R}^{ d_{1}+1}}	|z|^{v-1}(1+|y|)\rho_{1}^{n}(s)\rho_{2}^{n}(z)dzds\leq C\cdot n^{2-v}(1+|y|),
	\end{split}\nonumber
\end{equation}
\begin{equation}\label{5.58-32}
	\begin{split}
		|\nabla_{x}g_{n}(\cdot,\cdot,y)|&\leq \int_{\mathbb{R}^{ d_{1}+1}}	|\nabla_{x}g(t-s,x-z,y)|\rho_{1}^{n}(s)|\rho_{2}^{n}(z)|dzds\\
		&\leq C\cdot \int_{\mathbb{R}^{ d_{1}+1}}	(1+|y|)\rho_{1}^{n}(s)\rho_{2}^{n}(z)dzds\leq C\cdot (1+|y|),
	\end{split}
\end{equation}
the proof of estimate related to $\partial_{t}g_{n}(t,x,y)$ can be proved as \eqref{5.L5}.
\end{proof}
\begin{remark}
 Above results claim that due to the definition of $C_{p}^{\frac{v}{\alpha_{1}}, v,\sigma }$ and \eqref{cp}, these estimates we need are uniformly bounded  both in $t$ and $x$, and  bounded from above by $|y|$ of order $1$, this conculsion plays important role in strong and weak convergence estimates, which is also consistent with moment estimates in Theorem $\ref{T32}$ when $p=1$, and regularity estimates in Theorem $\ref{T51}$, $p$ in Theorem $\ref{T51-1}$, Theorem $\ref{T61}$, Theorem $\ref{T61-1}$ where orders of $|y|$ are $1$.
\end{remark}

\begin{remark}
Although the relationship $	\|(-\Delta_{x})^{\frac{\alpha_{1}}{2}}g_{n}(\cdot,\cdot,y)\|_{\infty}\leq \|\nabla_{x}^{2}g_{n}(\cdot,\cdot,y)\|_{\infty}$ provides computational convenience, we will employ the more precise estimates \eqref{5.58-4} in subsequent analysis to achieve sharper results.
\end{remark}

Actually, not need all regularity estimates in Theorem \ref{T51} are necessary in LLN type estimate and Regime 1, 2 in CLT type estimate, hence we assume lower regularity to relax the conditions.

Let $g(t,x,y)$  satisfies Lipschitz condition, growth condition, dissipative condition,
\begin{equation}\label{5.1-1}
	\mathcal{L}_{2}(x,y)u(t,x,y)+g(t,x,y)-\bar{g}(t,x)=0,
\end{equation}
here $ \bar{g}(t,x)=\int_{\mathbb{R}^{d_{2}}}g(t,x,y)\mu^{x}(dy),$ then we have the following regularity estimates.
\begin{theorem}\label{T51-1}
	  $\forall x\in \mathbb{R}^{d_{1}}, $  $y\in \mathbb{R}^{d_{2}} $,  $t\in [0,T]$, $g(t,x,\cdot)\in C_{b}^{2}(\mathbb{R}^{d_{2}})$, 
	   we define 
	\begin{equation}\label{5.2-1}
		\begin{split}
			u(t,x,y)=\int^{\infty}_{0}\left( \mathbb{E}g(t,x,Y_{s}^{x,y})-\bar{g}(t,x)\right) ds,
		\end{split}
	\end{equation}
	then $u(t,x,y)$ is a solution of \eqref{5.1-1},  
	$\exists C_{T}>0$ s.t.,
	\begin{equation}\label{5.3-1}
		\sup_{t\in [0,T]}\sup_{x\in \mathbb{R}^{d_{1}}}|u(t,x,y)|\leq C_{T}(1+|y|),
	\end{equation}
	\begin{equation}\label{5.4-1}
		\sup_{t\in [0,T]} \sup_{\substack{x\in \mathbb{R}^{d_{1}}, y\in \mathbb{R}^{d_{2}}}}|\nabla_{y}u(t,x,y)|\leq C_{T}.
	\end{equation}
\end{theorem}
\begin{proof}
Similar to Theorem \ref{T51}, our proof is based on \cite[Proposition 3.3]{SXX}. We can see that $u(t,x,y)$ is a solution of \eqref{5.1-1} can be deduced by It\^{o} formula. 

From \eqref{5.2-1} and Lemma \ref{L42},
\begin{equation}
	\begin{split}
		\sup_{t\in [0,T]}\sup_{x\in \mathbb{R}^{d_{1}}} |u(t,x,y)|\leq \int_{0}^{\infty}|\mathbb{E}g(t,x,Y_{s}^{x,y})-\bar{g}(t,x)|ds
		&\leq  C_{T}(1+|y|)\int_{0}^{\infty}e^{-\frac{\beta s}{2}}ds\leq  C_{T}(1+|y|),
	\end{split}\nonumber
\end{equation}
so \eqref{5.3-1} is asserted.
Moreover by Leibniz chain rule,
\begin{equation}
	\begin{split}
	 \nabla_{y}u(t,x,y)= \int_{0}^{\infty}\mathbb{E}\nabla_{y}g(t,x,Y_{s}^{x,y})\nabla_{y}Y_{s}^{x,y}ds,
	\end{split}\nonumber
\end{equation}
here $ \nabla_{y}Y_{s}^{x,y}$ satisfies
\begin{equation}
	\left\{
	\begin{aligned}
		&d\nabla_{y}Y_{s}^{x,y}= \nabla_{y}f(t,x,Y_{s}^{x,y})\cdot \nabla_{y}Y_{s}^{x,y}ds\\
		&\nabla_{y}Y_{0}^{x,y}=\frac{Y_{0}^{x,y_{1}}-Y_{0}^{x,y_{2}}}{y_{1}-y_{2}}=\frac{y_{1}-y_{2}}{y_{1}-y_{2}}=I,
	\end{aligned}\nonumber
	\right.
\end{equation}
and by Lemma \ref{L41}, we have 
\begin{equation}
	\begin{aligned}
 \sup_{\substack{x\in \mathbb{R}^{d_{1}}, y\in \mathbb{R}^{d_{2}}}}|\nabla_{y}Y_{s}^{x,y}|\leq C_{T}e^{-\frac{\beta s}{2}},\ s>0,
	\end{aligned}\nonumber
\end{equation}
with the boundness of $\nabla_{y}g(t,x,y)$, we can deduce that $\exists C_{T}>0 $ s.t.,
$$	\sup_{t\in [0,T]}  \sup_{\substack{x\in \mathbb{R}^{d_{1}}, y\in \mathbb{R}^{d_{2}}}}|\nabla_{y}u(t,x,y)|\leq C_{T},$$
we obtain \eqref{5.4-1}.
\end{proof}

\subsection{LLN type estimate for $ b(t,x,y)$}

In this section, we deal with the difficulty arised from $b(t,x,y)-\bar{b}(t,x)$, which satisfies Centering condition, i.e., $\int_{\mathbb{R}^{d_{2}}}b(t,x,y)-\bar{b}(t,x)\mu^{x}(dy)=0$,  then we have the following theorem. Recall that $(a)^{+}=max\{a, 0\}$.
\begin{theorem}\label{T52}
	Suppose that $b(\cdot,\cdot,\cdot)\in C_{p}^{\frac{v}{\alpha_{1}},v, 2+\gamma}\cap C_{b}^{2}(\mathbb{R}^{d_{2}})$, $v\in((\alpha_{1}-\alpha_{2})^{+},\alpha_{1}]$, $\gamma\in(0,1)$ satisfies Lipschitz condition, growth condition, dissipative condition, then we have 
	\begin{equation}\label{5.44}
		\begin{split}
		\mathbb{E}&\left( \sup_{t\in [0,T]}\left| \int_{0}^{t}\left( b(s,X_{s}^{\varepsilon},Y_{s}^{\varepsilon})-\bar{b}(s,X_{s}^{\varepsilon})\right)ds\right|^{p} \right)\\
		&\leq  C_{T,p}\left(\eta_{\varepsilon}^{p\left[ \left( \frac{v}{\alpha_{2}}\right) \wedge \left( 1-\frac{1\vee(\alpha_{1}-v)}{\alpha_{2}}\right) \right] }+\left( \frac{\eta_{\varepsilon}}{\beta_{\varepsilon}}\right)^{p} + \left( \frac{\eta_{\varepsilon}^{1-\frac{1-(1\wedge v)}{\alpha_{2}}}}{\gamma_{\varepsilon}}\right)^{p}\right).
		\end{split}
	\end{equation}
\end{theorem}
\begin{proof}
Our method mainly refers to \cite[Lemma 4.2]{RX}.
Since we let $b(\cdot,\cdot,\cdot)\in C_{p}^{\frac{v}{\alpha_{1}},v, 2+\gamma}\cap C_{b}^{2}(\mathbb{R}^{d_{2}})$, from PDE theory and Theorem \ref{T51-1}, we know that there exist  $ u(\cdot,\cdot,\cdot)\in C_{p}^{\frac{v}{\alpha_{1}},v, 2+\gamma}\cap C_{b}^{2}(\mathbb{R}^{d_{2}})$ such that
\begin{equation}\label{5.45}
	\mathcal{L}_{2}(x,y)u(t,x,y)+b(t,x,y)-\bar{b}(t,x)=0.
\end{equation}

Set $u_{n}$ be the mollifyer of $u$, which is solution of \eqref{5.45}, by It\^{o} formula we deduce that
\begin{equation}\label{5.46}
	\begin{split}
		u_{n}(t,X_{t}^{\varepsilon},Y_{t}^{\varepsilon})
	&=u_{n}(x,y)+\int_{0}^{t}\partial_{s}u_{n}(s,X_{s}^{\varepsilon},Y_{s}^{\varepsilon})ds+\int_{0}^{t}\mathcal{L}_{1}(s,x,y)u_{n}(s,X_{s}^{\varepsilon},Y_{s}^{\varepsilon})ds\\
	&+\frac{1}{\eta_{\varepsilon}}\int_{0}^{t}\mathcal{L}_{2}(x,y)u_{n}(s,X_{s}^{\varepsilon},Y_{s}^{\varepsilon})ds+\frac{1}{\gamma_{\varepsilon}}\int_{0}^{t}\mathcal{L}_{3}(s,x,y)u_{n}(s,X_{s}^{\varepsilon},Y_{s}^{\varepsilon})ds\\
	&+\frac{1}{\beta_{\varepsilon}}\int_{0}^{t}\mathcal{L}_{4}(x,y)u_{n}(s,X_{s}^{\varepsilon},Y_{s}^{\varepsilon})ds+M_{n,t}^{1,\varepsilon}+M_{n,t}^{2,\varepsilon},
	\end{split}
\end{equation}
here $M_{n,t}^{1,\varepsilon}$, $M_{n,t}^{2,\varepsilon}$ are two $ \mathcal{F}_{t}$ martingales defined as
\begin{equation}\label{5.46-1} M_{n,t}^{1,\varepsilon}=\int_{0}^{t}\int_{\mathbb{R}^{d_{1}}}(u_{n}(s-,X_{s-}^{\varepsilon}+z,Y_{s-}^{\varepsilon})-u_{n}(s-,X_{s-}^{\varepsilon},Y_{s-}^{\varepsilon}))\tilde{N}^{1}(ds,dz),
\end{equation}
and
\begin{equation}\label{5.46-2} M_{n,t}^{2,\varepsilon}=\int_{0}^{t}\int_{\mathbb{R}^{d_{2}}}(u_{n}(s-,X_{s-}^{\varepsilon},Y_{s-}^{\varepsilon}+\eta_{\varepsilon}^{-\frac{1}{\alpha_{2}}}z)-u_{n}(s-,X_{s-}^{\varepsilon},Y_{s-}^{\varepsilon}))\tilde{N}^{2}(ds,dz),
\end{equation}
where $\tilde{N}^{1}$, $\tilde{N}^{2}$ are compensated Poisson measure defined in Section 3.

Above calculations lead us to 
\begin{equation}\label{5.47}
	\begin{split}
		&\int_{0}^{t}\mathcal{L}_{2}(x,y)u_{n}(s,X_{s}^{\varepsilon},Y_{s}^{\varepsilon})ds\\
	&=-\eta_{\varepsilon}\left[ u_{n}(x,y)-u_{n}(s,X_{t}^{\varepsilon},Y_{t}^{\varepsilon})+\int_{0}^{t}\partial_{s}u_{n}(s,X_{s}^{\varepsilon},Y_{s}^{\varepsilon})ds
	+\int_{0}^{t}\mathcal{L}_{1}(s,x,y)u_{n}(s,X_{s}^{\varepsilon},Y_{s}^{\varepsilon})ds\right.	\\
	&\left. +\frac{1}{\gamma_{\varepsilon}}\int_{0}^{t}\mathcal{L}_{3}(s,x,y)u_{n}(s,X_{s}^{\varepsilon},Y_{s}^{\varepsilon})ds+\frac{1}{\beta_{\varepsilon}}\int_{0}^{t}\mathcal{L}_{4}(x,y)u_{n}(s,X_{s}^{\varepsilon},Y_{s}^{\varepsilon})ds +M_{n,t}^{1,\varepsilon}+M_{n,t}^{2,\varepsilon}\right],
	\end{split}
\end{equation}
in addition from the non-local Poisson equation \eqref{5.45},
\begin{align}
&\mathbb{E}\left( \sup_{t\in [0,T]}\left| \int_{0}^{t} b(s,X_{s}^{\varepsilon},Y_{s}^{\varepsilon})-\bar{b}(s,X_{s}^{\varepsilon})ds\right| ^{p}\right)\leq  \mathbb{E}\left(  \int_{0}^{T}\left|\mathcal{L}_{2}(x,y)u_{n}(s,X_{s}^{\varepsilon},Y_{s}^{\varepsilon})-\mathcal{L}_{2}(x,y)u(s,X_{s}^{\varepsilon},Y_{s}^{\varepsilon})\right|^{p}ds \right)\nonumber\\
&+ C_{T,p}\cdot \eta_{\varepsilon}^{p}\left[ \mathbb{E}\left( \sup_{t\in [0,T]}|u_{n}(x,y)-u_{n}(t,X_{t}^{\varepsilon},Y_{t}^{\varepsilon})|^{p}\right) +\mathbb{E}\left( \int_{0}^{T}|\mathcal{L}_{1}(s,x,y)u_{n}(s,X_{s}^{\varepsilon},Y_{s}^{\varepsilon})|^{p}ds\right)\right.	\nonumber\\
&\left.+\frac{1}{\gamma_{\varepsilon}^{p}}\mathbb{E}\left( \int_{0}^{T}|\mathcal{L}_{3}(s,x,y)u_{n}(s,X_{s}^{\varepsilon},Y_{s}^{\varepsilon})|^{p}ds\right)+\frac{1}{\beta_{\varepsilon}^{p}}\mathbb{E}\left( \int_{0}^{T}|\mathcal{L}_{4}(x,y)u_{n}(s,X_{s}^{\varepsilon},Y_{s}^{\varepsilon})|^{p}ds\right)\right.\nonumber	\\
&\left.+\mathbb{E}\left( \sup_{t\in [0,T]}|M_{n,t}^{1,\varepsilon}|^{p}\right)+\mathbb{E}\left( \sup_{t\in [0,T]}|M_{n,t}^{2,\varepsilon}|^{p}\right) + \mathbb{E}\left(  \int_{0}^{T}\left|\partial_{s}u_{n}(s,X_{s}^{\varepsilon},Y_{s}^{\varepsilon})\right|^{p}ds \right)\right] \nonumber\\
&=I_{0}+C_{T,p}\cdot \eta_{\varepsilon}^{p}\left( I_{1}+I_{2}+I_{3}+I_{4}+I_{5}+I_{6}+I_{7}\right),\label{5.48}
\end{align}
we will estiamte the above terms respectively. 

Since $2+\gamma> \delta$, we can use \eqref{3.6}, \eqref{5.57} in Lemma \ref{L51}, for $I_{0}$, analogous to proof of  \cite[Lemma 4.2]{RX},
\begin{equation}\label{5.49-1}
	\begin{split}
		I_{0}&=\mathbb{E}\left(  \int_{0}^{T}\left|\mathcal{L}_{2}(x,y)u_{n}(s,X_{s}^{\varepsilon},Y_{s}^{\varepsilon})-\mathcal{L}_{2}(x,y)u(s,X_{s}^{\varepsilon},Y_{s}^{\varepsilon})\right|^{p}ds \right)\\
		&\leq  C_{T,p}n^{-pv}\mathbb{E}\int_{0}^{T}(\left|1+|Y_{s}^{\varepsilon}|^{p}\right|)ds\leq C_{T,p}(1+|y|^{p})n^{-pv},
	\end{split}
\end{equation}
by definition of $u_{n}$, \eqref{5.3-1} in Theorem \ref{T51-1}, and Lemma \ref{L43}
\begin{equation}\label{5.49}
	\begin{split}
		I_{1}&=\mathbb{E}\left( \sup_{t\in [0,T]}|u_{n}(x,y)-u_{n}(t,X_{t}^{\varepsilon},Y_{t}^{\varepsilon})|^{p}\right)\leq \mathbb{E}\left( \sup_{t\in [0,T]}|u(x,y)-u(t,X_{t}^{\varepsilon},Y_{t}^{\varepsilon})|^{p}\right) \\
		&\leq C_{T,p}(1+|y|^{p})+\mathbb{E}\left( \sup_{t\in [0,T]}|Y_{t}^{\varepsilon}|^{p}\right)
	\leq C_{T,p}\eta_{\varepsilon}^{-\frac{p}{\alpha_{2}}}(1+|y|^{p}),
	\end{split}
\end{equation}
for $I_{2}$,  since we have growth condition $|b(t,x,y)| \leq C_{4}(1+K_{t})$, by \eqref{kt}, \eqref{5.58-4} and \eqref{5.58-3} in Lemma \ref{L51},
\begin{equation}\label{5.50}
	\begin{split}
	I_{2}&=\mathbb{E}\left(\int_{0}^{T}|\mathcal{L}_{1}(s,x,y)u_{n}(s,X_{s}^{\varepsilon},Y_{s}^{\varepsilon})|^{p}ds\right)\leq C_{T,p}  \mathbb{E}\left( \int_{0}^{T}|(b(s,X_{s}^{\varepsilon},Y_{s}^{\varepsilon}),\nabla_{x}u_{n}(s,X_{s}^{\varepsilon},Y_{s}^{\varepsilon}))|^{p}ds\right)	\\
	&+ C_{T,p} \mathbb{E}\left(\int_{0}^{T}|-(-\Delta_{x})^{\frac{\alpha_{1}}{2}}u_{n}(s,X_{s}^{\varepsilon},Y_{s}^{\varepsilon})|^{p}ds\right)\\
	&\leq C_{T,p}n^{p(\alpha_{1}-v)}(1+|y|^{p}),
	\end{split}
\end{equation}
for  $I_{3}$, from growth condition $|H(t,x,y)| \leq C_{4}(1+K_{t})$, \eqref{kt}, \eqref{3.4}, \eqref{3.6}, \eqref{5.58-3},
\begin{align}
		I_{3}&=\mathbb{E}\left(\frac{1}{\gamma_{\varepsilon}^{p}}\int_{0}^{T}|H(s,X_{s}^{\varepsilon},Y_{s}^{\varepsilon})\nabla_{x}u_{n}(s,X_{s}^{\varepsilon},Y_{s}^{\varepsilon})|^{p}ds\right)\leq  \frac{C_{T,p}}{\gamma_{\varepsilon}^{p}}n^{1-(1\wedge v)}  \mathbb{E}\left(\int_{0}^{T}\left( 1+|X_{s}^{\varepsilon}|+|Y_{s}^{\varepsilon}|\right)^{p}ds\right)\nonumber\\
		&\leq \frac{C_{T,p}}{\gamma_{\varepsilon}^{p}}n^{1-(1\wedge v)} (1+|y|^{p}),\label{5.51}
\end{align}
and for $I_{4}$, recall that $|c(x,y)|_{\infty} \leq C_{5}$ similar to above analysis,
\begin{equation}\label{5.52}
	\begin{split}
		I_{4}&=\mathbb{E}\left(\frac{1}{\beta_{\varepsilon}^{p}}\int_{0}^{T}|\mathcal{L}_{4}(x,y)u_{n}(s,X_{s}^{\varepsilon},Y_{s}^{\varepsilon})|^{p}ds\right)\leq  \frac{C_{T,p}}{\beta_{\varepsilon}^{p}} (1+|y|^{p}).
	\end{split}
\end{equation}

We can deduce from Burkholder-Davies-Gundy's inequality, \eqref{3.6}, \eqref{5.58-3}
\begin{align}
	&\mathbb{E}\left( \sup_{t\in [0,T]}|M_{n,t}^{1,\varepsilon}|^{p}\right)\leq C_{T,p}
	\mathbb{E}\left(\sup_{t\in [0,T]}\left| \int_{0}^{t}\left(\int_{|z|\leq 1}u_{n}(s,X_{s}^{\varepsilon}+z,Y_{s}^{\varepsilon})-u_{n}(s,X_{s}^{\varepsilon},Y_{s}^{\varepsilon})\tilde{N}_{1}(ds,dz) \right)\right| ^{p}\right)\nonumber\\
	&+ C_{T,p}
	\mathbb{E}\left(\sup_{t\in [0,T]}\left| \int_{0}^{t}\left(\int_{|z|> 1}u_{n}(s,X_{s}^{\varepsilon}+z,Y_{s}^{\varepsilon})-u_{n}(s,X_{s}^{\varepsilon},Y_{s}^{\varepsilon})\tilde{N}_{1}(ds,dz) \right)\right| ^{p}\right)\nonumber\\
	&\leq  C_{T,p}  \int_{0}^{T}\mathbb{E}\left[\left(\int_{|z|\leq 1}|z\nabla_{x} u_{n}(s,X_{t}^{\varepsilon},Y_{t}^{\varepsilon})|^{2}\nu_{1}(dz) \right)^{\frac{p}{2}}+\int_{|z|> 1}|z\nabla_{x} u_{n}(s,X_{t}^{\varepsilon},Y_{t}^{\varepsilon})|^{p}\nu_{1}(dz)\right]ds \nonumber\\
	&\leq C_{T,p}n^{1-(1\wedge v)}  \int_{0}^{T}\mathbb{E}\left[\left(\int_{|z|\leq 1}|z|^{2}(1+|Y_{s}^{\varepsilon}|^{2})\nu_{1}(dz) \right)^{\frac{p}{2}}+\int_{|z|> 1}|z|^{p}(1+|Y_{s}^{\varepsilon}|^{p})\nu_{1}(dz) \right]ds\nonumber\\
	& \leq C_{T,p}n^{1-(1\wedge v)}(1+|y|^{p}),\label{5.53}
\end{align}
  since $1+\gamma> 1> \delta$,   $ u(\cdot,\cdot,\cdot)\in C_{p}^{\frac{v}{\alpha_{1}},v, 2+\gamma}\cap C_{b}^{2}(\mathbb{R}^{d_{2}})$, then $ \nabla_{y}u_{n}=(\nabla_{y}u)\ast\rho_{2}^{n}\ast\rho_{1}^{n}\in C_{p}^{\frac{v}{\alpha_{1}},v, 1+\gamma}\cap C_{b}^{1}(\mathbb{R}^{d_{2}})$,  with \eqref{5.4-1} in Theorem \ref{T51-1},
\begin{align}
		&\mathbb{E}\left( \sup_{t\in [0,T]}|M_{n,t}^{2,\varepsilon}|^{p}\right)\leq C_{T,p}
	\mathbb{E}\left(\sup_{t\in [0,T]}\left| \int_{0}^{t}\left(\int_{|z|\leq 1}u_{n}(s,X_{s}^{\varepsilon},Y_{s}^{\varepsilon}+\eta_{\varepsilon}^{-\frac{1}{\alpha_{2}}}z)-u_{n}(s,X_{s}^{\varepsilon},Y_{s}^{\varepsilon})\tilde{N}_{2}(ds,dz) \right)\right| ^{p}\right)\nonumber\\
	&+ C_{T,p}
	\mathbb{E}\left(\sup_{t\in [0,T]}\left| \int_{0}^{t}\left(\int_{|z|> 1}u_{n}(s,X_{s}^{\varepsilon},Y_{s}^{\varepsilon}+\eta_{\varepsilon}^{-\frac{1}{\alpha_{2}}}z)-u_{n}(s,X_{s}^{\varepsilon},Y_{s}^{\varepsilon})\tilde{N}_{2}(ds,dz) \right)\right| ^{p}\right)\nonumber\\
	&\leq  C_{T,p} \eta_{\varepsilon}^{-\frac{p}{\alpha_{2}}} \int_{0}^{T}\mathbb{E}\left[\left(\int_{|z|\leq 1}|z\nabla_{y} u_{n}(s,X_{t}^{\varepsilon},Y_{t}^{\varepsilon})|^{2}\nu_{2}(dz) \right)^{\frac{p}{2}}+\int_{|z|> 1}|z\nabla_{y} u_{n}(s,X_{t}^{\varepsilon},Y_{t}^{\varepsilon})|^{p}\nu_{2}(dz) \right]ds \nonumber\\
	&\leq C_{T,p}\eta_{\varepsilon}^{-\frac{p}{\alpha_{2}}}  \int_{0}^{T}\left[\left(\int_{|z|\leq 1}|z|^{2}\nu_{2}(dz) \right)^{\frac{p}{2}}+\int_{|z|> 1}|z|^{p}\nu_{2}(dz) \right]ds \leq C_{T,p}\eta_{\varepsilon}^{-\frac{p}{\alpha_{2}}},\label{5.54}
\end{align}
for $I_{7}$, by \eqref{5.58-1},
\begin{equation}\label{5.55-1}
	\begin{split}
\mathbb{E}\left(  \int_{0}^{T}\left|\partial_{s}u_{n}(s,X_{s}^{\varepsilon},Y_{s}^{\varepsilon})\right|^{p}ds
\right)  \leq C_{T,p}n^{p(\alpha_{1}-v)}(1+|y|^{p}),
	\end{split}
\end{equation}
combining \eqref{5.49-1}-\eqref{5.55-1} together, take $n=\eta_{\varepsilon}^{-\frac{1}{\alpha_{2}}}$,
\begin{align}
		&\mathbb{E}\left( \sup_{t\in [0,T]}\left| \int_{0}^{t}\left( b(s,X_{s}^{\varepsilon},Y_{s}^{\varepsilon})-\bar{b}(s,X_{s}^{\varepsilon})\right)ds\right|^{p} \right)\nonumber\\
		&\leq C_{T,p}\left( \eta_{\varepsilon}^{p(1-\frac{1}{\alpha_{2}})}+\eta_{\varepsilon}^{p(1-\frac{\alpha_{1}-v}{\alpha_{2}})}+\eta_{\varepsilon}^{\frac{pv}{\alpha_{2}}}+\left( \frac{\eta_{\varepsilon}}{\gamma_{\varepsilon}}\right)^{p}+\left( \frac{\eta_{\varepsilon}}{\beta_{\varepsilon}}\right)^{p}+\left( \frac{\eta_{\varepsilon}^{1-\frac{1-(1\wedge v)}{\alpha_{2}}}}{\gamma_{\varepsilon}}\right)^{p} \right)\nonumber\\
		&\leq  C_{T,p}\left(\eta_{\varepsilon}^{\frac{pv}{\alpha_{2}}}+\eta_{\varepsilon}^{p(1-\frac{1\vee(\alpha_{1}-v)}{\alpha_{2}})}+\left( \frac{\eta_{\varepsilon}}{\beta_{\varepsilon}}\right)^{p} + \left( \frac{\eta_{\varepsilon}^{1-\frac{1-(1\wedge v)}{\alpha_{2}}}}{\gamma_{\varepsilon}}\right)^{p}\right)\nonumber\\
		&\leq  C_{T,p}\left(\eta_{\varepsilon}^{p\left[ \left( \frac{v}{\alpha_{2}}\right) \wedge \left( 1-\frac{1\vee(\alpha_{1}-v)}{\alpha_{2}}\right) \right] }+\left( \frac{\eta_{\varepsilon}}{\beta_{\varepsilon}}\right)^{p} + \left( \frac{\eta_{\varepsilon}^{1-\frac{1-(1\wedge v)}{\alpha_{2}}}}{\gamma_{\varepsilon}}\right)^{p}\right),\label{5.55}
\end{align}
we used the fact that   $\eta_{\varepsilon}=o(\eta_{\varepsilon}^{1-\frac{1}{\alpha_{2}}})$ in second inequality.
\end{proof}

\subsection{CLT type estimate for $\frac{1}{\gamma_{\varepsilon}}H(t,x,y)$}

We will discuss this section in four regimes, which divided by the relationships among $\gamma_{\varepsilon}$, $\beta_{\varepsilon}$, $\eta_{\varepsilon}$.
We assume that $H(t,x,y)$ satisfies Centering condition in \eqref{2.13}, i.e., $\int_{\mathbb{R}^{ d_{2}}}H(t,x,y)\mu^{x}(dy)=0,$
here $\mu^{x}$ is the invariant measure of \eqref{4.1}. 

Before we prove next theorem, recall that
\begin{equation}
\bar{c}(t,x)=\int_{\mathbb{R}^{ d_{2}}}c(x,y)\nabla_{y}u(t,x,y)\mu^{x}(dy),
\nonumber
\end{equation}
here $u(t,x,y)$ is the solution of following nonlocal Poisson equation
\begin{equation}\label{5.59-1}
\mathcal{L}_{2}(x,y)u(t,x,y)+H(t,x,y)=0.
\end{equation}

\begin{theorem}\label{T53}
	Suppose that  Lipschitz condition, growth condition, dissipative condition valid,  then we have for  $v\in((\alpha_{1}-\alpha_{2})^{+},\alpha_{1}]$, $\gamma\in(0,1)$, $ \lim\limits_{\varepsilon \rightarrow0}\frac{\eta_{\varepsilon}^{\left[ \left( \frac{v}{\alpha_{2}}\right) \wedge \left( 1-\frac{1\vee(\alpha_{1}-v)}{\alpha_{2}}\right) \right] }}{\gamma_{\varepsilon} } =0$,

Regime 1:   $H(t,x,y)\in C_{p}^{\frac{v}{\alpha_{1}},v, 2+\gamma}\cap C_{b}^{2}(\mathbb{R}^{d_{2}})$,    
\begin{equation}\label{5.59}
	\mathbb{E}\left( \sup_{t\in [0,T]}\left| \int_{0}^{t} \frac{1}{\gamma_{\varepsilon}}H(s,X_{s}^{\varepsilon},Y_{s}^{\varepsilon})ds\right|^{p} \right)  \leq    C_{T,p}\left(\left(  \frac{\eta_{\varepsilon}}{\gamma_{\varepsilon}\beta_{\varepsilon}}\right) ^{p}+\left(  \frac{\eta_{\varepsilon}^{1-\frac{1-(1\wedge v)}{\alpha_{2}}}}{\gamma_{\varepsilon}^{2}}\right) ^{p}+\left( \frac{\eta_{\varepsilon}^{\left[ \left( \frac{v}{\alpha_{2}}\right) \wedge \left( 1-\frac{1\vee(\alpha_{1}-v)}{\alpha_{2}}\right) \right] }}{\gamma_{\varepsilon} } \right) ^{p}\right);
\end{equation}

Regime 2:  $H(t,x,y)\in C_{p}^{\frac{v}{\alpha_{1}},v, 2+\gamma}\cap C_{b}^{2}(\mathbb{R}^{d_{2}})$,  
\begin{equation}\label{5.60}
		\begin{split}
	\mathbb{E}&\left( \sup_{t\in [0,T]}\left| \int_{0}^{t}\left(  \frac{1}{\gamma_{\varepsilon}}H(s,X_{s}^{\varepsilon},Y_{s}^{\varepsilon})- \bar{c}(s,X_{s}^{\varepsilon})\right) ds\right|^{p} \right) \\
	& \leq  C_{T,p}\left(                                                                                                                                                                                                                                                                           \left(  \frac{\eta_{\varepsilon}^{1-\frac{1-(1\wedge v)}{\alpha_{2}}}}{\gamma_{\varepsilon}^{2}}\right) ^{p}+ \left(  \frac{\eta_{\varepsilon}^{\left[ \left( \frac{v}{\alpha_{2}}\right) \wedge \left( 1-\frac{1\vee(\alpha_{1}-v)}{\alpha_{2}}\right) \right] }}{\gamma_{\varepsilon} } \right) ^{p}+\gamma_{\varepsilon}^{p}\right) .
		\end{split}
\end{equation}
\end{theorem}
\begin{proof}
  Our method mainly refers to \cite[Lemma 4.4]{RX}. 	We first prove Regime 1, in this case, take $n=\eta_{\varepsilon}^{-\frac{1}{\alpha_{2}}}$, we deduce from Theorem \ref{T52} that 
\begin{equation}
	\begin{split}
		\mathbb{E}\left( \sup_{t\in [0,T]}\left| \int_{0}^{t} \frac{1}{\gamma_{\varepsilon}}H(s,X_{s}^{\varepsilon},Y_{s}^{\varepsilon})ds\right|^{p} \right)  &\leq  C_{T,p}\left(\left(  \frac{\eta_{\varepsilon}}{\gamma_{\varepsilon}\beta_{\varepsilon}}\right) ^{p}+\left(  \frac{\eta_{\varepsilon}^{1-\frac{1-(1\wedge v)}{\alpha_{2}}}}{\gamma_{\varepsilon}^{2}}\right) ^{p}+\left( \frac{\eta_{\varepsilon}^{\left[ \left( \frac{v}{\alpha_{2}}\right) \wedge \left( 1-\frac{1\vee(\alpha_{1}-v)}{\alpha_{2}}\right) \right] }}{\gamma_{\varepsilon} } \right) ^{p}\right).
	\nonumber
	\end{split}
\end{equation}

Let $u_{n}$ be the mollifyer of $u$, which is the solution of \eqref{5.59-1}, then by It\^{o} formula, similar to \eqref{5.46},
\begin{equation}\label{5.63}
	\begin{split}
		u_{n}(t,X_{t}^{\varepsilon},Y_{t}^{\varepsilon})
		&=M_{t,n}^{1,\varepsilon}+M_{t,n}^{2,\varepsilon}+u_{n}(x,y)+\int_{0}^{t}\partial_{s}u_{n}(s,X_{s}^{\varepsilon},Y_{s}^{\varepsilon})ds+\int_{0}^{t}\mathcal{L}_{1}(s,x,y)u_{n}(s,X_{s}^{\varepsilon},Y_{s}^{\varepsilon})ds\\
		&+\frac{1}{\eta_{\varepsilon}}\int_{0}^{t}\mathcal{L}_{2}(x,y)u_{n}(s,X_{s}^{\varepsilon},Y_{s}^{\varepsilon})ds+\frac{1}{\gamma_{\varepsilon}}\int_{0}^{t}\mathcal{L}_{3}(s,x,y)u_{n}(s,X_{s}^{\varepsilon},Y_{s}^{\varepsilon})ds\\
		&+\frac{1}{\beta_{\varepsilon}}\int_{0}^{t}\mathcal{L}_{4}(x,y)u_{n}(s,X_{s}^{\varepsilon},Y_{s}^{\varepsilon})ds,
	\end{split}
\end{equation}
hence from \eqref{5.1},
\begin{equation}\label{5.64}
	\begin{split}
		&\int_{0}^{t} g(s,X_{s}^{\varepsilon},Y_{s}^{\varepsilon})-\bar{g}(s,X_{s}^{\varepsilon},Y_{s}^{\varepsilon})ds=\int_{0}^{t} \mathcal{L}_{2}(x,y)u_{n}(s,X_{s}^{\varepsilon},Y_{s}^{\varepsilon})- \mathcal{L}_{2}(x,y)u(s,X_{s}^{\varepsilon},Y_{s}^{\varepsilon})ds\\
		&+\eta_{\varepsilon}\left[ u_{n}(x,y)-u_{n}(t,X_{t}^{\varepsilon},Y_{t}^{\varepsilon})+\int_{0}^{t}\mathcal{L}_{1}(s,x,y)u_{n}(s,X_{s}^{\varepsilon},Y_{s}^{\varepsilon})ds+\frac{1}{\gamma_{\varepsilon}}\int_{0}^{t}\mathcal{L}_{3}(s,x,y)u_{n}(s,X_{s}^{\varepsilon},Y_{s}^{\varepsilon})ds\right.	\\
		&\left. +\frac{1}{\beta_{\varepsilon}}\int_{0}^{t}\mathcal{L}_{4}(x,y)u_{n}(s,X_{s}^{\varepsilon},Y_{s}^{\varepsilon})ds +M_{t,n}^{1,\varepsilon}+M_{t,n}^{2,\varepsilon}+\int_{0}^{t} \partial_{s}u_{n}(s,X_{s}^{\varepsilon},Y_{s}^{\varepsilon}) ds\right],
	\end{split}
\end{equation}
then  from the structure of $H(t,x,y)$, for Regime 2,  we can see that
\begin{align}
		&\mathbb{E}\left( \sup_{t\in [0,T]}\left| \int_{0}^{t} \left( \frac{1}{\gamma_{\varepsilon}}H(s,X_{s}^{\varepsilon},Y_{s}^{\varepsilon})- \bar{c}(s,X_{s}^{\varepsilon})\right) ds\right|^{p} \right)\nonumber \\
		&\leq C_{T,p}\cdot\frac{1}{\gamma_{\varepsilon}^{p}}\int_{0}^{T}| \mathcal{L}_{2}(x,y)u_{n}(s,X_{s}^{\varepsilon},Y_{s}^{\varepsilon})-\mathcal{L}_{2}(x,y)u(s,X_{s}^{\varepsilon},Y_{s}^{\varepsilon})|^{p} ds\nonumber \\
		&+  \frac{\eta_{\varepsilon}^{p}}{\gamma_{\varepsilon}^{p}}\left[ \mathbb{E}\left( \sup_{t\in [0,T]}|u_{n}(x,y)-u_{n}(t,X_{t}^{\varepsilon},Y_{t}^{\varepsilon})|^{p}\right)\right] +\mathbb{E}\left( \int_{0}^{T}|\mathcal{L}_{1}(s,x,y)u_{n}(s,X_{s}^{\varepsilon},Y_{s}^{\varepsilon})|^{p}ds\right)\nonumber	\\
		&+\frac{1}{\gamma_{\varepsilon}^{p}}\mathbb{E}\left( \int_{0}^{T}|\mathcal{L}_{3}(s,x,y)u_{n}(s,X_{s}^{\varepsilon},Y_{s}^{\varepsilon})|^{p}ds\right)+\mathbb{E}\left( \sup_{t\in [0,T]}|M_{t,n}^{1,\varepsilon}|^{p}\right)+\mathbb{E}\left( \sup_{t\in [0,T]}|M_{t,n}^{2,\varepsilon}|^{p}\right)\nonumber\\
		&
		 + \mathbb{E}\left(  \int_{0}^{T}\left|\partial_{s}u_{n}(s,X_{s}^{\varepsilon},Y_{s}^{\varepsilon})\right|^{p}ds
		 \right) +\mathbb{E}\left( \int_{0}^{T}|\mathcal{L}_{4}(x,y)u_{n}(s,X_{s}^{\varepsilon},Y_{s}^{\varepsilon})- \bar{c}(s,X_{s}^{\varepsilon})|^{p}ds\right)\nonumber\\
		&=I_{0}+ I_{1}+I_{2}+I_{3}+I_{4}+I_{5}+I_{6}+I_{7},\label{5.65}
\end{align}
by Lemma \ref{L51}, we have 
\begin{equation}\label{5.66}
	\begin{aligned}
		I_{0}+I_{1}+I_{2}+I_{3}+I_{4}+I_{5}+I_{6}&\leq C_{T,p}\left( \frac{n^{-vp}}{\gamma_{\varepsilon}^{p}}+\frac{\eta_{\varepsilon}^{p}}{\gamma_{\varepsilon}^{p}}n^{p(\alpha_{1}-v)}+\frac{\eta_{\varepsilon}^{p}}{\gamma_{\varepsilon}^{2p}}n^{p(1-(1\wedge v))}+\frac{\eta_{\varepsilon}^{p}}{\gamma_{\varepsilon}^{p}}+\frac{\eta_{\varepsilon}^{(1-\frac{1}{\alpha_{2}})p}}{\gamma_{\varepsilon}^{p}} \right), \\
	\end{aligned}
\end{equation}
in particular,
\begin{equation}\label{5.67}
	\begin{split}
		I_{7}&=\mathbb{E}\left( \int^{T}_{0}|c(X_{s}^{\varepsilon},Y_{s}^{\varepsilon})\nabla_{y} u_{n}(s,X_{s}^{\varepsilon},Y_{s}^{\varepsilon})-\bar{c}(s,X_{s}^{\varepsilon})|^{p}ds\right) \\
		&\leq\mathbb{E}\left(  \int^{T}_{0}\left|c(X_{s}^{\varepsilon},Y_{s}^{\varepsilon})\nabla_{y} u_{n}(s,X_{s}^{\varepsilon},Y_{s}^{\varepsilon})-c(X_{s}^{\varepsilon},Y_{s}^{\varepsilon})\nabla_{y} u(s,X_{s}^{\varepsilon},Y_{s}^{\varepsilon})\right| ^{p}ds\right)\\ &+\mathbb{E}\left(  \int^{T}_{0}\left|c(X_{s}^{\varepsilon},Y_{s}^{\varepsilon})\nabla_{y} u(s,X_{s}^{\varepsilon},Y_{s}^{\varepsilon})-\bar{c}(s,X_{s}^{\varepsilon})\right|^{p}ds\right)=I_{71}+I_{72},
	\end{split}
\end{equation}
for $H(t,x,y)\in C_{p}^{\frac{v}{\alpha_{1}},v, 2+\gamma}\cap C_{b}^{2}(\mathbb{R}^{d_{2}})$, then $u \in C_{p}^{\frac{v}{\alpha_{1}},v, 2+\gamma}\cap C_{b}^{2}(\mathbb{R}^{d_{2}})$, using Lemma \ref{L51} and growth condition,
\begin{equation}\label{5.68}
	\begin{split}
		I_{71}&\leq  C_{T,p}\mathbb{E}\left( \int^{T}_{0}\| \nabla_{y} u_{n}(s,X_{s}^{\varepsilon},Y_{s}^{\varepsilon})-\nabla_{y} u(s,X_{s}^{\varepsilon},Y_{s}^{\varepsilon})\|^{p}_{\infty}(1+|x|^{p}+|y|^{p})ds\right) \leq C_{T,p} n^{-pv},
	\end{split}
\end{equation}
also $c(X_{s}^{\varepsilon},Y_{s}^{\varepsilon})\nabla_{y} u(s,X_{s}^{\varepsilon},Y_{s}^{\varepsilon})\in C_{p}^{\frac{v}{\alpha_{1}},v, 1+\gamma}\cap C_{b}^{1}(\mathbb{R}^{d_{2}})$, and  $I_{72}$ satisfies Centering condition, by \eqref{5.4-1} in Theorem \ref{T51-1}, Theorem \ref{T52},
\begin{equation}\label{5.69}
	\begin{split}
		I_{72}\leq  C_{T,p}\left(\eta_{\varepsilon}^{\frac{pv}{\alpha_{2}}}+\eta_{\varepsilon}^{p(1-\frac{1\vee(\alpha_{1}-v)}{\alpha_{2}})}+\left( \frac{\eta_{\varepsilon}}{\beta_{\varepsilon}}\right)^{p} + \left( \frac{\eta_{\varepsilon}^{1-\frac{1-(1\wedge v)}{\alpha_{2}}}}{\gamma_{\varepsilon}}\right)^{p}\right),
	\end{split}
\end{equation}
finally we get 
\begin{equation}\label{5.70}
	\begin{aligned}
		&\mathbb{E}\left( \sup_{t\in [0,T]}\left| \int_{0}^{t}\left(  \frac{1}{\gamma_{\varepsilon}}H(s,X_{s}^{\varepsilon},Y_{s}^{\varepsilon})- \bar{c}(s,X_{s}^{\varepsilon})\right) ds\right|^{p} \right)\\
		&\leq C_{T,p}\left( \frac{n^{-vp}}{\gamma_{\varepsilon}^{p}}+\frac{\eta_{\varepsilon}^{p}}{\gamma_{\varepsilon}^{p}}n^{p(2-v)}+\frac{\eta_{\varepsilon}^{p}}{\gamma_{\varepsilon}^{2p}}n^{p(1-(1\wedge v))}+\frac{\eta_{\varepsilon}^{p}}{\gamma_{\varepsilon}^{p}}n^{p(1-(1\wedge v))}+\frac{\eta_{\varepsilon}^{p}}{\beta_{\varepsilon}^{p}}+n^{-pv}+\eta_{\varepsilon}^{(1-\frac{1}{\alpha_{2}})p} \right),\\
	\end{aligned}
\end{equation}
for $\eta_{\varepsilon}=\gamma_{\varepsilon}\beta_{\varepsilon}$,  take $n=\eta_{\varepsilon}^{-\frac{1}{\alpha_{2}}}$,  then 
\begin{equation}\label{5.71}
	\begin{aligned}
		\mathbb{E}&\left( \sup_{t\in [0,T]}\left| \int_{0}^{t}\left(  \frac{1}{\gamma_{\varepsilon}}H(s,X_{s}^{\varepsilon},Y_{s}^{\varepsilon})- \bar{c}(s,X_{s}^{\varepsilon})\right) ds\right|^{p} \right)\\
		&\leq C_{T,p}\left(                                                                                                                                                                                                                                                                           \left(  \frac{\eta_{\varepsilon}^{1-\frac{1-(1\wedge v)}{\alpha_{2}}}}{\gamma_{\varepsilon}^{2}}\right) ^{p}+ \left( \frac{\eta_{\varepsilon}^{\left[ \left( \frac{v}{\alpha_{2}}\right) \wedge \left( 1-\frac{1\vee(\alpha_{1}-v)}{\alpha_{2}}\right) \right] }}{\gamma_{\varepsilon} } \right) ^{p}+\gamma_{\varepsilon}^{p}\right) ,\\
	\end{aligned}
\end{equation}
the proof this theorem is complete.
\end{proof}

\begin{remark}\label{R51}
	Next, we clarify why the averaged equations cannot be derived as neither
	\begin{equation}\label{5.R3}
		d\bar{X}^{3}_{t}=(\bar{b}(t,\bar{X}^{3}_{t})+\bar{H}(t,\bar{X}^{3}_{t}))dt+dL_{t}^{1},
	\end{equation}
	or
		\begin{equation}\label{5.R4}
		d\bar{X}^{4}_{t}=(\bar{b}(t,\bar{X}^{4}_{t})+\bar{c}(t,\bar{X}^{4}_{t})+\bar{H}(t,\bar{X}^{4}_{t}))dt+dL_{t}^{1},
	\end{equation}
	where \begin{equation}
		\bar{H}(t,x)=\int_{\mathbb{R}^{ d_{2}}}H(t,x,y)\nabla_{x}u(t,x,y)\mu^{x}(dy),
		\nonumber
	\end{equation}
	here $u(t,x,y)$ is the solution of \eqref{5.59-1}. 
	
	In Regimes $3$ we aim to derive \eqref{5.R3}, here we let $ \lim\limits_{\varepsilon \rightarrow0}\frac{\gamma_{\varepsilon}}{\beta_{\varepsilon} }=0$,  $\eta_{\varepsilon}=\gamma_{\varepsilon}^{2}$, then the  boundedness of $	\underset{t\in [0,T]}{\sup}\underset{ x\in \mathbb{R}^{d_{1}}}{\sup}|\nabla_{x}u(t,x,y)|$ becomes crucial for controlling the term $H\cdot\nabla_{x} u$, as indicated in \eqref{5.671}. This necessitates our employment of Theorem $\ref{T51}$ rather than Theorem $\ref{T51-1}$, consequently requiring the regularity assumption  $H(t,x,y)\in C_{p}^{\frac{v}{\alpha_{1}},2+\gamma,2+\gamma}\cap C_{b}^{2+\gamma,2+\gamma}(\mathbb{R}^{d_{1}+d_{2}})$. Particularly, we emphasize that $1+\gamma>\alpha_{1}\geq v$ and $\alpha_{1}>1$,  so that $\|\nabla_{x}u_ {n}(\cdot,\cdot,y)\|_{\infty}\leq C\cdot (1+|y|)$, see computations of \eqref{5.58-32} in Lemma $\ref{L51}$. We mention that the regime classifications  $ \lim\limits_{\varepsilon \rightarrow0}\frac{\gamma_{\varepsilon}}{\beta_{\varepsilon} }=0$ and $\eta_{\varepsilon}=\gamma_{\varepsilon}^{2}$ enable us to prove that  $\mathcal{L}_{3}(t,x,y)u_{n}-\bar{H}$ converges to $0$,   however, these assumptions may introduce contradictions in the following analysis.
\begin{align}
&\mathbb{E}\left( \sup_{t\in [0,T]}\left| \int_{0}^{t} \frac{1}{\gamma_{\varepsilon}}H(s,X_{s}^{\varepsilon},Y_{s}^{\varepsilon})- \bar{H}(s,X_{s}^{\varepsilon})ds\right|^{p} \right)\nonumber \\
&\leq C_{T,p}\cdot\frac{1}{\gamma_{\varepsilon}^{p}} \mathbb{E}\left( \int_{0}^{T}| \mathcal{L}_{2}(x,y)u_{n}(s,X_{s}^{\varepsilon},Y_{s}^{\varepsilon})-\mathcal{L}_{2}(x,y)u(s,X_{s}^{\varepsilon},Y_{s}^{\varepsilon})|^{p} ds\right) \nonumber \\
&+  \frac{\eta_{\varepsilon}^{p}}{\gamma_{\varepsilon}^{p}}\left[ \mathbb{E}\left( \sup_{t\in [0,T]}|u_{n}(x,y)-u_{n}(t,X_{t}^{\varepsilon},Y_{t}^{\varepsilon})|^{p}\right)\right] +\mathbb{E}\left( \int_{0}^{T}|\mathcal{L}_{1}(s,x,y)u_{n}(s,X_{s}^{\varepsilon},Y_{s}^{\varepsilon})|^{p}ds\right)\nonumber	\\
&+\frac{1}{\gamma_{\varepsilon}^{p}}\mathbb{E}\left( \int_{0}^{T}|\mathcal{L}_{4}(x,y)u_{n}(s,X_{s}^{\varepsilon},Y_{s}^{\varepsilon})|^{p}ds\right)+\mathbb{E}\left( \sup_{t\in [0,T]}|M_{t,n}^{1,\varepsilon}|^{p}\right)+\mathbb{E}\left( \sup_{t\in [0,T]}|M_{t,n}^{2,\varepsilon}|^{p}\right)\nonumber	\\
&
 +  \mathbb{E}\left(  \int_{0}^{T}\left|\partial_{s}u_{n}(s,X_{s}^{\varepsilon},Y_{s}^{\varepsilon})\right|^{p} ds
\right) +\mathbb{E}\left( \int_{0}^{T}|\mathcal{L}_{3}(s,x,y)u_{n}(s,X_{s}^{\varepsilon},Y_{s}^{\varepsilon})- \bar{H}(s,X_{s}^{\varepsilon})|^{p}ds\right)\nonumber\\
&=I_{0}+ I_{1}+I_{2}+I_{3}+I_{4}+I_{5}+I_{6}+I_{7},\label{5.72}
\end{align}
from Lemma $\ref{L51}$,
\begin{equation}\label{5.661}
	\begin{aligned}
			I_{0}+I_{1}+I_{2}+I_{3}+I_{4}+I_{5}+I_{6}&\leq C_{T,p}\left( \frac{n^{-vp}}{\gamma_{\varepsilon}^{p}}+\frac{\eta_{\varepsilon}^{p}}{\gamma_{\varepsilon}^{p}}n^{p(\alpha_{1}-v)}+\frac{\eta_{\varepsilon}^{p}}{\gamma_{\varepsilon}^{p}\beta_{\varepsilon}^{p}}+\frac{\eta_{\varepsilon}^{p}}{\gamma_{\varepsilon}^{2p}}+\frac{\eta_{\varepsilon}^{(1-\frac{1}{\alpha_{2}})p}}{\gamma_{\varepsilon}^{p}} \right),\\
	\end{aligned}
\end{equation}
thus, 
\begin{align}
		I_{7}&=\mathbb{E}\left( \int^{T}_{0}|H(s,X_{s}^{\varepsilon},Y_{s}^{\varepsilon})\nabla_{x} u_{n}(s,X_{s}^{\varepsilon},Y_{s}^{\varepsilon})-\bar{H}(s,X_{s}^{\varepsilon})|^{p}ds\right) \nonumber\\
		&\leq\mathbb{E}\left( \int^{T}_{0}|H(s,X_{s}^{\varepsilon},Y_{s}^{\varepsilon})\nabla_{x} u_{n}(s,X_{s}^{\varepsilon},Y_{s}^{\varepsilon})-H(s,X_{s}^{\varepsilon},Y_{s}^{\varepsilon})\nabla_{x} u(s,X_{s}^{\varepsilon},Y_{s}^{\varepsilon})|^{p}ds\right)\nonumber\\ 
		&+\mathbb{E}\left( \int^{T}_{0}|H(s,X_{s}^{\varepsilon},Y_{s}^{\varepsilon})\nabla_{x} u(s,X_{s}^{\varepsilon},Y_{s}^{\varepsilon})-\bar{H}(s,X_{s}^{\varepsilon})|^{p}ds\right)=I_{71}+I_{72},\label{5.671}
\end{align}
similar to Regime $2$, together with Theorem $\ref{T51}$, Lemma $\ref{L51}$ and Theorem $\ref{T52}$,
\begin{equation}\label{5.75}
	\begin{split}
		I_{7}\leq  C_{T,p}\left(\eta_{\varepsilon}^{\frac{pv}{\alpha_{2}}}+\eta_{\varepsilon}^{p(1-\frac{1\vee(\alpha_{1}-v)}{\alpha_{2}})}+\left( \frac{\eta_{\varepsilon}}{\beta_{\varepsilon}}\right)^{p} + \left( \frac{\eta_{\varepsilon}^{1-\frac{1-(1\wedge v)}{\alpha_{2}}}}{\gamma_{\varepsilon}}\right)^{p}\right),
	\end{split}
\end{equation}
however we notice that  $\eta_{\varepsilon}=\gamma_{\varepsilon}^{2}$, then in \eqref{5.661} $\frac{\eta_{\varepsilon}^{1-\frac{1}{\alpha_{2}}}}{\gamma_{\varepsilon}}=\gamma_{\varepsilon}^{1-\frac{2}{\alpha_{2}}}$, but when $1<\alpha_{2}<2$, $\gamma_{\varepsilon}^{1-\frac{2}{\alpha_{2}}}$ definitely diverges as  $\gamma_{\varepsilon}\rightarrow0$, which   prevents the convergence  to 0 of $ \dfrac{\eta_{\varepsilon}^{p}}{\gamma_{\varepsilon}^{p}}\mathbb{E}\left( \underset{t\in [0,T]}{\sup}|u_{n}(t,X_{t}^{\varepsilon},Y_{t}^{\varepsilon})|^{p}\right)$ and the scaled martingale term  $\dfrac{\eta_{\varepsilon}^{p}}{\gamma_{\varepsilon}^{p}}\mathbb{E}\left(\underset{t\in [0,T]}{\sup} |M_{t,n}^{2,\varepsilon}|^{p}\right)$ associated with  $ Y^{\varepsilon}_{t}$, see \eqref{5.49} and \eqref{5.54} respectively.

As for Regime $4$ we target to \eqref{5.R4}, in this case,  to maintain consistency with the terms $ \mathcal{L}_{4}(x,y)u_{n}-\bar{c}$ and $\mathcal{L}_{3}(t,x,y)u_{n}-\bar{H}$ respectively,  we must impose the conditions  $\eta_{\varepsilon}=\gamma_{\varepsilon}^{2}=\gamma_{\varepsilon}\beta_{\varepsilon}$, then
\begin{align}
	&\mathbb{E}\left( \sup_{t\in [0,T]}\left| \int_{0}^{t} \frac{1}{\gamma_{\varepsilon}}H(s,X_{s}^{\varepsilon},Y_{s}^{\varepsilon})- \bar{H}(s,X_{s}^{\varepsilon})- \bar{c}(s,X_{s}^{\varepsilon})ds\right|^{p} \right)\nonumber \\
	&\leq C_{T,p}\cdot\frac{1}{\gamma_{\varepsilon}^{p}} \mathbb{E}\left( \int_{0}^{T}| \mathcal{L}_{2}(x,y)u_{n}(s,X_{s}^{\varepsilon},Y_{s}^{\varepsilon})-\mathcal{L}_{2}(x,y)u(s,X_{s}^{\varepsilon},Y_{s}^{\varepsilon})|^{p} ds\right) \nonumber \\
	&+  \frac{\eta_{\varepsilon}^{p}}{\gamma_{\varepsilon}^{p}}\left[ \mathbb{E}\left( \sup_{t\in [0,T]}|u_{n}(x,y)-u_{n}(t,X_{t}^{\varepsilon},Y_{t}^{\varepsilon})|^{p}\right)\right] +\mathbb{E}\left( \int_{0}^{T}|\mathcal{L}_{1}(s,x,y)u_{n}(s,X_{s}^{\varepsilon},Y_{s}^{\varepsilon})|^{p}ds\right)\nonumber	\\
	&
	+\mathbb{E}\left( \sup_{t\in [0,T]}|M_{t,n}^{1,\varepsilon}|^{p}\right)+\mathbb{E}\left( \sup_{t\in [0,T]}|M_{t,n}^{2,\varepsilon}|^{p}\right) + \mathbb{E}\left( \int_{0}^{T}|\partial_{s}u_{n}(s,X_{s}^{\varepsilon},Y_{s}^{\varepsilon})| ^{p}ds
	\right) \nonumber\\
	&		+\mathbb{E}\left( \int_{0}^{T}|\mathcal{L}_{3}(s,x,y)u_{n}(s,X_{s}^{\varepsilon},Y_{s}^{\varepsilon})- \bar{H}(s,X_{s}^{\varepsilon})|^{p}ds	+\mathbb{E}\left( \int_{0}^{T}|\mathcal{L}_{4}(x,y)u_{n}(s,X_{s}^{\varepsilon},Y_{s}^{\varepsilon})- \bar{c}(s,X_{s}^{\varepsilon})|^{p}ds\right)\right)\nonumber\\
	&=I_{0}+ I_{1}+I_{2}+I_{3}+I_{4}+I_{5}+I_{6}+I_{7},\label{5.77}
\end{align}
so that 
\begin{equation}\label{5.78}
	\begin{aligned}
		I_{0}+I_{1}+I_{2}+I_{3}+I_{4}+I_{5}&\leq C_{T,p}\left( \frac{n^{-vp}}{\gamma_{\varepsilon}^{p}}+\frac{\eta_{\varepsilon}^{p}}{\gamma_{\varepsilon}^{p}}n^{p(\alpha_{1}-v)}+ \left( \frac{\eta_{\varepsilon}^{1-\frac{1}{\alpha_{2}}}}{\gamma_{\varepsilon} } \right) ^{p}\right),\\
	\end{aligned}
\end{equation}
and from \eqref{5.65} and \eqref{5.72},
\begin{equation}\label{5.79}
	\begin{split}
		I_{6}+I_{7}\leq  C_{T,p}\left(\eta_{\varepsilon}^{\frac{pv}{\alpha_{2}}}+\eta_{\varepsilon}^{p(1-\frac{1\vee(\alpha_{1}-v)}{\alpha_{2}})}+\left( \frac{\eta_{\varepsilon}}{\beta_{\varepsilon}}\right)^{p} + \left( \frac{\eta_{\varepsilon}^{1-\frac{1-(1\wedge v)}{\alpha_{2}}}}{\gamma_{\varepsilon}}\right)^{p}\right),
	\end{split}
\end{equation}
then $\frac{\eta_{\varepsilon}^{1-\frac{1}{\alpha_{2}}}}{\gamma_{\varepsilon}}=\gamma_{\varepsilon}^{1-\frac{2}{\alpha_{2}}}$ in \eqref{5.78} leads to contradictions again.
\end{remark}

\section{Weak convergence estimates for \ref{1.1} }

\subsection{Nonlocal Poisson equation for \eqref{1.1} in weak convergence}

Firstly we consider the following Kolmogorov equation
\begin{equation}\label{6.10}
	\left\{
	\begin{aligned}
		&\partial_{t}u(t,x)=	-(-\Delta_{x})^{\frac{\alpha_{1}}{2}}u(t,x)+( \bar{b}(t,x),\nabla_{x}u(t,x)) ,\ t\in[0,T],\\
		&u(0,x)=\phi(x),
	\end{aligned}
	\right.
\end{equation}
here we assume that  $\phi(x)\in C^{2+\gamma}_{b}(\mathbb{R}^{d_{1}})$,  $ \bar{b}(t,x)=\int_{\mathbb{R}^{d_{2}}}b(t,x,y)\mu^{x}(dy),$ $\mathcal{\bar{L}}$ can be regarded as the infinitesimal generator of transition semigroup associated with the averaged process $\bar{X_{t}}$, which takes the form as $d\bar{X_{t}}=\bar{b}(t,\bar{X}_{t})dt+dL_{t}^{1}$. By classical parabolic PDE theory, there exists a unique solution 
\begin{equation}\label{6.11}
 u(t,x)=\mathbb{E}\phi(\bar{X}_{t}(x)),\ t\in[0,T],
\end{equation}
so that $u(t,\cdot)\in C_{b}^{2+\gamma}(\mathbb{R}^{d_{1}})$, $\nabla_{x}u(t,\cdot)\in C_{b}^{1+\gamma}(\mathbb{R}^{d_{1}})$, $\nabla_{x}u(\cdot,x)\in C^{1}([0,T])$, and $\exists C_{T}>0$ s.t.,
\begin{equation}\label{6.12}
 \sup_{t\in [0,T]}\Arrowvert u(t,\cdot)\Arrowvert_{C_{b}^{2+\gamma}(\mathbb{R}^{d_{1}})}\leq C_{T},\  \sup_{t\in [0,T]}\Arrowvert \nabla_{x}u(t,\cdot)\Arrowvert_{C_{b}^{1+\gamma}(\mathbb{R}^{d_{1}})}\leq C_{T},\  \sup_{t\in [0,T]}\Arrowvert \partial_{t}(\nabla_{x}u(\cdot,x))\Arrowvert\leq C_{T}.
\end{equation}

For any fixed $t>0$, let $\hat{u}_{t}(s,x)=u(t-s,x),\ s\in [0,t]$, by It\^{o} formula, 
\begin{equation}\label{6.13}
		\begin{split}
&\hat{u}_{t}(t,X_{t}^{\varepsilon})=\hat{u}_{t}(0,x)+\int_{0}^{t}\partial_{s}\hat{u}_{t}(s,X_{s}^{\varepsilon})ds+\int_{0}^{t}\mathcal{L}_{1}\hat{u}_{t}(s,X_{s}^{\varepsilon})ds+\frac{1}{\gamma_{\varepsilon}}\int_{0}^{t}\mathcal{L}_{3}\hat{u}_{t}(s,X_{s}^{\varepsilon})ds+\hat{M}^{1}_{t},
	\end{split}
\end{equation}
where $$ \hat{M}^{1}_{t}=\int_{0}^{t}\int_{\mathbb{R}^{d_{1}}}\left(  \hat{u}_{t}(s,X_{s^{-}}^{\varepsilon}+x)-\hat{u}_{t}(s,X_{s^{-}}^{\varepsilon})\right) \tilde{N}^{1}(ds,dx), $$
observe that $ \mathbb{E}\hat{M}^{1}_{t}=0,$ $\hat{u}_{t}(t,X_{t}^{\varepsilon})=u(0,X_{t}^{\varepsilon})=\phi(X_{t}^{\varepsilon}),$ $\hat{u}_{t}(0,x)=u(t,x)=\mathbb{E}\phi(\bar{X}_{t}(x)),$ and
\begin{equation}\nonumber
	\begin{split}
		\partial_{s}\hat{u}_{t}(s,X_{s}^{\varepsilon})&=\partial_{s}u(t-s,X_{s}^{\varepsilon})=-\mathcal{\bar{L}}u_{t}(s,X_{s}^{\varepsilon})=(-\Delta_{x})^{\frac{\alpha_{1}}{2}}\hat{u}_{t}(s,X_{s}^{\varepsilon})-( \bar{b}(s,X_{s}^{\varepsilon}),\nabla_{x}\hat{u}_{t}(s,X_{s}^{\varepsilon})),
	\end{split}
	\end{equation}
then we get from \eqref{6.13},
\begin{equation}\label{6.14}
	\begin{split}
	\mathbb{E}\phi(X_{t}^{\varepsilon})-\mathbb{E}\phi(\bar{X}_{t})
	&= \mathbb{E}\int_{0}^{t}-\mathcal{\bar{L}}\hat{u}_{t}(s,X_{s}^{\varepsilon})+\mathcal{L}_{1}\hat{u}_{t}(s,X_{s}^{\varepsilon})ds+\mathbb{E}\int_{0}^{t}\frac{1}{\gamma_{\varepsilon}}\mathcal{L}_{3}\hat{u}_{t}(s,X_{s}^{\varepsilon})ds\\
	&= \mathbb{E}\int_{0}^{t}(b(s,X_{s}^{\varepsilon},Y_{s}^{\varepsilon})- \bar{b}(s,X_{s}^{\varepsilon}),\nabla_{x}\hat{u}_{t}(s,X_{s}^{\varepsilon} ))ds+\mathbb{E}\int_{0}^{t}\frac{1}{\gamma_{\varepsilon}}\mathcal{L}_{3}\hat{u}_{t}(s,X_{s}^{\varepsilon})ds,
	\end{split}
\end{equation}
$\forall s\in [0,T]$, $x\in \mathbb{R}^{d_{1}} $, define 
\begin{equation}\label{6.15}
\check{b}_{t}(s,x,y)=( b(s,x,y),\nabla_{x}\hat{u}_{t}(s,x)),
\end{equation}
so that $\bar{\check{b}}_{t}(s,x)=\int_{\mathbb{R}^{d_{2}}}\check{b}_{t}(s,x,y)\mu^{x}(dy)=(\bar{b}_{t}(s,x),\nabla_{x}\hat{u}_{t}(s,x))$, let $b(t,x,y)\in C_{b}^{\frac{v}{\alpha_{1}},1+\gamma,2+\gamma}$, then $\bar{b}(t,x)\in C_{b}^{\frac{v}{\alpha_{1}},1+\gamma} $, with the boundedness of $b(s,x,y)$, and $\hat{u}_{t}(s,x)\in  C^{1,2+\gamma}_{b}$, we have
$ \check{b}_{t}(s,x,y),\bar{\check{b}}_{t}(s,x)\in C_{b}^{\frac{v}{\alpha_{1}},1+\gamma,2+\gamma},$ and we can see that $$\int_{\mathbb{R}^{ d_{2}}}(\check{b}_{t}(s,x,y)- \bar{\check{b}}_{t}(s,x))\mu^{x}(dy)=\int_{\mathbb{R}^{ d_{2}}}(b(t,x,Y_{s}^{x,y})-\bar{b}(t,x),\nabla_{x}\hat{u}_{t}(s,x))\mu^{x}(dy)=0,$$ which means that $\check{b}_{t}(s,x,y)- \bar{\check{b}}_{t}(s,x)$ satisfies the Centering condition.

We next construct the nonlocal Poisson equation as   ``corrector equation" by \eqref{6.14},
\begin{equation}\label{6.19}
	\begin{split}
			\mathcal{L}_{2}\Phi(t,x,y)+ \check{b}_{t}(s,x,y)-\bar{\check{b}}_{t}(s,x)=\mathcal{L}_{2}\Phi(t,x,y)+ (b(t,x,y)-\bar{b}(t,x),\nabla_{x}\hat{u}_{t}(s,x))=0,
	\end{split}
\end{equation}
here 
\begin{equation}\label{6.18}
	\mathcal{L}_{2}\Phi(t,x,y)=-(-\Delta_{y})^{\frac{\alpha_{2}}{2}}\Phi(t,x,y)+f(x,y)\nabla_{y}\Phi(t,x,y),
\end{equation}
and \eqref{6.19} is to eliminate the difference between drifts.  We give some regularity estimates of $\Phi(t,x,y)$.

\begin{theorem}\label{T61}
	For any  initial point $x\in\mathbb{R}^{d_{1}},\ y\in\mathbb{R}^{d_{2}}$, $ b(t,x,y)\in C_{b}^{\frac{v}{\alpha_{1}},1+\gamma,2+\gamma}$, we define 
	\begin{equation}\label{6.21}
		\begin{split}
			\Phi(t,x,y)=\int^{\infty}_{0}\mathbb{E}\left[ \check{b}_{t}(s,x,Y_{s}^{x,y})-\bar{\check{b}}_{t}(s,x)\right]ds,
		\end{split}
	\end{equation}
	then \eqref{6.21} is a solution of \eqref{6.19},  $\forall T>0$, $t\in [0,T]$,  $	\Phi(t,\cdot,y)\in C^{1+\gamma}_{b}(\mathbb{R}^{d_{1}})$, $	\Phi(t,x,\cdot)\in C_{b}^{2+\gamma}(\mathbb{R}^{ d_{2}})$,  $\exists C_{T}>0$ s.t.,
	\begin{equation}\label{6.23}
		\sup_{t\in [0,T]} \sup_{x\in \mathbb{R}^{d_{1}}}	 |\Phi(t,x,y)| \leq C_{T}(1+|y|),
	\end{equation}
	\begin{equation}\label{6.24}
		\sup_{t\in [0,T]} 	\sup_{x\in \mathbb{R}^{d_{1}}}[|\nabla_{x}\Phi(t,x,y)|+|\nabla_{y}\Phi(t,x,y)|]\leq C_{T}(1+|y|),
	\end{equation}
	\begin{equation}\label{6.25}
		\sup_{t\in [0,T]} |\nabla_{x}\Phi(t,x_{1},y)-\nabla_{x}\Phi(t,x_{2},y)|\leq C_{T}|x_{1}-x_{2}|^{\gamma}(1+|x_{1}-x_{2}|^{1-\gamma})(1+|y|),
	\end{equation}
		here $\gamma\in (\alpha_{1}-1,1)$.
\end{theorem}
\begin{proof}
	Our proof mainly refers on Theorem \ref{T51} and \cite[Proposition 3.3]{SXX} . By It\^{o} formula, \eqref{6.21} is a solution of \eqref{6.19}, and we have  $	\Phi(t,\cdot,y)\in C^{1+\gamma}_{b}(\mathbb{R}^{d_{1}})$, $	\Phi(t,x,\cdot)\in C_{b}^{2+\gamma}(\mathbb{R}^{ d_{2}})$ can be induced from the regularity of $b(t,x,y)$, other properties follow from \cite{SXX}.
\end{proof}

\begin{remark}\label{R61}
When we consider $\alpha_{2}=2$, $\eta_{\varepsilon}=\varepsilon$ and attempt to take diffusive scaling to dreive the averaged equation as It\^{o} process, the corrector equation forms
\begin{equation}\label{6.R1}
	\begin{split}
		\mathcal{L}_{2}\Phi(t,x,y)-(-\Delta_{x})^{\frac{\alpha_{1}}{2}}\hat{u}(t,x)-\Delta_{x}\hat{u}(t,x)=0,
	\end{split}
\end{equation}
 unfortunately, direct computation demonstrates that the integral $$\int_{\mathbb{R}^{ d_{2}}} \left[  -(-\Delta_{x})^{\frac{\alpha_{1}}{2}}\hat{u}(t,x)-\Delta_{x}\hat{u}(t,x)\right]  \mu^{x}(dy)\neq0,$$ thereby violating the critical Centering condition, which precludes the existence and local boundedness of the solution in probabilistic representation
 	\begin{equation}\label{6.R2}
 	\begin{split}
 		\Phi(t,x,y)=\int^{\infty}_{0}\mathbb{E}\left[ -(-\Delta_{x})^{\frac{\alpha_{1}}{2}}\hat{u}(t,x)-\Delta_{x}\hat{u}(t,x)\right]ds.
 	\end{split}
 \end{equation}
 Actually, averaging principle of multiscale systems can also be viewed as the variant of functional law of large numbers \cite{RX2}, i.e., for the Markov process $Y_{t}\in\R^d$,
 \begin{align*}
 	\lim_{\varepsilon\rightarrow0}\varepsilon\int_{0}^{\frac{t}{\varepsilon}}f(x,Y_{s})ds=\lim_{\varepsilon\rightarrow0}\int_{0}^{t}f(x,Y_{\frac{s}{\varepsilon}})ds=\int_{\R^d} f(x,y)\mu(dy)=\bar{f}(x),
 \end{align*}
 here $\mu$ is the invariant measure for the transition
 semigroup of $Y_{t}$ and is independent of $x$, so we can deduce that the original $\al$-stable process cannot be replaced by Brownian motion.
\end{remark}

\subsection{LLN type estimate for $b(t,x,y)$ in weak convergence}

Our method mainly follows from Section 5.2,  here we define  $\bar{b}(t,x)=\int_{\mathbb{R}^{ d_{2}}}b(t,x,y)\mu^{x}(dy),$ consider the following nonlocal Poisson equation, 
\begin{equation}\label{6.19-1}
	\begin{split}
		\mathcal{L}_{2}\Phi(t,x,y)+( b(t,x,y)-\bar{b}(t,x))=0.
	\end{split}
\end{equation}
\begin{remark}
	From definition of $\phi(x)\in C^{2+\gamma}_{b}(\mathbb{R}^{d_{1}})$ in \eqref{6.10}, with analysis of \eqref{6.12} and \eqref{6.15}, regularities of $b(t,x,y)$ with respect to $t$ and $x$ are essential to our analysis, so we let $b(\cdot,\cdot,y)\in C_{b}^{1,1+\gamma}(\mathbb{R}^{1+d_{1}})$.
\end{remark}
Let $b(\cdot,\cdot,y)\in C_{b}^{1,1+\gamma}(\mathbb{R}^{1+d_{1}})$ satisfies Lipschitz condition, growth condition, dissipative condition, $ \bar{b}(t,x)=\int_{\mathbb{R}^{d_{2}}}b(t,x,y)\mu^{x}(dy),$
then we have the following theorem similar to Theorem  \ref{T51-1}.
\begin{theorem}\label{T61-1}
	$\forall x\in \mathbb{R}^{d_{1}}, $  $y\in \mathbb{R}^{d_{2}} $, and $t\in [0,T]$, 
	 $b(\cdot,\cdot,\cdot)\in C_{b}^{1, 1+\gamma,2}$, $\gamma\in (0,1)$ we define 
	\begin{equation}\label{6.2-1}
		\begin{split}
		\Phi(t,x,y)=\int^{\infty}_{0}\left( \mathbb{E}b(t,x,Y_{s}^{x,y})-\bar{b}(t,x)\right) ds,
		\end{split}
	\end{equation}
	then $\Phi(t,x,y)$ is a solution of \eqref{6.19-1} 
	$\exists C>0$ s.t.,
	\begin{equation}\label{6.3-1}
		\sup_{t\in [0,T]}\sup_{x\in \mathbb{R}^{d_{1}}}|\Phi(t,x,y)|\leq C_{T}(1+|y|),
	\end{equation}
	\begin{equation}\label{6.4-1}
		\sup_{t\in [0,T]}\sup_{\substack{x\in \mathbb{R}^{d_{1}}\\ y\in \mathbb{R}^{d_{2}}}}|\nabla_{y}\Phi(t,x,y)|\leq C_{T},
	\end{equation}
\end{theorem}
\begin{proof}
The proof is analogous to Theorem \ref{T51-1}.
\end{proof}

\begin{theorem}\label{T62}
	Suppose that $b(t,x,y)\in C_{p}^{\frac{v}{\alpha_{1}},v,2+\gamma}\cap C_{b}^{1,1+\gamma,2}$, $v\in((\alpha_{1}-\alpha_{2})^{+},\alpha_{1}]$,  $\gamma\in (0,1)$ satisfies Lipschitz condition, growth condition, dissipative condition,  then we have 
	\begin{equation}\label{6.31}
		\sup_{t\in [0,T]}\mathbb{E}\int_{0}^{t}\left( b(s,X_{s}^{\varepsilon},Y_{s}^{\varepsilon})-\bar{b}(s,X_{s}^{\varepsilon})\right)ds\leq C_{T,x,y}\cdot \left(\eta_{\varepsilon}^{\frac{v}{\alpha_{2}}}+\eta_{\varepsilon}^{1-\frac{\alpha_{1}-v}{\alpha_{2}}}+\frac{\eta_{\varepsilon}^{1-\frac{1-(1\wedge v)}{\alpha_{2}}}}{\gamma_{\varepsilon}}+\frac{\eta_{\varepsilon}}{\beta_{\varepsilon}}\right).
	\end{equation}
\end{theorem}
\begin{proof}
Similar to Theorem \ref{T52}, this proof also refers to \cite[Lemma 4.2]{RX}. Let $\Phi^{n}$ be the mollifyer of $\Phi$, which is the solution of \eqref{6.19-1}, since here we let $b(\cdot,\cdot,\cdot)\in C_{p}^{\frac{v}{\alpha_{1}},v,2+\gamma}\cap C_{b}^{1,1+\gamma,2}$, from PDE theory and Theorem \ref{T51-1}, we have $\Phi(\cdot,\cdot,\cdot)\in C_{p}^{\frac{v}{\alpha_{1}},v,2+\gamma}\cap C_{b}^{1,1+\gamma,2}$. Similar to \eqref{5.46}, after applying It\^{o} formula, taking expectation and utilizing the martingale property $\mathbb{E}M_{n,t}^{1,\varepsilon}=\mathbb{E}M_{n,t}^{2,\varepsilon}=0,$ we have
\begin{equation}\label{6.32}
	\begin{split}
		&\mathbb{E}\Phi^{n}(t,X_{t}^{\varepsilon},Y_{t}^{\varepsilon})=\Phi^{n}(0,x,y)+\mathbb{E}\int_{0}^{t}\partial_{s}\Phi^{n}(s,X_{s}^{\varepsilon},Y_{s}^{\varepsilon})ds+\mathbb{E}\int_{0}^{t} \mathcal{L}_{1}\Phi^{n}(s,X_{S}^{\varepsilon},Y_{s}^{\varepsilon})ds\\
	&+\frac{1}{\eta_{\varepsilon}}\left[ \mathbb{E}\int_{0}^{t}\mathcal{L}_{2}\Phi^{n}(s,X_{s}^{\varepsilon},Y_{s}^{\varepsilon})ds\right]+\frac{1}{\gamma_{\varepsilon}}\left[ \mathbb{E}\int_{0}^{t}\mathcal{L}_{3}\Phi^{n}(s,X_{s}^{\varepsilon},Y_{s}^{\varepsilon})ds\right]+\frac{1}{\beta_{\varepsilon}}\left[ \mathbb{E}\int_{0}^{t}\mathcal{L}_{4}\Phi^{n}(s,X_{s}^{\varepsilon},Y_{s}^{\varepsilon})ds\right],
	\end{split}
\end{equation}
then we have
\begin{equation}\label{6.33}
	\begin{split}
			&-\frac{1}{\eta_{\varepsilon}} \mathbb{E}\int_{0}^{t}\mathcal{L}_{2}\Phi^{n}(s,X_{s}^{\varepsilon},Y_{s}^{\varepsilon})ds	=\Phi^{n}(0,x,y)-\mathbb{E}\Phi^{n}(s,X_{t}^{\varepsilon},Y_{t}^{\varepsilon})+\mathbb{E}\int_{0}^{t}\partial_{s}\Phi^{n}(s,X_{s}^{\varepsilon},Y_{s}^{\varepsilon})ds\\
			&+\mathbb{E}\int_{0}^{t} \mathcal{L}_{1}\Phi^{n}(s,X_{s}^{\varepsilon},Y_{s}^{\varepsilon})ds+\frac{1}{\gamma_{\varepsilon}}\left[ \mathbb{E}\int_{0}^{t}\mathcal{L}_{3}\Phi^{n}(s,X_{s}^{\varepsilon},Y_{s}^{\varepsilon})ds\right]+\frac{1}{\beta_{\varepsilon}}\left[ \mathbb{E}\int_{0}^{t}\mathcal{L}_{4}\Phi^{n}(s,X_{s}^{\varepsilon},Y_{s}^{\varepsilon})ds\right],
	\end{split}
\end{equation}
from \eqref{6.19-1}, 
\begin{equation}\label{6.34}
	\begin{split}
		&	\sup_{t\in [0,T]}\mathbb{E}\int_{0}^{t}\left( b(s,X_{s}^{\varepsilon},Y_{s}^{\varepsilon})-\bar{b}(s,X_{s}^{\varepsilon})\right)ds\\
		&
	\leq\mathbb{E} \int_{0}^{T}\left|\mathcal{L}_{2}\Phi^{n}(s,X_{s}^{\varepsilon},Y_{s}^{\varepsilon})-\mathcal{L}_{2}\Phi(s,X_{s}^{\varepsilon},Y_{s}^{\varepsilon})\right|ds\\
	&+\eta_{\varepsilon}\sup_{t\in [0,T]}\left[ \mathbb{E}|\Phi^{n}(0,x,y)|+\mathbb{E}|\Phi^{n}(t,X_{t}^{\varepsilon},Y_{t}^{\varepsilon})|\right]  +\mathbb{E}\int_{0}^{T}|\partial_{s}\Phi^{n}(s,X_{s}^{\varepsilon},Y_{s}^{\varepsilon})|ds\\
	&+\mathbb{E}\int_{0}^{T} |\mathcal{L}_{1}\Phi^{n}(s,X_{s}^{\varepsilon},Y_{s}^{\varepsilon})|ds+\frac{1}{\gamma_{\varepsilon}} \mathbb{E}\int_{0}^{T}|\mathcal{L}_{3}\Phi_{t}^{n}(s,X_{s}^{\varepsilon},Y_{s}^{\varepsilon})|ds+\frac{1}{\beta_{\varepsilon}} \mathbb{E}\int_{0}^{T}|\mathcal{L}_{4}\Phi^{n}(s,X_{s}^{\varepsilon},Y_{s}^{\varepsilon})|ds	\\
	&=I_{1}+I_{2}+I_{3}+I_{4}+I_{5}+I_{6},
	\end{split}
\end{equation}
specially, by \eqref{6.3-1} in Theorem \ref{T61-1}, and \eqref{3.6},  we estimate $I_{2}$ here,
\begin{equation}
	\begin{split}
	\eta_{\varepsilon}\sup_{t\in [0,T]}\left[  \mathbb{E}|\Phi^{n}(0,x,y)|+\mathbb{E}|\Phi^{n}(t,X_{t}^{\varepsilon},Z_{t}^{\varepsilon})|\right]  &\leq \eta_{\varepsilon}C_{T}\sup_{t\in [0,T]}\left[  \mathbb{E}|\Phi(0,x,y)|+\mathbb{E}|\Phi(t,X_{t}^{\varepsilon},Y_{t}^{\varepsilon})|\right] \\
	&\leq \eta_{\varepsilon}C_{T}\sup_{t\in [0,T]}\mathbb{E}(1+|y|+|Y_{t}^{\varepsilon}|)\leq \eta_{\varepsilon}C_{T}(1+|y|),
	\end{split}\nonumber
\end{equation}
set $n=\eta_{\varepsilon}^{-\frac{1}{\alpha_{2}}}$, take similar precedure in the proof of Theorem \ref{T52},  we obtain
\begin{equation}\label{6.35}
	\begin{split}
	\sup_{t\in [0,T]}\mathbb{E}\int_{0}^{t}\left( b(s,X_{s}^{\varepsilon},Y_{s}^{\varepsilon})-\bar{b}(s,X_{s}^{\varepsilon})\right)ds&\leq C_{T,x,y}\cdot \left(\eta_{\varepsilon}^{\frac{v}{\alpha_{2}}}+\eta_{\varepsilon}^{1-\frac{\alpha_{1}-v}{\alpha_{2}}}+\frac{\eta_{\varepsilon}^{1-\frac{1-(1\wedge v)}{\alpha_{2}}}}{\gamma_{\varepsilon}}+\frac{\eta_{\varepsilon}}{\beta_{\varepsilon}}\right),
	\end{split}
\end{equation}
proof is complete.
\end{proof}

\subsection{CLT type estimate for $\frac{1}{\gamma_{\varepsilon}}H(t,X_{t}^{\varepsilon},Y_{t}^{\varepsilon})$  in weak convergence}

We recall that  $ H(t,x,y)$ satisfies Centering condition, then  $$\int_{\mathbb{R}^{ d_{2}}}H(t,x,y)\mu^{x}(dy)=0,$$
here $\mu^{x}$ is the invariant measure of \eqref{4.1}, and define
\begin{equation}
	\bar{c}(t,x)=\int_{\mathbb{R}^{ d_{2}}}c(x,y)\nabla_{y}\Phi(t,x,y)\mu^{x}(dy),
	\nonumber
\end{equation}
\begin{equation}
	\bar{H}(t,x)=\int_{\mathbb{R}^{ d_{2}}}H(t,x,y)\nabla_{x}\Phi(t,x,y)\mu^{x}(dy),
	\nonumber
\end{equation}
$\Phi(t,x,y)$ is the solution of following equation
\begin{equation}\label{6.36-1}
	\mathcal{L}_{2}(t,x,y)\Phi(t,x,y)+H(t,x,y)=0.
	\end{equation}

\begin{theorem}\label{T63}
	Suppose that  Lipschitz condition, growth condition, dissipative condition valid,   then we have 
	
	Regime 1:  $H(t,x,y)\in C_{p}^{\frac{v}{\alpha_{1}},v,2+\gamma}\cap C_{b}^{1,1+\gamma,2}$, $v\in((\alpha_{1}-\alpha_{2})^{+},\alpha_{1}]$,  $\gamma\in (0,1)$, $ \lim\limits_{\varepsilon \rightarrow0}\frac{ \eta_{\varepsilon}^{\left[ \frac{v}{\alpha_{2}}\wedge \left( 1-\frac{\alpha_{1}-v}{\alpha_{2}} \right)  \right]  }}{\gamma_{\varepsilon}}=0$, and  $b(t,x,y)\in C_{p}^{\frac{v}{\alpha_{1}},v,2+\gamma}\cap C_{b}^{1,1+\gamma,2}$,
	\begin{equation}\label{6.37}
	\sup_{t\in [0,T]}\mathbb{E}\int_{0}^{t}\left(\frac{1}{\gamma_{\varepsilon}} H(s,X_{s}^{\varepsilon},Y_{s}^{\varepsilon}) \right)ds\leq C_{T,x,y}\cdot\left( \frac{ \eta_{\varepsilon}^{\left[ \frac{v}{\alpha_{2}}\wedge \left( 1-\frac{\alpha_{1}-v}{\alpha_{2}} \right)  \right]  }}{\gamma_{\varepsilon}}+\frac{ \eta_{\varepsilon}^{1-\frac{1-(1\wedge v)}{\alpha_{2}}}}{\gamma^{2}_{\varepsilon}}+\frac{\eta_{\varepsilon}}{\gamma_{\varepsilon}\beta_{\varepsilon}}\right);
	\end{equation}
	
	Regime 2:  $H(t,x,y)\in C_{p}^{\frac{v}{\alpha_{1}},v,2+\gamma}\cap C_{b}^{1,1+\gamma,2}$, $v\in((\alpha_{1}-\alpha_{2})^{+},\alpha_{1}]$,  $\gamma\in (0,1)$, $ \lim\limits_{\varepsilon \rightarrow0}\frac{ \eta_{\varepsilon}^{\left[ \frac{v}{\alpha_{2}}\wedge \left( 1-\frac{\alpha_{1}-v}{\alpha_{2}} \right)  \right]  }}{\gamma_{\varepsilon}}=0$, let $b(t,x,y)\in C_{p}^{\frac{v}{\alpha_{1}},v,2+\gamma}\cap C_{b}^{1,1+\gamma,2}$,
	\begin{equation}\label{6.38}
		\sup_{t\in [0,T]}\mathbb{E} \int_{0}^{t}\left(  \frac{1}{\gamma_{\varepsilon}}H(s,X_{s}^{\varepsilon},Y_{s}^{\varepsilon})- \bar{c}(s,X_{s}^{\varepsilon})\right) ds  \leq  C_{T,x,y}\left(\frac{ \eta_{\varepsilon}^{\left[ \frac{v}{\alpha_{2}}\wedge \left( 1-\frac{\alpha_{1}-v}{\alpha_{2}} \right)  \right]  }}{\gamma_{\varepsilon}}+\frac{ \eta_{\varepsilon}^{1-\frac{1-(1\wedge v)}{\alpha_{2}}}}{\gamma^{2}_{\varepsilon}}+\gamma_{\varepsilon}\right) ;
	\end{equation}
	
	Regime 3:  $H(t,x,y)\in C_{p}^{\frac{v}{\alpha_{1}},2+\gamma,2+\gamma}\cap C_{b}^{1,2+\gamma,2+\gamma}$, $v\in ( \frac{\alpha_{2}}{2}\vee \frac{2\alpha_{1}-\alpha_{2}}{2}  ,\alpha_{1}]$, $\gamma\in(\alpha_{1}-1,1)$, and specially we have $b(\cdot,\cdot,\cdot)\in C_{p}^{\frac{v}{\alpha_{1}},v,2+\gamma}\cap C_{b}^{1,1+\gamma,2+\gamma}$,
	\begin{equation}\label{6.39}
\sup_{t\in [0,T]}	\mathbb{E} \int_{0}^{t}\left(  \frac{1}{\gamma_{\varepsilon}}H(s,X_{s}^{\varepsilon},Y_{s}^{\varepsilon})- \bar{H}(s,X_{s}^{\varepsilon})\right) ds  \leq  C_{T,x,y}\left( \gamma_{\varepsilon}^{\frac{2v}{\alpha_{2}}-\left[ 1\vee \left( \frac{2\alpha_{1}}{\alpha_{2}}-1\right)   \right] }+\frac{\gamma_{\varepsilon}}{\beta_{\varepsilon}}\right) ;
	\end{equation}
	
	Regime 4:  $H(t,x,y)\in C_{p}^{\frac{v}{\alpha_{1}},2+\gamma,2+\gamma}\cap C_{b}^{1,2+\gamma,2+\gamma}$,  $v\in ( \frac{\alpha_{2}}{2}\vee \frac{2\alpha_{1}-\alpha_{2}}{2}  ,\alpha_{1}]$, $\gamma\in(\alpha_{1}-1,1)$, here we have $b(\cdot,\cdot,\cdot)\in C_{p}^{\frac{v}{\alpha_{1}},v,2+\gamma}\cap C_{b}^{1,1+\gamma,2+\gamma}$,
	\begin{equation}\label{6.40}
		\begin{split}
			 \sup_{t\in [0,T]}\mathbb{E} \int_{0}^{t}\left(  \frac{1}{\gamma_{\varepsilon}}H(s,X_{s}^{\varepsilon},Y_{s}^{\varepsilon})- \bar{c}(s,X_{s}^{\varepsilon})-\bar{H}(s,X_{s}^{\varepsilon})\right)ds\leq C_{T,x,y}\cdot \gamma_{\varepsilon}^{\frac{2v}{\alpha_{2}}-\left[ 1\vee \left( \frac{2\alpha_{1}}{\alpha_{2}}-1\right)   \right] }.
		\end{split}
	\end{equation}
\end{theorem}
\begin{proof}
Similar to Theorem \ref{T53}, our method refers to \cite[Lemma 4.4]{RX}. 
For Regime 1, as  $H(t,x,y)\in C_{p}^{\frac{v}{\alpha_{1}},v,2+\gamma}\cap C_{b}^{1,1+\gamma,2}$ satisfies Centering condition, from Theorem \ref{T62}
\begin{equation}\label{6.41}
	\begin{split}
		\sup_{t\in [0,T]}\mathbb{E}\int_{0}^{t}\left(\frac{1}{\gamma_{\varepsilon}} H(s,X_{s}^{\varepsilon},Y_{s}^{\varepsilon}) \right)ds&\leq C_{T,x,y}\cdot \left( \frac{ \eta_{\varepsilon}^{\left[ \frac{v}{\alpha_{2}}\wedge \left( 1-\frac{\alpha_{1}-v}{\alpha_{2}} \right)  \right]  }}{\gamma_{\varepsilon}}+\frac{ \eta_{\varepsilon}^{1-\frac{1-(1\wedge v)}{\alpha_{2}}}}{\gamma^{2}_{\varepsilon}}+\frac{\eta_{\varepsilon}}{\gamma_{\varepsilon}\beta_{\varepsilon}}\right) .
	\end{split}
\end{equation}

For Regime 2, let $\Phi^{n}$ be the mollifyer of $\Phi$, which is the solution of \eqref{6.36-1}, then by It\^{o} formula,
\begin{equation}\label{6.42}
	\begin{split}
		&\Phi^{n}(t,X_{t}^{\varepsilon},Y_{t}^{\varepsilon})
		=\Phi^{n}(x,y)+\int_{0}^{t}\partial_{s}\Phi_{t}^{n}(s,X_{s}^{\varepsilon},Y_{s}^{\varepsilon})ds+\int_{0}^{t}\mathcal{L}_{1}(s,x,y)\Phi^{n}(s,X_{s}^{\varepsilon},Y_{s}^{\varepsilon})ds\\
		&+\frac{1}{\eta_{\varepsilon}}\int_{0}^{t}\mathcal{L}_{2}(x,y)\Phi^{n}(s,X_{s}^{\varepsilon},Y_{s}^{\varepsilon})ds+\frac{1}{\gamma_{\varepsilon}}\int_{0}^{t}\mathcal{L}_{3}(s,x,y)\Phi^{n}(s,X_{s}^{\varepsilon},Y_{s}^{\varepsilon})ds\\
		&+\frac{1}{\beta_{\varepsilon}}\int_{0}^{t}\mathcal{L}_{4}(x,y)\Phi^{n}(s,X_{s}^{\varepsilon},Y_{s}^{\varepsilon})ds,
	\end{split}
\end{equation}
for $ \eta_{\varepsilon}=\gamma_{\varepsilon}\beta_{\varepsilon}$, then we have,
\begin{equation}\label{6.43}
	\begin{split}
	&	\sup_{t\in [0,T]}\mathbb{E}\int_{0}^{t}(\frac{1}{\gamma_{\varepsilon}} H(s,X_{s}^{\varepsilon},Y_{s}^{\varepsilon})- \bar{c}(s,X_{s}^{\varepsilon}))ds\\
		&=\frac{1}{\gamma_{\varepsilon}}	\sup_{t\in [0,T]}\mathbb{E}\int_{0}^{t}\left( \mathcal{L}_{2}(x,y)\Phi^{n}(s,X_{s}^{\varepsilon},Y_{s}^{\varepsilon})-\mathcal{L}_{2}(x,y)\Phi(s,X_{s}^{\varepsilon},Y_{s}^{\varepsilon})\right) ds\\
		&+\frac{\eta_{\varepsilon}}{{\gamma_{\varepsilon}}}	\sup_{t\in [0,T]}\mathbb{E}\left[\left(  \Phi^{n}(x,y)-\Phi^{n}(s,X_{t}^{\varepsilon},Y_{t}^{\varepsilon})\right) +\int_{0}^{t}\partial_{s}\Phi^{n}(s,X_{s}^{\varepsilon},Y_{s}^{\varepsilon})+\mathcal{L}_{1}(s,x,y)\Phi^{n}(s,X_{s}^{\varepsilon},Y_{s}^{\varepsilon})ds\right.	\\
		&\left. +\frac{1}{\gamma_{\varepsilon}}\int_{0}^{t}\mathcal{L}_{3}(s,x,y)\Phi^{n}(s,X_{s}^{\varepsilon},Y_{s}^{\varepsilon})ds\right]+\sup_{t\in [0,T]}\mathbb{E}\int_{0}^{t}(\mathcal{L}_{4}(x,y)\Phi^{n}(s,X_{s}^{\varepsilon},Y_{s}^{\varepsilon})- \bar{c}(s,X_{s}^{\varepsilon}))ds\\
		&=I_{1}+I_{2}+I_{3}+I_{4}+I_{5}+I_{6},
	\end{split}
\end{equation}
thus analogous to the proof of \eqref{5.66} in Theorem \ref{T53}, by Theorem \ref{T61-1},
\begin{equation}\label{6.44}
	\begin{split}
I_{1}+I_{2}+I_{3}+I_{4}+I_{5}\leq C_{T,x,y}\left( \frac{n^{-v}}{\gamma_{\varepsilon}}+\frac{\eta_{\varepsilon}}{\gamma_{\varepsilon}}n^{\alpha_{1}-v}+\frac{\eta_{\varepsilon}}{\gamma_{\varepsilon}^{2}}n^{1-(1\wedge v)}\right) ,
	\end{split}
\end{equation}
in particular,
\begin{equation}\label{6.45}
	\begin{split}
	I_{6}&\leq\mathbb{E}\left( \int^{T}_{0}|c(X_{s}^{\varepsilon},Y_{s}^{\varepsilon})\nabla_{y} \Phi^{n}(s,X_{s}^{\varepsilon},Y_{s}^{\varepsilon})-\bar{c}(s,X_{s}^{\varepsilon})|ds\right) \\
		&\leq\mathbb{E}\left( \int^{T}_{0}|c(X_{s}^{\varepsilon},Y_{s}^{\varepsilon})\nabla_{y} \Phi^{n}(s,X_{s}^{\varepsilon},Y_{s}^{\varepsilon})-c(X_{s}^{\varepsilon},Y_{s}^{\varepsilon})\nabla_{y} \Phi(s,X_{s}^{\varepsilon},Y_{s}^{\varepsilon})|ds\right)\\ 
		&+\mathbb{E}\left( \int^{T}_{0}|c(X_{s}^{\varepsilon},Y_{s}^{\varepsilon})\nabla_{y} \Phi(s,X_{s}^{\varepsilon},Y_{s}^{\varepsilon})-\bar{c}(s,X_{s}^{\varepsilon})|ds\right)=I_{61}+I_{62},
	\end{split}
\end{equation}
similar to proof in Theorem \ref{T53}, using Lemma \ref{L51}, we have 
\begin{equation}\label{6.46}
	\begin{split}
		I_{61}\leq  C_{T,x,y}n^{-v},
	\end{split}
\end{equation}
from  $H(t,x,y)\in C_{p}^{\frac{v}{\alpha_{1}},v,2+\gamma}\cap C_{b}^{1,1+\gamma,2}$, we have $ \Phi \in   C_{p}^{\frac{v}{\alpha_{1}},v,2+\gamma}\cap C_{b}^{1,1+\gamma,2}$, then we have $ c\cdot \nabla_{y} \Phi \in C_{p}^{\frac{v}{\alpha_{1}},v,1+\gamma}\cap C_{b}^{1,1+\gamma,1}$,  and $I_{62} $ satisfies Centering condition, recall $1+\gamma>1>\delta$, by Theorem \ref{T62},
\begin{equation}\label{6.47}
	\begin{split}
	I_{62}\leq  C_{T,x,y} \left(\eta_{\varepsilon}^{\frac{v}{\alpha_{2}}}+\eta_{\varepsilon}^{1-\frac{\alpha_{1}-v}{\alpha_{2}}}+\frac{\eta_{\varepsilon}^{1-\frac{1-(1\wedge v)}{\alpha_{2}}}}{\gamma_{\varepsilon}}+\frac{\eta_{\varepsilon}}{\beta_{\varepsilon}}\right),
	\end{split}
\end{equation}
 take $n=\eta_{\varepsilon}^{-\frac{1}{\alpha_{2}}}$,    finally we get
\begin{equation}\label{6.48}
	\begin{split}
		\sup_{t\in [0,T]}\mathbb{E}  \int_{0}^{t}\left(  \frac{1}{\gamma_{\varepsilon}}H(s,X_{s}^{\varepsilon},Y_{s}^{\varepsilon})- \bar{c}(s,X_{s}^{\varepsilon})\right) ds  \leq C_{T,x,y}\left( \frac{ \eta_{\varepsilon}^{\left[\left(  \frac{v}{\alpha_{2}}\right) \wedge \left( 1-\frac{\alpha_{1}-v}{\alpha_{2}} \right)  \right]  }}{\gamma_{\varepsilon}}+\frac{ \eta_{\varepsilon}^{1-\frac{1-(1\wedge v)}{\alpha_{2}}}}{\gamma^{2}_{\varepsilon}}+\gamma_{\varepsilon}\right) .
	\end{split}
\end{equation}

For Regime 3,  as analysed in Remark \ref{R51}, the term $	\underset{t\in [0,T]}{\sup}\underset{x\in \mathbb{R}^{d_{1}}}{\sup}|\nabla_{x}\Phi(t,x,y)|$ plays a critical role for the control of $H\cdot\nabla_{x} \Phi$, see \eqref{6.51}. This requirement necessitates the application of Theorem \ref{T61} rather than Theorem \ref{T61-1}, leading us to impose the Hölder regularity condition  $H(t,x,y)\in C_{p}^{\frac{v}{\alpha_{1}},2+\gamma,2+\gamma}\cap C_{b}^{1,2+\gamma,2+\gamma}$, moerover we let $b(\cdot,\cdot,\cdot)\in C_{p}^{\frac{v}{\alpha_{1}},v,2+\gamma}\cap C_{b}^{1,1+\gamma,2+\gamma}$ in Theorem \ref{T62},
\begin{equation}\label{6.49}
	\begin{split}
		&	\sup_{t\in [0,T]}\mathbb{E}\int_{0}^{t}(\frac{1}{\gamma_{\varepsilon}} H(s,X_{s}^{\varepsilon},Y_{s}^{\varepsilon})- \bar{H}(s,X_{s}^{\varepsilon}))ds\\
		&=\frac{1}{\gamma_{\varepsilon}}	\sup_{t\in [0,T]}\mathbb{E}\int_{0}^{t}\left( \mathcal{L}_{2}(x,y)\Phi^{n}(s,X_{s}^{\varepsilon},Y_{s}^{\varepsilon})-\mathcal{L}_{2}(x,y)\Phi(s,X_{s}^{\varepsilon},Y_{s}^{\varepsilon})\right) ds\\
		&+\frac{\eta_{\varepsilon}}{{\gamma_{\varepsilon}}}	\sup_{t\in [0,T]}\mathbb{E}\left[ \Phi^{n}(x,y)-\Phi^{n}(s,X_{t}^{\varepsilon},Y_{t}^{\varepsilon}) +\int_{0}^{t}\partial_{s}\Phi^{n}(s,X_{s}^{\varepsilon},Y_{s}^{\varepsilon})+\mathcal{L}_{1}(s,x,y)\Phi^{n}(s,X_{s}^{\varepsilon},Y_{s}^{\varepsilon})ds\right.	\\
		&\left. +\frac{1}{\beta_{\varepsilon}}\int_{0}^{t}\mathcal{L}_{4}(x,y)\Phi^{n}(s,X_{s}^{\varepsilon},Y_{s}^{\varepsilon})ds\right]+	\sup_{t\in [0,T]}\mathbb{E}\int_{0}^{t}(\mathcal{L}_{3}(s,x,y)\Phi^{n}(s,X_{s}^{\varepsilon},Y_{s}^{\varepsilon})- \bar{H}(s,X_{s}^{\varepsilon}))ds\\
		&=I_{1}+I_{2}+I_{3}+I_{4}+I_{5}+I_{6},
	\end{split}
\end{equation}
then 
\begin{equation}\label{6.50}
	\begin{split}
		I_{1}+I_{2}+I_{3}+I_{4}+I_{5}\leq C_{T,x,y}\left( \frac{n^{-v}}{\gamma_{\varepsilon}}+\frac{\eta_{\varepsilon}}{\gamma_{\varepsilon}}n^{\alpha_{1}-v}+\frac{\eta_{\varepsilon}}{\beta_{\varepsilon}\gamma_{\varepsilon}}\right) ,
	\end{split}
\end{equation}
and 
\begin{align}
		I_{6}&\leq\mathbb{E}\left( \int^{T}_{0}|H(s,X_{s}^{\varepsilon},Y_{s}^{\varepsilon})\nabla_{x} \Phi^{n}(s,X_{s}^{\varepsilon},Y_{s}^{\varepsilon})-\bar{H}(s,X_{s}^{\varepsilon})|ds\right)\nonumber \\
		&\leq\mathbb{E}\left( \int^{T}_{0}|H(s,X_{s}^{\varepsilon},Y_{s}^{\varepsilon})\nabla_{x} \Phi^{n}(s,X_{s}^{\varepsilon},Y_{s}^{\varepsilon})-H(s,X_{s}^{\varepsilon},Y_{s}^{\varepsilon})\nabla_{x} \Phi(s,X_{s}^{\varepsilon},Y_{s}^{\varepsilon})|ds\right)\nonumber\\ 
		&+\mathbb{E}\left( \int^{T}_{0}|H(s,X_{s}^{\varepsilon},Y_{s}^{\varepsilon})\nabla_{x} \Phi(s,X_{s}^{\varepsilon},Y_{s}^{\varepsilon})-\bar{H}(s,X_{s}^{\varepsilon})|ds\right)=I_{61}+I_{62},\label{6.51}
\end{align}
since $H(t,x,y)\in C_{p}^{\frac{v}{\alpha_{1}},2+\gamma,2+\gamma}\cap C_{b}^{1,2+\gamma,2+\gamma}$, we have $ H\cdot \nabla_{x} \Phi \in C_{p}^{\frac{v}{\alpha_{1}},1+\gamma,2+\gamma}\cap C_{b}^{1,1+\gamma,2+\gamma}$,  we mention that $1+\gamma>\alpha_{1}\geq v$, by \eqref{6.23} and \eqref{6.24} in Theorem \ref{T61}, with Theorem \ref{T62} where we let $b(\cdot,\cdot,\cdot)\in C_{p}^{\frac{v}{\alpha_{1}},v,2+\gamma}\cap C_{b}^{1,1+\gamma,2+\gamma}$, thus
 $$I_{6}\leq C_{T,x,y}\left( n^{-v}+\eta_{\varepsilon}^{\frac{v}{\alpha_{2}}}+\eta_{\varepsilon}^{1-\frac{\alpha_{1}-v}{\alpha_{2}}}+\frac{\eta_{\varepsilon}^{1-\frac{1-(1\wedge v)}{\alpha_{2}}}}{\gamma_{\varepsilon}}+\frac{\eta_{\varepsilon}}{\beta_{\varepsilon}}\right) ,$$
we notice that $\eta_{\varepsilon}=\gamma^{2}_{\varepsilon}$, then $\frac{ \eta_{\varepsilon}^{1-\frac{\alpha_{1}-v}{\alpha_{2}}}}{\gamma_{\varepsilon}}=\gamma_{\varepsilon}^{1-\frac{2\alpha_{1}-2v}{\alpha_{2}}}$, $\frac{ \eta_{\varepsilon}^{\frac{v}{\alpha_{2}}}}{\gamma_{\varepsilon}}=\gamma_{\varepsilon}^{\frac{2v}{\alpha_{2}}-1}$, so  we get
\begin{equation}\label{6.52}
	\begin{split}
		\sup_{t\in [0,T]}\mathbb{E}  \int_{0}^{t}\left(  \frac{1}{\gamma_{\varepsilon}}H(s,X_{s}^{\varepsilon},Y_{s}^{\varepsilon})- \bar{H}(s,X_{s}^{\varepsilon})\right) ds  \leq C_{T,x,y}\left( \gamma_{\varepsilon}^{\frac{2v}{\alpha_{2}}-\left[ 1\vee \left( \frac{2\alpha_{1}}{\alpha_{2}}-1\right)   \right] }+\frac{\gamma_{\varepsilon}}{\beta_{\varepsilon}}\right) .
	\end{split}
\end{equation}

Finally for Regime 4, let  $H(t,x,y)\in C_{p}^{\frac{v}{\alpha_{1}},2+\gamma,2+\gamma}\cap C_{b}^{1,2+\gamma,2+\gamma}$, and $b(\cdot,\cdot,\cdot)\in C_{p}^{\frac{v}{\alpha_{1}},v,2+\gamma}\cap C_{b}^{1,1+\gamma,2+\gamma}$,
\begin{align*}
		&	\sup_{t\in [0,T]}\mathbb{E}\int_{0}^{t}\Big(\frac{1}{\gamma_{\varepsilon}} H(s,X_{s}^{\varepsilon},Y_{s}^{\varepsilon})- \bar{c}(s,X_{s}^{\varepsilon})- \bar{H}(s,X_{s}^{\varepsilon})\Big)ds\\
		&=\frac{1}{\gamma_{\varepsilon}}	\sup_{t\in [0,T]}\mathbb{E}\int_{0}^{t}\left( \mathcal{L}_{2}(x,y)\Phi^{n}(s,X_{s}^{\varepsilon},Y_{s}^{\varepsilon})-\mathcal{L}_{2}(x,y)\Phi(s,X_{s}^{\varepsilon},Y_{s}^{\varepsilon})\right) ds\\
		&+\frac{\eta_{\varepsilon}}{{\gamma_{\varepsilon}}}	\sup_{t\in [0,T]}\mathbb{E}\left[ (\Phi^{n}(x,y)-\Phi^{n}(s,X_{t}^{\varepsilon},Y_{t}^{\varepsilon}))+\int_{0}^{t}\partial_{s}\Phi^{n}(s,X_{s}^{\varepsilon},Y_{s}^{\varepsilon})ds+\int_{0}^{t}\mathcal{L}_{1}(s,x,y)\Phi^{n}(s,X_{s}^{\varepsilon},Y_{s}^{\varepsilon})ds\right]\\
		& +\sup_{t\in [0,T]}\mathbb{E}\left[ \int_{0}^{t}(\mathcal{L}_{3}(s,x,y)\Phi^{n}(s,X_{s}^{\varepsilon},Y_{s}^{\varepsilon})- \bar{H}(s,X_{s}^{\varepsilon}))ds+\int_{0}^{t}(\mathcal{L}_{4}(x,y)\Phi^{n}(s,X_{s}^{\varepsilon},Y_{s}^{\varepsilon})-\bar{c}(s,X_{s}^{\varepsilon}))ds\right]  \\
		&=I_{1}+I_{2}+I_{3}+I_{4}+I_{5}+I_{6},
\end{align*}
 we have $
		I_{1}+I_{2}+I_{3}+I_{4}\leq C_{T,x,y}(\frac{n^{-v}}{\gamma_{\varepsilon}}+\frac{\eta_{\varepsilon}}{\gamma_{\varepsilon}}n^{\alpha_{1}-v})
$, additionally, we can deduce from \eqref{6.45} and \eqref{6.51}, 
\begin{align*}
I_{5}+I_{6}\leq C_{T,x,y}\left(  \eta_{\varepsilon}^{\frac{v}{\alpha_{2}}}+\eta_{\varepsilon}^{1-\frac{\alpha_{1}-v}{\alpha_{2}}}+\frac{\eta_{\varepsilon}^{1-\frac{1-(1\wedge v)}{\alpha_{2}}}}{\gamma_{\varepsilon}}+\frac{\eta_{\varepsilon}}{\beta_{\varepsilon}}\right),
\end{align*}
combining above estimates, we obtain
\begin{equation}\label{6.55}
	\begin{split}
		 \sup_{t\in [0,T]}\mathbb{E} \int_{0}^{t}\left(  \frac{1}{\gamma_{\varepsilon}}H(X_{s}^{\varepsilon},Y_{s}^{\varepsilon})- \bar{c}(s,X_{s}^{\varepsilon})- \bar{H}(s,X_{s}^{\varepsilon}) ds \right) \leq C_{T,x,y}\cdot \gamma_{\varepsilon}^{\frac{2v}{\alpha_{2}}-\left[ 1\vee \left( \frac{2\alpha_{1}}{\alpha_{2}}-1\right)   \right] },
	\end{split}\nonumber
\end{equation}
proof is complete.
\end{proof}

\section{Statements of main results}

In this section, we present the proofs of \textbf{Theorem  \ref{SCRJ}} and \textbf{Theorem  \ref{WCRD}}. Our methods are inspired by the studies in \cite{CEB} and \cite{SXX}, which are beneficial for quantitative estimates.

\subsection{Proof of \textbf{Theorem \ref{SCRJ}} }
\begin{proof}
	Observe that in Regime 1, we refer to \cite[Section 4.1, 356]{SXX}, then
	\begin{equation}\label{7.1}
		d\bar{X}^{1}_{t}=\bar{b}(t,\bar{X}^{1}_{t})dt+dL_{t}^{1},
	\end{equation}
	so that
	\begin{equation}
		X_{t}^{\varepsilon}-\bar{X}^{1}_{t}=\int_{0}^{t}\left( b(s,X_{s}^{\varepsilon},Y_{s}^{\varepsilon})-\bar{b}(s,\bar{X}^{1}_{s})+ \frac{1}{\gamma_{\varepsilon}}H(s,X_{s}^{\varepsilon},Y_{s}^{\varepsilon})\right)ds,\nonumber
	\end{equation}
	then  from Theorem \ref{T52}, \eqref{5.59} in Theorem \ref{T53}, we know that 
	\begin{equation}\label{7.6}
		\begin{split}
			\mathbb{E}\left( \sup_{t\in [0,T]}|X_{t}^{\varepsilon}-\bar{X}^{1}_{t}|^{p}\right)	&\leq\mathbb{E}\left( \sup_{t\in [0,T]}\left| \int_{0}^{t}\left( b(s,X_{s}^{\varepsilon},Y_{s}^{\varepsilon})-\bar{b}(s,\bar{X}^{1}_{s})+ \frac{1}{\gamma_{\varepsilon}}H(s,X_{s}^{\varepsilon},Y_{s}^{\varepsilon})\right) ds\right| ^{p}\right)\\
			&\leq  C_{T,p}\left(\left(  \frac{\eta_{\varepsilon}}{\gamma_{\varepsilon}\beta_{\varepsilon}}\right) ^{p}+\left(  \frac{\eta_{\varepsilon}^{1-\frac{1-(1\wedge v)}{\alpha_{2}}}}{\gamma_{\varepsilon}^{2}}\right) ^{p}+\left( \frac{\eta_{\varepsilon}^{\left[ \left( \frac{v}{\alpha_{2}}\right) \wedge \left( 1-\frac{1\vee(\alpha_{1}-v)}{\alpha_{2}}\right) \right] }}{\gamma_{\varepsilon} } \right) ^{p}\right).
		\end{split}\nonumber
	\end{equation}

		Consider the following equation in \eqref{2.26},
	\begin{equation}\label{7.9}
		\mathcal{L}_{2}(x,y)u(t,x,y)+ H(t,x,y)=0,
	\end{equation}
	then we recall the definitions in \eqref{2.24},
	\begin{equation}
		\begin{split}
			\bar{c}(t,x)&=\int_{\mathbb{R}^{ d_{2}}}c(x,y)\nabla_{y}u(t,x,y)\mu^{x}(dy),
		\end{split}	\nonumber
	\end{equation}
	here $u(t,x,y)$ is the solution of \eqref{7.9}.

	For Regime 2,  we have 
	\begin{equation}\label{7.8}
		d\bar{X}^{2}_{t}=(\bar{b}(t,\bar{X}^{2}_{t})+\bar{c}(t,\bar{X}^{2}_{t}))dt+dL_{t}^{1},
	\end{equation}
from Theorem \ref{T52} and \eqref{5.60} in Theorem \ref{T53},  $\eta_{\varepsilon}=\gamma_{\varepsilon}\beta_{\varepsilon}$, we conclude that 
	\begin{equation}\label{7.10}
		\begin{split}
			\mathbb{E}\left( \sup_{t\in [0,T]}|X_{t}^{\varepsilon}-\bar{X}^{2}_{t}|^{p}\right)&=\mathbb{E}\left( \sup_{t\in [0,T]}\left| \int_{0}^{t}\left( b(s,X_{s}^{\varepsilon},Y_{s}^{\varepsilon})-\bar{b}(s,\bar{X}^{2}_{s})+ \frac{1}{\gamma_{\varepsilon}}H(s,X_{s}^{\varepsilon},Y_{s}^{\varepsilon})-\bar{c}(s,\bar{X}^{2}_{s})\right) ds\right| ^{p}\right)\\
			&\leq  C_{T,p}\left(                                                                                                                                                                                                                                                                           \left(  \frac{\eta_{\varepsilon}^{1-\frac{1-(1\wedge v)}{\alpha_{2}}}}{\gamma_{\varepsilon}^{2}}\right) ^{p}+ \left(  \frac{\eta_{\varepsilon}^{\left[ \left( \frac{v}{\alpha_{2}}\right) \wedge \left( 1-\frac{1\vee(\alpha_{1}-v)}{\alpha_{2}}\right) \right] }}{\gamma_{\varepsilon} } \right) ^{p}+\gamma_{\varepsilon}^{p}\right),
		\end{split}\nonumber
	\end{equation}
proof is complete.
\end{proof}

	\begin{corollary}\label{c71}
We observe that when $v\geq[(\alpha_{1}-1)\vee (\alpha_{2}-1)]$, the following simplifications hold:
$$\dfrac{\eta_{\varepsilon}^{\left[ \left( \frac{v}{\alpha_{2}}\right) \wedge \left( 1-\frac{1\vee(\alpha_{1}-v)}{\alpha_{2}}\right) \right] }}{\gamma_{\varepsilon} } =\dfrac{\eta_{\varepsilon}^{1-\frac{1}{\alpha_{2}}}}{\gamma_{\varepsilon} },$$
 obviously $\frac{\eta_{\varepsilon}^{1-\frac{1}{\alpha_{2}}}}{\gamma_{\varepsilon} } $ corresponds to optimal strong convergence order $1-\frac{1}{\alpha}$ 
 demostrated in \cite{SXX}. From the structure of \eqref{1.1}, we can deduce that imposing sufficient H\"{o}lder regularity conditions with respect to $t$ and $x$ on time-dependent drift $H(t,x,y)$ of slow process $X_{t}^{\varepsilon}$ leads to optimal strong convergence rates.

Meanwhile, it is necessary to emphasize that when $v\geq1$ the regime classification in \eqref{1.1-1},
$$ \frac{\eta_{\varepsilon}^{1-\frac{1-(1\wedge v)}{\alpha_{2}}}}{\gamma_{\varepsilon}^{2}}= \frac{\eta_{\varepsilon}}{\gamma_{\varepsilon}^{2}},$$ 
the term $ \frac{\eta_{\varepsilon}}{\gamma_{\varepsilon}^{2}}$ intrinsically separates distinct dynamical behaviors, while maintaining consistency with the multiscale stochastic framework first developed in \cite{EY2,EY3} and  more precise classifications in \cite{RX}.
	\end{corollary}

\subsection{Proof of \textbf{Theorem \ref{WCRD}}}
\begin{proof}
Analously,  in Regime 1, we have 
\begin{equation}\label{7.15}
	d\bar{X}^{1}_{t}=\bar{b}(t,\bar{X}^{1}_{t})dt+dL_{t}^{1},
\end{equation}
thus by regularity estimates in   Theorem \ref{T62}, and \eqref{6.37} in Theorem \ref{T63}, for $\phi(x)\in C^{2+\gamma}_{b}(\mathbb{R}^{d_{1}})$ in \eqref{6.10}, we obtain
\begin{equation}\label{7.18}
	\begin{split}
\sup_{t\in [0,T]}|\mathbb{E}\phi(X_{t}^{\varepsilon})-\mathbb{E}\phi(\bar{X}^{1}_{t})|&\leq\sup_{t\in [0,T]}\mathbb{E} \left| \int_{0}^{t}-\mathcal{\bar{L}}\hat{u}_{t}(s,X_{s}^{\varepsilon})+\mathcal{L}_{1}\hat{u}_{t}(s,X_{s}^{\varepsilon}) +(\frac{1}{\gamma_{\varepsilon}} \mathcal{L}_{3}\hat{u}_{t}(s,X_{s}^{\varepsilon}),\nabla_{x}\hat{u}_{t}(s,x))ds \right| \\
&\leq C_{T,x,y}\cdot\left( \frac{ \eta_{\varepsilon}^{\left[ \frac{v}{\alpha_{2}}\wedge \left( 1-\frac{\alpha_{1}-v}{\alpha_{2}} \right)  \right]  }}{\gamma_{\varepsilon}}+\frac{ \eta_{\varepsilon}^{1-\frac{1-(1\wedge v)}{\alpha_{2}}}}{\gamma^{2}_{\varepsilon}}+\frac{\eta_{\varepsilon}}{\gamma_{\varepsilon}\beta_{\varepsilon}}\right).
	\end{split}
\end{equation}

As for Regime 2, consider the following equation
\begin{equation}\label{7.20-1}
	\mathcal{L}_{2}(x,y)\Phi(t,x,y)+H(t,x,y)=0,
\end{equation}
then we have the definitions,
\begin{equation}
	\begin{split}
		\bar{c}(t,x)&=\int_{\mathbb{R}^{ d_{2}}}c(x,y)\nabla_{y}\Phi(t,x,y)\mu^{x}(dy),\\
		\bar{H}(t,x)&=\int_{\mathbb{R}^{ d_{2}}}H(t,x,y)\nabla_{x}\Phi(t,x,y)\mu^{x}(dy),
	\end{split}	\nonumber
\end{equation}
here $\Phi(t,x,y)$ is the solution of \eqref{7.20-1}, by Theorem \ref{T62} and Theorem \ref{T63},
\begin{equation}\label{7.20}
	\begin{split}
		&\sup_{t\in [0,T]}|\mathbb{E}\phi(X_{t}^{\varepsilon})-\mathbb{E}\phi(\bar{X}^{2}_{t})|\\
		&\leq\sup_{t\in [0,T]}\mathbb{E} \left|\int_{0}^{t}-\mathcal{\bar{L}}\hat{u}_{t}(s,X_{s}^{\varepsilon})+\mathcal{L}_{1}\hat{u}_{t}(s,X_{s}^{\varepsilon}) +(\frac{1}{\gamma_{\varepsilon}} \mathcal{L}_{3}\hat{u}_{t}(s,X_{s}^{\varepsilon})-\bar{c}(s,\bar{X}^{2}_{s}),\nabla_{x}\hat{u}_{t}(s,x))ds \right| \\
		&\leq C_{T,x,y}\left(\frac{ \eta_{\varepsilon}^{\left[ \frac{v}{\alpha_{2}}\wedge \left( 1-\frac{\alpha_{1}-v}{\alpha_{2}} \right)  \right]  }}{\gamma_{\varepsilon}}+\frac{ \eta_{\varepsilon}^{1-\frac{1-(1\wedge v)}{\alpha_{2}}}}{\gamma^{2}_{\varepsilon}}+\gamma_{\varepsilon}\right),
	\end{split}\nonumber
\end{equation}
here
\begin{equation}\label{7.21}
	d\bar{X}^{2}_{t}=(\bar{b}(t,\bar{X}^{2}_{t})+\bar{c}(t,\bar{X}^{2}_{t}))dt+dL_{t}^{1}.
\end{equation}

By this way, for Regime 3 we have 
\begin{equation}\label{7.23}
	d\bar{X}^{3}_{t}=(\bar{b}(t,\bar{X}^{3}_{t})+\bar{H}(t,\bar{X}^{3}_{t}))dt+dL_{t}^{1},
\end{equation}
consequently,
\begin{equation}\label{7.24}
	\begin{split}
		&\sup_{t\in [0,T]}|\mathbb{E}\phi(X_{t}^{\varepsilon})-\mathbb{E}\phi(\bar{X}^{3}_{t})|\\
		&\leq\sup_{t\in [0,T]}\mathbb{E} \left|\int_{0}^{t}-\mathcal{\bar{L}}\hat{u}_{t}(s,X_{s}^{\varepsilon})+\mathcal{L}_{1}\hat{u}_{t}(s,X_{s}^{\varepsilon}) +(\frac{1}{\gamma_{\varepsilon}} \mathcal{L}_{3}\hat{u}_{t}(s,X_{s}^{\varepsilon})-\bar{H}(s,\bar{X}^{3}_{s}),\nabla_{x}\hat{u}_{t}(s,x))ds \right| \\
		&\leq C_{T,x,y}\left( \gamma_{\varepsilon}^{\frac{2v}{\alpha_{2}}-\left[ 1\vee \left( \frac{2\alpha_{1}}{\alpha_{2}}-1\right)   \right] }+\frac{\gamma_{\varepsilon}}{\beta_{\varepsilon}}\right) .
	\end{split}\nonumber
\end{equation}

Hence  for Regime 4,
\begin{equation}\label{7.25}
	d\bar{X}^{4}_{t}=(\bar{b}(t,\bar{X}^{4}_{t})+\bar{c}(t,\bar{X}^{4}_{t})+\bar{H}(t,\bar{X}^{4}_{t}))dt+dL_{t}^{1},
\end{equation}
and 
\begin{align*}
		&\sup_{t\in [0,T]}|\mathbb{E}\phi(X_{t}^{\varepsilon})-\mathbb{E}\phi(\bar{X}^{4}_{t})|\\
		&\leq\sup_{t\in [0,T]}\mathbb{E} \left|\int_{0}^{t}-\mathcal{\bar{L}}\hat{u}_{t}(s,X_{s}^{\varepsilon})+\mathcal{L}_{1}\hat{u}_{t}(s,X_{s}^{\varepsilon}) +(\frac{1}{\gamma_{\varepsilon}} \mathcal{L}_{3}\hat{u}_{t}(s,X_{s}^{\varepsilon})-\bar{c}(s,\bar{X}^{4}_{s})-\bar{H}(s,\bar{X}^{4}_{s}),\nabla_{x}\hat{u}_{t}(s,x)) ds\right| \\
		&\leq C_{T,x,y}\cdot \gamma_{\varepsilon}^{\frac{2v}{\alpha_{2}}-\left[ 1\vee \left( \frac{2\alpha_{1}}{\alpha_{2}}-1\right)   \right] },
\end{align*}
proof is complete.
\end{proof}

\begin{corollary}\label{c72}
The parameter relationships become particularly transparent when taking $v=\alpha_{1}=\alpha_{2}$, we have
$$ \frac{\eta_{\varepsilon}^{1-\frac{1-(1\wedge v)}{\alpha_{2}}}}{\gamma_{\varepsilon}^{2}}= \frac{\eta_{\varepsilon}}{\gamma_{\varepsilon}^{2}},\quad 	\frac{ \eta_{\varepsilon}^{\left[ \frac{v}{\alpha_{2}}\wedge \left( 1-\frac{\alpha_{1}-v}{\alpha_{2}} \right)\right]  }}{\gamma_{\varepsilon}}=\frac{\eta_{\varepsilon}}{\gamma_{\varepsilon}},\quad \gamma_{\varepsilon}^{\frac{2v}{\alpha_{2}}-\left[ 1\vee \left( \frac{2\alpha_{1}}{\alpha_{2}}-1\right)   \right] }=\gamma_{\varepsilon},$$
the first equality is about regime classification, the second equality in our analysis corresponds to Regime $1$ and Regime $2$, whereas the third equality is associated with Regime $3$ and Regime $4$. 
From the structure of \eqref{1.1}, analogous to the analysis in  Corollary $\ref{c71}$, we observe that $\dfrac{\eta_{\varepsilon}}{\gamma_{\varepsilon}}$ and $\gamma_{\varepsilon}$ align with the weak convergence order $1$ 
established in \cite{SXX}.
\end{corollary}

\section{Acknowledgement}
The author would like to express his sincere gratitude to Professor Xin Chen at Shanghai Jiao Tong University  for his invaluable guidance as the doctoral supervisor, particularly for providing insightful suggestions and helpful discussions throughout this research endeavor. Particular acknowledgement is extended to the pioneering scholars, whose seminal researches in \cite{LRSX,EY1,EY2,EY3,RX,SXX} have profoundly inspired the framework of this study, especially Professor Long-jie Xie in Jiangsu Normal University.

\end{document}